\documentclass[11pt,a4paper]{article}

\usepackage[ngerman,spanish,english]{babel}
\usepackage[utf8]{inputenc}

\usepackage[DIV=12]{typearea}

\usepackage{stmaryrd}

\usepackage[affil-it]{authblk}

\usepackage{xparse}
\usepackage{amsmath, amsfonts, amssymb, amsthm}
\usepackage{mathtools}
\usepackage{enumitem}

\usepackage{microtype}

\usepackage[noadjust]{cite}

\usepackage[capitalise]{cleveref}

\crefname{equation}{}{}

\newcommand{\R}{\mathbb{R}}

\newcommand{\N}{\mathbb{N}}
\newcommand{\C}{\mathbb{C}}
\newcommand{\ee}{\mathrm{e}}

\DeclareDocumentCommand\dd{ o g d() }{
	\IfNoValueTF{#2}{
		\IfNoValueTF{#3}
			{\mathrm{d}\IfNoValueTF{#1}{}{^{#1}}}
			{\mathinner{\mathrm{d}\IfNoValueTF{#1}{}{^{#1}}\argopen(#3\argclose)}}
		}
		{\mathinner{\mathrm{d}\IfNoValueTF{#1}{}{^{#1}}#2} \IfNoValueTF{#3}{}{(#3)}}
	}

\newcommand{\dx}{\dd{x}}
\newcommand{\dy}{\dd{y}}
\newcommand{\dz}{\dd{z}}
\newcommand{\deta}{\dd{\eta}}
\newcommand{\dxi}{\dd{\xi}}
\newcommand{\del}{\partial}
\newcommand{\eps}{\varepsilon}

\newcommand{\mom}{M}

\newcommand{\T}{\mathcal{T}} 
\newcommand{\M}{\mathcal{M}} 
\newcommand{\LL}{\mathcal{L}} 
\newcommand{\A}{A}
\newcommand{\id}{\text{id}}
\newcommand{\B}{B}

\newcommand{\X}[3][]{X^{#2,#3}_{#1}}

\newcommand{\prof}{u} 
\newcommand{\pr}[1][\eps]{\prof^{(#1)}} 
\newcommand{\dpr}[2][\eps]{\pr_{#2}}
\newcommand{\Prp}[2][\eps]{U^{(#1)}_{#2}}
\newcommand{\hs}[1]{\mathfrak{h}_{#1}} 
\newcommand{\ode}{h} 
\newcommand{\rhs}{g} 
\newcommand{\rode}{v} 

\newcommand{\K}{P}
\newcommand{\J}{Q}

\newcommand{\weight}[2]{\omega_{#1,#2}} 

\newcommand{\vcc}{\vcentcolon}

\DeclarePairedDelimiter\abs{\lvert}{\rvert}
\DeclarePairedDelimiter\norm{\Vert}{\rVert}

\theoremstyle{plain}
\newtheorem{theorem}{Theorem}[section]
\newtheorem{lemma}[theorem]{Lemma}
\newtheorem{proposition}[theorem]{Proposition}

\theoremstyle{definition}
\newtheorem{definition}[theorem]{Definition}

\theoremstyle{remark}
\newtheorem{remark}[theorem]{Remark}

\numberwithin{equation}{section}

\AtBeginDocument{
   \def\MR#1{}
}

\title{Stability and uniqueness of self-similar profiles in $L^1$ spaces for perturbations of the constant kernel in Smoluchowski's coagulation equation }

 \author{Sebastian Throm \thanks{\texttt{throm@correo.ugr.es}}}
  
 \affil{\em Universidad de Granada, Departamento de Matem{\'a}tica Aplicada, \em Avenida de Fuentenueva S/N, 18071 Granada, Spain}

\date{}

\begin{document}
 \maketitle
 
 \begin{abstract}
  In this work, we consider self-similar profiles for Smoluchowski's coagulation equation for kernels which are possibly unbounded perturbations of the constant one. For this model, we show that the self-similar solutions for the perturbed kernel are close in weighted $L^1$ spaces to the profile of the unperturbed equation, i.e.\@ the profiles are stable with respect to the perturbation. Additionally, we revisit the problem of uniqueness for these coagulation kernels. In fact, we will improve a corresponding result (see \cite{NTV15,NTV16}) by relaxing the conditions on the perturbation significantly while at the same time the proof can also be notably shortened.
 \end{abstract}

 \section{Introduction}
 
 This article is concerned with the study of self-similar profiles for Smoluchowski's coagulation equation which reads
 \begin{equation}\label{eq:Smol}
  \del_{t}\phi(\xi,t)=\frac{1}{2}\int_{0}^{\xi}K(\xi-\eta,\eta)\phi(\xi-\eta)\phi(\eta)\deta-\phi(\xi)\int_{0}^{\infty}K(\xi,\eta)\phi(\eta)\deta\qquad \phi(\cdot)=\phi(\cdot,t).
 \end{equation}
 This equation arises as a mean-field model for systems of aggregating particles where $\phi(\xi,t)$ corresponds to the density of clusters of size/mass $\xi\in(0,\infty)$ at time $t>0$. The two integrals on the right-hand side account for the gain and loss of particles of size $\xi$ due to the coagulation process. In fact, two mergers of sizes $\xi-\eta$ and $\eta$ with $\eta<\xi$ form a cluster of mass $\xi$ and the rate at which such collisions take place is described by the integral kernel $K$. The factor $1/2$ is due to the symmetry of the coagulation process. In the same manner, the second integral in~\eqref{eq:Smol} takes into account, that particles of size $\xi$ will be removed from the system once they merge with any cluster of size $\eta>0$ to form a larger one.  
  
 A fundamental property of~\eqref{eq:Smol} is the (formal) conservation of total mass which corresponds to the first moment $\mom_{1}[\phi](t)\vcc=\int_{0}^{\infty}\xi\phi(\xi,t)\dxi$ of $\phi$. More precisely, if one multiplies~\eqref{eq:Smol} by $\xi$, integrates over $(0,\infty)$ and formally interchanges the order of integration one gets $\del_{t} \mom_{1}[\phi](t)=0$ i.e.\@ $\mom_{1}[\phi]$ is constant in time. However, as already mentioned, this argument is in general not correct and it has been proven in~\cite{EMP02,ELM03} that for kernels which grow faster than linearly at infinity $\mom_{1}[\phi](t)$ is in fact decaying as $t\to\infty$. The latter property is also known as \emph{gelation} and typically interpreted as a phase transition. Yet, in this work, we will only consider kernels $K$ which preserve the total mass.
 
 More precisely, we are interested in \emph{self-similar profiles} for~\eqref{eq:Smol}, i.e.\@ solutions of the special form $\phi(\xi,t)=t^{-2}\prof(\xi/t)$. The reason for this is that for homogeneous kernels $K$, based on formal considerations, it is conjectured (see e.g.\@ \cite{LaM04}) that such profiles describe the long-time behaviour of solutions to~\eqref{eq:Smol} in the sense that 
 \begin{equation}\label{eq:scaling:hyp}
  t^2\phi(t\xi,t)\longrightarrow \prof(\xi)\qquad \text{as }t\to \infty.
 \end{equation}
 Though up to now this question is still unsolved for kernels which arise typically in applications, such as Smoluchowski's kernel
 \begin{equation}\label{eq:kernel:Smol}
  K(\xi,\eta)=(\xi^{1/3}+\eta^{1/3})(\xi^{-1/3}+\eta^{-1/3}),
 \end{equation}
there are two prominent mass-conserving models, the \emph{solvable kernels}, which are well-understood. In fact, for $K\equiv 2$ and $K(\xi,\eta)=\xi+\eta$ explicit solution formulas are available at least in terms of the Laplace transform. Due to this, the conjecture~\eqref{eq:scaling:hyp}, also known as \emph{scaling hypothesis}, could be verified in these two cases (\cite{MeP04}). Even more, in~\cite{MeP04} the authors showed that besides the well-known fast-decaying profiles, there is a whole family of fat-tailed self-similar profiles with algebraic behaviour at infinity. Since the profiles for the two mass-conserving solvable kernels can be computed explicitly, one in particular obtains that they are unique upon a suitable normalisation.  

For kernels different from the solvable ones, such as~\eqref{eq:kernel:Smol} and many other examples from applications (see e.g.\@ \cite{Ald99,LaM04,Dra72}), the picture is much less complete. On the one hand, the well-posedness of~\eqref{eq:Smol} could be verified for large classes of kernels (i.a.\@ \cite{FoL06a,EsM06,EMR05}). Similarly, the existence of self-similar solutions and their properties are quite well understood (i.a.\@ \cite{EMR05,FoL05,FoL06,NiV14,NiV11,CaM11}). However, the actual question, namely if~\eqref{eq:scaling:hyp} holds true or not, is still unsolved. Even worse, for most kernels also uniqueness of the profiles could not yet be established. However, for the latter problem there exist at least some recent results which we will briefly summarise. In \cite{LNV18,Lau18} uniqueness of self-similar solutions has been proven for the two families of kernels $K(\xi,\eta)=\xi^{\lambda+1}\delta(\xi-\eta)$ (the \emph{diagonal kernel}) with $\lambda<1$ and $K(\xi,\eta)=(\xi\eta)^{-\lambda/2}$ with $\lambda>0$. Here $\delta$ denotes the Dirac distribution and $\lambda$ is the homogeneity of the kernel. Both proofs heavily rely on the specific structure of the considered kernel. In fact, for the diagonal kernel, the equation for self-similar profiles reduces to a non-local ODE. On the other hand, the proof in~\cite{Lau18} exploits that for $K(\xi,\eta)=(\xi\eta)^{-\lambda/2}$ the moment $\mom_{-\lambda/2}$ of self-similar profiles is already fixed by prescribing the total mass.

Moreover, in~\cite{NTV15,NTV16,Thr17a} the uniqueness problem has been attacked by a perturbative approach. More specifically, in these works the kernel is assumed to satisfy
\begin{equation*}
 \begin{gathered}
   0\leq K(\xi,\eta)-2\leq \eps \bigl(\xi^{\alpha}\eta^{-\alpha}+\xi^{-\alpha}\eta^{\alpha}\bigr) \qquad \text{with }\alpha\in[0,1/2) \text{ for all }\xi,\eta>0\\
   \text{and that }K(\cdot,1) \text{ admits an analytic extension to }\C\setminus (-\infty,0]
 \end{gathered}
\end{equation*}
along with further technical estimates on the latter. Working on the level of the Laplace transform, it has then been shown that self-similar profiles for such kernels are perturbations of the explicit profile for $K=2$ and based on this, a contraction estimate could be obtained providing uniqueness. However, the assumed analyticity on $K$ appears to be very restrictive since even for an extremely small non-analytic perturbation the proof breaks down. In fact, this regularity was necessary because of working with the Laplace-transformed equations which required to express all functionals in terms of their Laplace transform. For instance, also the perturbation $K(\xi-\eta)-2$ had to be expressed as Laplace integral. 

In this work, we will revisit this model of perturbations of the constant kernel but conversely to~\cite{NTV15,NTV16,Thr17a}, we will use an $L^1$ functional setup. The two main results which we will show are the following. First, we provide a stability statement for self-similar profiles in the perturbative regime, i.e.\@ that for sufficiently small $\eps>0$ all self-similar profiles are close to the one for $\eps=0$ in suitably weighted $L^1$ spaces (Theorem~\ref{Thm:closeness:profiles:NEW}).

As a second main result, we then provide another proof of uniqueness of self-similar solutions now in the weighted $L^1$ topology rather than for the Laplace transformed quantities (see Theorem~\ref{Thm:uniqueness}). This approach has several advantages. First of all, the whole proof can be significantly simplified. However, at the same time, we get a much stronger result, namely working in $L^1$ allows to get rid off the analyticity condition and most of the other technical assumptions. Furthermore, the results in~\cite{NTV15,NTV16,Thr17a} are restricted to exponents $\alpha\in[0,1/2)$ while in this work, we can now extend uniqueness to all $\alpha\in[0,1)$. A more detailed comparison of the current work to the results in \cite{NTV15,NTV16,Thr17a} can be found in Section~\ref{Sec:comparison:old:proof} below.

The remainder of the article is structured as follows. In Section~\ref{Sec:assumptions:results} we introduce the functional setup which we will use throughout this work, we collect the assumptions on the coagulation kernel and we present the two main statements which we will prove. In Section~\ref{Sec:previous:results} we summarise several results from~\cite{NiV14a,NTV15,NTV16} on which we will rely and we derive some immediate consequences. \Cref{Sec:pointwise:convergence,Sec:proof:stability} are then concerned with the proof of our first main result, Theorem~\ref{Thm:closeness:profiles:NEW}. In Section~\ref{Sec:preparation:uniqueness} we collect the key results which we need to prove the uniqueness of self-similar profiles in Section~\ref{Sec:proof:uniqueness}. The proofs of two of these preparing statements (\cref{Prop:continuity:L:inverse,Prop:bound:layer:est}) are relatively long and technical which is why they will be given separately in \cref{Sec:proof:inversion,Sec:proof:bl}. In the appendix we finally collect some additional material. More precisely, in \cref{Sec:properties:weight} we summarise some elementary properties of the weight and certain particular functions which we will frequently use. \Cref{Sec:inverse:derivation} contains a formal derivation of an explicit formula for the inverse of the linearised coagulation operator in self-similar variables. The latter is required to prove the boundedness of the inverse (see Proposition~\ref{Prop:continuity:L:inverse}). Finally, \cref{Sec:regularity:profiles} provides the regularity of self-similar profiles which is required to derive~\eqref{eq:bl:1} in the proof of Proposition~\ref{Prop:bound:layer:est}.

\section{Notation, assumptions and main results}\label{Sec:assumptions:results}

\subsection{Function spaces}

For $a,b\in\R$, we introduce the  weight function 
\begin{equation}\label{eq:def:weight}
 \weight{a}{b} \colon (0,\infty)\to (0,\infty) \qquad \text{such that}\qquad \weight{a}{b}(x)\vcc=\begin{cases}
          x^{a} &\text{if }x\leq 1\\
          x^{b} &\text{if }x\geq 1.
         \end{cases}
\end{equation}
We also note the following elementary properties:
\begin{align}
 \weight{a_1}{b_1}(x)\weight{a_2}{b_2}(x)&=\weight{a_1+a_2}{b_1+b_2}(x) \label{eq:weight:additivity}\\
 x^{\gamma}\weight{a}{b}(x)&=\weight{a+\gamma}{b+\gamma}(x) \label{eq:weight:shift}\\ 
 \weight{a_1}{b_1}(x)&\leq \weight{a_2}{b_2}(x)\qquad\text{if } a_2\leq a_1\quad \text{and}\quad b_1\leq b_2 \label{eq:weight:monotonicity}\\
 (1-\ee^{-x})\weight{a}{b}(x)&\leq \weight{a+1}{b}(x). \label{eq:weight:regularising}
\end{align}
With this notation, we can then define the sub-Banach space $\X{a}{b}$ of $L^1(0,\infty)$ for $a,b\in\R$ as
\begin{multline*}
 \X{a}{b}\vcc=\biggl\{g\in L^{1}(0,\infty)\;\bigg|\; \int_{0}^{\infty}\abs{g(x)}\weight{a}{b}(x)\dx<\infty\bigg\}\\*
 \text{with the norm }\norm{g}_{\X{a}{b}}\vcc=\int_{0}^{\infty}\abs{g(x)}\weight{a}{b}(x)\dx.
\end{multline*}
Moreover, we will need another sub-Banach space of $\X{a}{b}$ where the linearised coagulation operator $\LL$ (see~\eqref{eq:linearised:operator:abstract}) is injective. In fact, for $a\leq 1$ and $b\geq 1$ we define
\begin{equation*}
 \X[0]{a}{b}\vcc=\biggl\{g\in \X{a}{b}\;\bigg|\; \int_{0}^{\infty}xg(x)\dx=0\bigg\} \quad \text{with norm }\norm{\cdot}_{\X{a}{b}},
\end{equation*}
the subspace of $\X{a}{b}$ with vanishing first moment.
As a direct consequence of~\eqref{eq:weight:monotonicity} and Lebesgue's theorem we obtain the continuous embeddings
\begin{equation}\label{eq:spaces:embedding}
 \begin{gathered}
  \X{a_2}{b_2}\subseteq \X{a_1}{b_1}\qquad \text{and}\qquad  \X[0]{a_2}{b_2}\subseteq \X[0]{a_1}{b_1}\qquad \text{if } a_2\leq a_1 \text{ and }b_1\leq b_2\\
  \text{together with the estimate } \norm{g}_{\X{a_1}{b_1}}\lesssim \norm{g}_{\X{a_2}{b_2}} \quad \text{for }g\in\X{a_2}{b_2}.
 \end{gathered}
\end{equation}
For the latter embedding, we have to assume of course that $a_k\leq 1$ and $b_k\geq 1$ for $k=1,2$.

\begin{remark}
 We note that throughout this article, the notation $a\lesssim b$ means that the quantity $a$ can be estimated up to a constant by $b$, i.e.\@ there exists $C>0$ such that $a\leq Cb$.
\end{remark}

\subsection{Assumptions on the kernel}

To make our statement precise, let us specify the assumptions on the coagulation kernel $K_{\eps}$. We assume that $K_{\eps}$ is continuous, symmetric and homogeneous of degree zero, i.e.\@ 
\begin{equation}\label{eq:Ass:K1}
 K_{\eps}\in C(\R_{>0}\times \R_{>0}), \quad K_{\eps}(x,y)=K_{\eps}(y,x)\quad \text{and}\quad K_{\eps}(\lambda x, \lambda y)=K_{\eps}(x,y)\quad \text{for all }\lambda, x, y>0.
\end{equation}
Moreover, let $K_{\eps}$ satisfy
\begin{equation}\label{eq:Ass:K2}
 K_{\eps}(x,y)=2+\eps W(x,y) \qquad \text{with}\qquad W(x,y)\leq\Bigl(\frac{x}{y}\Bigr)^{\alpha}+\Bigl(\frac{y}{x}\Bigr)^{\alpha}\quad \text{for }\alpha\in(0,1).
\end{equation}
To simplify certain estimates we also note that~\eqref{eq:Ass:K2} in particular implies the bounds
\begin{equation}\label{eq:pert:est:weight}
 W(x,y)\lesssim \weight{-\alpha}{\alpha}(x)\weight{-\alpha}{\alpha}(y)\qquad \text{and}\qquad K_{\eps}(x,y)\lesssim \weight{-\alpha}{\alpha}(x)\weight{-\alpha}{\alpha}(y).
\end{equation}
For the latter estimate, as well as for the remainder of this article, we use implicitly that $\eps$ is bounded from above which is not really a restriction, since we can only expect that our results are true for sufficiently small $\eps$.

\begin{remark}
 We note that each bounded perturbation $W\lesssim 1$ in particular satisfies~\eqref{eq:Ass:K2} for any $\alpha\in(0,1)$.
\end{remark}

For the proof of uniqueness in the case $\alpha\geq 1/2$, we have to make an additional assumption on the perturbation $W$. In fact we need the lower bound
 \begin{equation}\label{eq:W:lower:bound}
  W(x,y)\geq c_{*}\biggl(\Bigl(\frac{x}{y}\Bigr)^{\alpha}+\Bigl(\frac{y}{x}\Bigr)^{\alpha}\biggr)
 \end{equation}
 with a constant $c_{*}>0$. This assumption leads to an exponential decay of the self-similar profiles close to zero which we have to exploit for some estimates if $\alpha\geq 1/2$.

\subsection{Notion of self-similar profiles}

The notion of self-similar solutions which we will use throughout this work follows that one in \cite{NTV16} and as outlined there, plugging the ansatz $\phi(\xi,t)=t^{-2}\prof(\xi/t)$ into~\eqref{eq:Smol} leads, up to an integration, to the equation
\begin{equation}\label{eq:selfsim}
  x^{2}\prof(x)=\int_{0}^{x}\int_{x-y}^{\infty}yK_{\eps}(y,z)\prof(y)\prof(z)\dz\dy.
\end{equation}

\begin{definition}\label{Def:profile}
 For $K_{\eps}$ satisfying \cref{eq:Ass:K1,eq:Ass:K2} with $\eps\geq 0$, a function $\pr\in L^{1}_{\text{loc}}(0,\infty)$ is denoted a self-similar solution/profile (of~\eqref{eq:Smol}) or equivalently a solution to~\eqref{eq:selfsim} provided that $\pr$ is almost everywhere non-negative, $\int_{0}^{\infty}x\pr(x)\dx<\infty$ and $\pr$ satisfies~\eqref{eq:selfsim} for almost every $x\in(0,\infty)$.
\end{definition}

We note that for each self-similar profile $\pr$ and each $c>0$ also the rescaled function $\prof^{(\eps),c}(x)\vcc=c\pr(cx)$ is a self-similar solution for~\eqref{eq:Smol}. Because of the mass-conserving property of~\eqref{eq:Smol}, the natural way to fix the parameter $c$ consists in normalising the profiles according to the total mass, i.e.\@ to prescribe the value of $\int_{0}^{\infty}x\pr(x)\dx$. This will also be done in this work and for simplicity, we normalise all self-similar profiles such that 
\begin{equation}\label{eq:normalisation}
 \int_{0}^{\infty}x\pr(x)\dx=1.
\end{equation}
The main reason for this choice is that for $\eps=0$ the unique self-similar profile is then given by (see e.g.\@ \cite{MeP04})
\begin{equation*}
 \pr[0](x)=\ee^{-x}.
\end{equation*}

\subsection{Main results}

The first main result of this work is a stability statement on self-similar profiles in the weighted $L^1$ spaces $\X{a}{b}$ for perturbations $K_{\eps}$ of the constant kernel. Precisely, we will show the following theorem.

\begin{theorem}\label{Thm:closeness:profiles:NEW}
 For $\eps>0$ let $K_{\eps}$ satisfy \cref{eq:Ass:K1,eq:Ass:K2} and let $a>-1$ and $b>0$ be given. For each $\delta>0$ there exists $\eps_{*}>0$ such that
 \begin{equation*}
  \norm{\pr-\ee^{-\cdot}}_{\X{a}{b}}\leq \delta \qquad \text{if }\eps\leq \eps_{*}
 \end{equation*}
 for each self-similar profile $\pr$ with total mass one. In particular, if $\dpr{1}$ and $\dpr{2}$ are two self-similar profiles we have $\norm{\dpr{1}-\dpr{2}}_{\X{a}{b}}\leq 2\delta \qquad \text{if }\eps\leq \eps_{\delta}$.
\end{theorem}

Our second main result is the following statement on uniqueness of self-similar profiles.

\begin{theorem}\label{Thm:uniqueness}
For $\eps>0$ let $K_{\eps}$ satisfy \cref{eq:Ass:K1,eq:Ass:K2}. If $\alpha\geq 1/2$, assume in addition that $W$ satisfies~\eqref{eq:W:lower:bound}. Then, if $\eps>0$ is sufficiently small there exists at most one self-similar profile which is normalised according to~\eqref{eq:normalisation}.
\end{theorem}

\subsection{Comparison to the uniqueness statement proved in \cite{NTV15,NTV16}}\label{Sec:comparison:old:proof}

Theorem~\ref{Thm:uniqueness} is an extension and improvement of the corresponding statement in \cite{NTV15,NTV16} in several respects. First of all, the proof presented here relies on estimates in the weighted $L^1$ spaces $\X{a}{b}$ which makes it much easier than the one in \cite{NTV15,NTV16}. In fact, although the abstract strategy, which can be interpreted as a non-linear version of the implicit function theorem, is still the same, the different choice of the topology simplifies the proof remarkably. More precisely, in \cite{NTV15,NTV16} uniqueness was shown by proving a contraction inequality for the Laplace transform of two self-similar profiles in a weighted $C^{2}$ norm. The advantage of this approach was that the stability of the profiles, i.e.\@ the analogue of Theorem~\ref{Thm:closeness:profiles:NEW} as well as the inversion of the linearised operator (the analogue of Proposition~\ref{Prop:continuity:L:inverse} below) could be obtained much easier. However, as a consequence, all functionals had to be expressed in terms of the Laplace transform which required first to write also the perturbation $W$ itself as Laplace transform of a suitable kernel. For this, extremely strong regularity assumptions were needed. In fact, in addition to \cref{eq:Ass:K1,eq:Ass:K2} it was required that $W(\cdot,1)$ is analytic in $\C\setminus (-\infty,0]$ and can be extended to a $C^{1,\gamma}$ function both on the closed upper and lower complex half-plane. These conditions were accompanied by several estimates on $W$ and its derivative (see (1.10)--(1.16) in~\cite{NTV15} for the precise assumptions). Similarly, proving an analogue to Proposition~\ref{Prop:bound:layer:est} required to represent a certain non-linear expression as a Laplace integral which made the corresponding argument extremely long and technical. Furthermore, the latter proof only worked if the exponent $\alpha$ in the perturbation was restricted to $\alpha<1/2$ while it remained unclear if this is only for technical problems or if uniqueness might really fail for $\alpha\geq 1/2$. 

As already mentioned before, we choose a different functional setup in this work, i.e.\@ we work with weighted $L^1$ spaces instead of the Laplace transform. This allows to get rid off most of the technical problems described before. Precisely, we can relax the assumptions on the perturbation $W$ by requiring only \cref{eq:Ass:K1,eq:Ass:K2} while analyticity is no longer needed. Furthermore, we can extend the uniqueness result also to the case $\alpha\in[1/2,1)$ which was not clear to be true before. Let us again emphasise here that even though Theorem~\ref{Thm:uniqueness} gives a much stronger result compared to that one in \cite{NTV15,NTV16}, the corresponding proof in summary is even simpler and much shorter. Of course, since we follow the same abstract strategy, some proofs are still similar to those in \cite{NTV15,NTV16}, however, let us finally summarise here the most important differences. First of all, the statement of the stability of self-similar profiles (Theorem~\ref{Thm:closeness:profiles:NEW}) is now much stronger since it shows closeness of the profiles in the strong $L^1$ topology instead of for the Laplace transform. As a consequence, also the proof gets more involved. Moreover, the inversion of the linearised operator (Proposition~\ref{Prop:continuity:L:inverse}) requires a different argument compared to \cite{NTV15,NTV16} and also becomes more technical. Conversely, as already explained before, the proof of Proposition~\ref{Prop:bound:layer:est}, though still relatively long and technical, is now much shorter and simpler than the one for the corresponding result in \cite{NTV15,NTV16} where it occupied more than half of the article.

\section{Previous results and easy consequences}\label{Sec:previous:results}

In this section we collect several results which have been obtained in~\cite{NiV14a,NTV15,NTV16} or which are easy consequences of such results and on which we will rely in this work.

The first such statement concerns the following lemma which provides uniform convergence of the (desingularised) Laplace transform of self-similar profiles and which is contained in~\cite[Lemma~2.8]{NiV14}.

\begin{lemma}\label{Lem:convergence:Laplace}
 Under the conditions of Theorem~\ref{Thm:closeness:profiles:NEW} we have
 \begin{equation*}
  \lim_{\eps\to 0}\sup_{p\in[0,\infty)}\abs*{\int_{0}^{\infty}(1-\ee^{-px})(\pr(x)-\ee^{-x})\dx}=0.
 \end{equation*}
\end{lemma}

The next result provides uniform exponential decay at infinity for self-similar profiles and can be found in~\cite[Lemma~2.5]{NiV14}.

\begin{lemma}\label{Lem:exp:decay}
 There exists a constant $a>0$ such that
 \begin{equation*}
  \pr(x)\lesssim \ee^{-ax}\quad \text{for all }x\geq 1.
 \end{equation*}
\end{lemma}

Moreover, we recall from~\cite[Lemma~2.4]{NiV14} a certain regularity of self-similar profiles close to zero in a weak form.

\begin{lemma}\label{Lem:regularity:zero}
 For any $\eta>0$ and sufficiently small $\eps>0$ we have
 \begin{equation*}
  \int_{\rho}^{2\rho}\pr(x)\dx\lesssim \rho^{1-\eta} \qquad \text{for all }\rho>0.
 \end{equation*}
\end{lemma}

\begin{remark}
 Although in~\cite{NiV14} this result is only formulated to hold if $\rho\in(0,\rho_{0}]$ for some $\rho_{0}>0$ one easily sees that, together with \cref{Lem:exp:decay,eq:normalisation} the slightly more general version stated above is also correct.
\end{remark}

Based on these results we will now show the following lemma, which provides uniform boundedness of certain moments of self-similar profiles and which we will frequently use throughout this work.

\begin{lemma}\label{Lem:moments}
 For each $a\in(-1,\infty)$ there exists $\eps_{*}>0$ such that we have for each $b\in\R$ that 
 \begin{equation*}
  \int_{0}^{\infty}\weight{a}{b}(x)\pr(x)\dx\leq C_{a,b}\qquad \text{for all }\eps\in[0,\eps_{*}]
 \end{equation*}
 and all self-similar profiles $\pr$ where $C_{a,b}>0$ is a constant independent of $\eps$ and $\pr$.
\end{lemma}

\begin{proof}
 The uniform boundedness of $\int_{1}^{\infty}x^{b}\pr(x)\dx$ is clear due to Lemma~\ref{Lem:exp:decay}. Additionally, choosing $\eps_{*}>0$ sufficiently small, we can also deduce that $\int_{0}^{1}x^{a}\pr(x)\dx$ is uniformly bounded by means of Lemma~\ref{Lem:regularity:zero} together with a dyadic argument (see e.g.\@ \cite{NiV14,NTV15a,Thr18}).
\end{proof}

Finally, we will prove the following lemma which gives weak convergence of the sequence $\pr$ and its mass density in the sense of measures.

\begin{proposition}\label{Prop:weak:convergence:NEW}
 Let $\eps_{k}\to0$ as $k\to\infty$ be given. Under the conditions of Theorem~\ref{Thm:closeness:profiles:NEW} we have $\pr[\eps_{k}]\rightharpoonup \ee^{-\cdot}$ as $k\to\infty$ weakly in the sense of measures, i.e.\@
 \begin{equation*}
  \lim_{k\to\infty}\int_{0}^{\infty}(\pr[\eps_{k}](x)-\ee^{-x})\varphi(x)\dx =0 \quad \text{for all }\varphi\in C_{b}(\R_{\geq0})
 \end{equation*}
 for each sequence of self-similar profiles $\pr[\eps_{k}]$. Moreover, $x\pr[\eps_{k}](x)\rightharpoonup x\ee^{-x}$ as $k\to\infty$ weakly in the sense of measures, i.e.\@ $\lim_{k\to\infty}\int_{0}^{\infty}x(\pr[\eps_{k}](x)-\ee^{-x})\varphi(x)\dx =0$ for all $\varphi\in C_{b}(\R_{\geq0})$.
\end{proposition}

\begin{proof}
  To prove the first claim, we first deduce from Lemma~\ref{Lem:convergence:Laplace} that 
 \begin{equation}\label{eq:weakd:convergence:1}
  \lim_{k\to\infty}\abs*{\int_{0}^{\infty}\ee^{-px}(\pr[\eps_{k}](x)-\ee^{-x})\dx}=0\quad \text{for all }p\in [0,\infty).
 \end{equation}
 In fact, for $\delta>0$ given we may fix $\eps_{*}>0$ such that
  \begin{equation*}
 \abs*{\int_{0}^{\infty}(1-\ee^{-px})(\pr[\eps_{k}](x)-\ee^{-x})\dx}<\delta \quad \text{for all } p\in[0,\infty)\text{ and all }\eps_{k}\in[0,\eps_{*}].
 \end{equation*}
 Passing to the limit $p\to\infty$ on the left-hand side which is possible due to Lebesgue's theorem, we have 
 \begin{equation*}
  \lim_{k\to\infty}\abs*{\int_{0}^{\infty}(\pr[\eps_{k}](x)-\ee^{-x})\dx}=0.
 \end{equation*}
 From this together with Lemma~\ref{Lem:convergence:Laplace} the claimed limit in~\eqref{eq:weakd:convergence:1} directly follows.
 
 Due to~\eqref{eq:weakd:convergence:1} we obtain from~\cite[Lemma~A.9]{SSV10} that 
 \begin{equation*}
  \lim_{k\to\infty}\int_{0}^{\infty}\bigl(\pr[\eps_{k}](x)-\ee^{-x}\bigr)\varphi(x)\dx =0\quad \text{for all }\varphi\in C_{b}(\R_{\geq 0}).
 \end{equation*}
 To prove the second claim of the lemma, we choose a partition of unity $1=\zeta_1+\zeta_2$ subordinate to the covering $(0,\infty)=(0,2R)\cup (R,\infty)$ with $R>1$ and $\zeta_1,\zeta_2\in C_{\text{b}}(\R_{\geq 0})$ satisfying $0\leq \zeta_1,\zeta_2\leq 1$. This allows to estimate
 \begin{multline*}
  \abs*{\int_{0}^{\infty}x(\pr[\eps_{k}](x)-\ee^{-x})\varphi(x)\dx}\\*
  \leq \abs*{\int_{0}^{\infty}(\pr[\eps_{k}](x)-\ee^{-x})x\varphi(x)\zeta_{1}(x)\dx}+\abs*{\int_{0}^{\infty}x(\pr[\eps_{k}](x)-\ee^{-x})\varphi(x)\zeta_{2}(x)\dx}.
 \end{multline*}
 Since $x\mapsto x\varphi(x)\zeta_{1}(x)$ is in $C_{b}(\R_{\geq0})$, the first integral on the right-hand side converges to zero as $\eps_{k}\to 0$. On the other hand, Lemma~\ref{Lem:exp:decay} together with the choice of $\zeta_2$ yields
 \begin{equation*}
  \abs*{\int_{0}^{\infty}x(\pr[\eps_{k}](x)-\ee^{-x})\varphi(x)\zeta_{2}(x)\dx}\lesssim \norm{\varphi}_{L^{\infty}}\int_{R}^{\infty}x\bigl(\ee^{-ax}+\ee^{-x}\bigr)\dx\lesssim \norm{\varphi}_{L^{\infty}} R\ee^{-\min\{a,1\}R}.
 \end{equation*}
 The right-hand side converges to zero as $R\to\infty$. Thus, choosing first $R$ large and then $\eps_{k}$ small, finishes the proof.
\end{proof}

\section{Pointwise convergence of self-similar profiles}\label{Sec:pointwise:convergence}

In this section, we will show that each sequence of self-similar profiles $\pr$ converges at least pointwise to the unique profile $\ee^{-x}$ as $\eps\to 0$. This result will be a rather straightforward consequence of the two preparing \cref{Lem:profiles:unfiorm:bound,Lem:primitive:uniform:convergence:NEW}. The first one gives a uniform upper bound for self-similar profiles.

\begin{lemma}\label{Lem:profiles:unfiorm:bound}
 For each $\nu\in(0,1-\alpha)$ there exists $\eps_{\nu}>0$ and $a>0$ such that
 \begin{equation*}
  \pr(x)\lesssim \weight{-\alpha-\nu}{0}(x)\ee^{-ax} \qquad \text{for all } x\in(0,\infty)\quad \text{and all }\eps\in[0,\eps_{\nu}]
 \end{equation*}
 where $\pr$ is any self-similar profile.
\end{lemma}

\begin{proof}
 For $x\geq 1$, the claim immediately follows from Lemma~\ref{Lem:exp:decay}. If $x\leq 1$, we recall from~\eqref{eq:pert:est:weight} that $K_{\eps}(y,z)\lesssim \weight{-\alpha}{\alpha}(y)\weight{-\alpha}{\alpha}(z)$. Thus, \eqref{eq:selfsim} together with Lemma~\ref{Lem:moments} yields
 \begin{multline*}
  \pr(x)=\frac{1}{x^2}\int_{0}^{x}\int_{x-y}^{\infty}yK_{\eps}(y,z)\pr(y)\pr(z)\dz\dy\\*
  \lesssim \frac{1}{x}\int_{0}^{x}\weight{-\alpha}{\alpha}(y)\pr(y)\int_{0}^{\infty}\weight{-\alpha}{\alpha}(z)\pr(z)\dz\dy\lesssim \frac{1}{x}\int_{0}^{x}\weight{-\alpha}{\alpha}(y)\pr(y)\dy.
 \end{multline*}
 To finish the proof, we use $x\leq 1$ and Lemma~\ref{Lem:moments} to estimate further
 \begin{equation*}
  \pr(x)\lesssim \frac{1}{x}\int_{0}^{x}y^{1-\alpha-\nu}y^{\nu-1}\pr(y)\dy\leq x^{-\alpha-\nu}\int_{0}^{x}y^{\nu-1}\pr(y)\dy\lesssim x^{-\alpha-\nu}.
 \end{equation*}
\end{proof}

\begin{remark}\label{Rem:Portmanteau}
For later use, we recall the following equivalence of the Portmanteau theorem (see e.g.\@ \cite[p.\@385]{Els11}). Let $\mathcal{X}$ be a metric space and $\mu_{n},\mu\in \M^{+}$ (the space of non-negative finite measures on $\mathcal{X}$). Then the following two statements are equivalent
\begin{enumerate}
 \item $\mu_{k}\rightharpoonup \mu$ as $k\to\infty$
 \item $\lim_{k\to\infty}\mu_{k}(B)=\mu(B)$ for each continuity set $B\subset \mathcal{X}$ of $\mu$.
\end{enumerate}
\end{remark}

The next lemma provides uniform convergence of certain primitives of self-similar profiles.

\begin{lemma}\label{Lem:primitive:uniform:convergence:NEW}
 Let $x>0$ be fixed. For each self-similar profile $\pr$ we define the (continuous) function $\Prp{x}\colon [0,x]\to [0,\infty)$ by
 \begin{equation*}
  \Prp{x}(y)\vcc=\int_{0}^{x-y}\pr(z)\dz.
 \end{equation*}
 Then, for each sequence $\eps_{k}\to 0$ as $k\to\infty$, we have that $\Prp[\eps_{k}]{x}$ converges uniformly on $[0,x]$ to the function $y\mapsto \Prp[0]{x}(y)\vcc=\int_{0}^{x-y}\ee^{-z}\dz=(1-\ee^{-(x-y)})$. In particular, if we extend $\Prp{x}$ to $[0,\infty)$ by setting $\Prp{x}(y)=0$ if $y>x$ we have that $\Prp[\eps_{k}]{x}\to \Prp[0]{x}$ uniformly on $[0,\infty)$. Moreover, on $[0,x]$ the function $y\mapsto \int_{x-y}^{\infty}\pr[\eps_{k}](z)\dz$ converges uniformly to $y\mapsto \int_{x-y}^{\infty}\ee^{-z}\dz=\ee^{-(x-y)}$.
\end{lemma}

\begin{proof}
 The last two assertions follow directly from the first one noting that $\Prp{x}(x)=0$ and $\int_{x-y}^{\infty}\pr(z)\dz=\int_{0}^{\infty}\pr(z)\dz-\Prp{x}(y)$ for all $\eps$ while $\int_{0}^{\infty}\pr[\eps_{k}](z)\dz\to \int_{0}^{\infty}\ee^{-z}\dz=1$ due to Proposition~\ref{Prop:weak:convergence:NEW}. Thus, it suffices to prove the first claim. For this, we note that \cref{Prop:weak:convergence:NEW,Rem:Portmanteau} immediately yield that 
 \begin{equation}\label{eq:pf:unif:conv:0}
  \Prp[\eps_{k}]{x}\longrightarrow \Prp[0]{x} \qquad \text{pointwise on }[0,x]\quad \text{as }k\to\infty.
 \end{equation}
 To prove the uniform convergence, we will show that for sufficiently small $\eps_{*}>0$ the set $\{\Prp[\eps_{k}]{x}\}_{\eps_{k}\leq \eps_{*}}$ is pre-compact in $C([0,x])$. To see this, we fix $\nu\in (0,1-\alpha)$ and $\eps_{*}>0$ according to Lemma~\ref{Lem:profiles:unfiorm:bound} such that Lemma~\ref{Lem:moments} immediately implies
 \begin{equation}\label{eq:pf:unif:conv:1}
  \sup_{\eps_{k}\in[0,\eps_{*}]}\sup_{y\in[0,x]}\abs{\Prp[\eps_{k}]{x}(y)}\leq \sup_{\eps_{k}\in[0,\eps_{*}]}\int_{0}^{\infty}\pr[\eps_{k}](z)\dz\lesssim 1.
 \end{equation}
 Moreover, for $y_1,y_2\in[0,x]$ with $y_1<y_2$ we have by means of Lemma~\ref{Lem:profiles:unfiorm:bound} and \cite[eq.\@ (A3-10)]{Alt16} that
 \begin{equation}\label{eq:pf:unif:conv:2}
  \begin{split}
   \sup_{\eps_{k}\in[0,\eps_{*}]}\abs{\Prp[\eps_{k}]{x}(y_1)-\Prp[\eps_{k}]{x}(y_2)}&=\int_{x-y_1}^{x-y_2}\pr[\eps_{k}](z)\dz\\
   &\lesssim \int_{x-y_1}^{x-y_2}\weight{-\alpha-\nu}{0}(z)\ee^{-az}\dz\longrightarrow 0 \qquad \text{as }\abs{y_1-y_2}\to 0.
  \end{split}
 \end{equation}
 Taking \cref{eq:pf:unif:conv:1,eq:pf:unif:conv:2} together, the Arzel\`{a}-Ascoli theorem (see e.g.\@ \cite[p.\@106]{Alt16}) shows the pre-compactness of $\{\Prp[\eps_{k}]{x}\}_{0\leq\eps_{k}\leq\eps_{*}}$. Thus, at least for a subsequence $\eps_{k_{\ell}}\to 0$ we have that $\Prp[\eps_{k_{\ell}}]{x}$ converges uniformly on $[0,x]$. However, due to~\eqref{eq:pf:unif:conv:0} the corresponding limit can be identified with $\Prp[0]{x}$. Since this argument holds true for each convergent subsequence, we conclude that in fact already $\Prp[\eps_{k}]{x}\to \Prp[0]{x}$ uniformly on $[0,x]$ as $k\to\infty$ which ends the proof.
\end{proof}

With these preparations, we can now show the pointwise convergence of self-similar profiles.

\begin{lemma}\label{Lem:pointwise:conv:mass:density:NEW}
 Let $\eps_{k}\to 0$ as $k\to\infty$ be given and let $\pr[\eps_{k}]$ be a corresponding sequence of self-similar profiles to~\eqref{eq:Smol}. Then we have $\pr[\eps_{k}](x)\to \ee^{-x}$ as $k\to\infty$ pointwise for all $x\in(0,\infty)$.
\end{lemma}

\begin{proof}
 The claimed convergence is obviously equivalent to $x^2\pr[\eps_{k}](x)\to x^2\ee^{-x}$ as $k\to\infty$ for $x\in(0,\infty)$. To see the latter, we assume $x>0$ and recall that $\pr[\eps_{k}]$ solves~\eqref{eq:selfsim} while $\ee^{-x}$ solves the same equation with $\eps=0$. Thus, we can rewrite
 \begin{multline*}
  x^2(\pr[\eps_{k}](x)-\ee^{-x})=\int_{0}^{x}\int_{x-y}^{\infty}y\bigl(K_{\eps_{k}}(y,z)-2\bigr)\pr[\eps_{k}](y)\pr[\eps_{k}](z)\dz\dy\\*
  +2\int_{0}^{x}y(\pr[\eps_{k}](y)-\ee^{-y})\int_{x-y}^{\infty}\pr[\eps_{k}](z)\dz\dy+2\int_{0}^{x}y\ee^{-y}\int_{x-y}^{\infty}(\pr[\eps_{k}](z)-\ee^{-z})\dz\dy.
 \end{multline*}
 To continue, we note that $y(K_{\eps_{k}}(y,z)-2)=\eps_{k} y W(y,z)\lesssim \eps_{k}\weight{1-\alpha}{1+\alpha}(y)\weight{-\alpha}{\alpha}(z)$ due to \cref{eq:pert:est:weight,eq:weight:shift,eq:Ass:K2} and recall the notation $\Prp[\eps_{k}]{x}$ as introduced in Lemma~\ref{Lem:primitive:uniform:convergence:NEW} which allows to estimate
 \begin{multline*}
  \abs*{x^2(\pr[\eps_{k}](x)-\ee^{-x})}\lesssim \eps_{k} \int_{0}^{\infty}\int_{0}^{\infty}\weight{1-\alpha}{1+\alpha}(y)\weight{-\alpha}{\alpha}(z)\pr[\eps_{k}](y)\pr[\eps_{k}](z)\dz\dy\\*
  +2\int_{0}^{\infty}\pr[\eps_{k}](z)\dz\abs*{\int_{0}^{x}y(\pr[\eps_{k}](y)-\ee^{-y})\dy}+2\abs*{\int_{0}^{\infty}y(\pr[\eps_{k}](y)-\ee^{-y})\Prp[\eps_{k}]{x}(y)\dy}\\*
  +2\abs*{\int_{0}^{x}y\ee^{-y}\biggl(\int_{x-y}^{\infty}\pr[\eps_{k}](z)-\ee^{-z}\dz\biggr)\dy}.
 \end{multline*}
 By means of Lemma~\ref{Lem:moments} we can estimate further to get
  \begin{multline*}
  \abs*{x^2(\pr[\eps_{k}](x)-\ee^{-x})}\lesssim \eps_{k} +\abs*{\int_{0}^{x}y(\pr[\eps_{k}](y)-\ee^{-y})\dy}+2\abs*{\int_{0}^{\infty}y(\pr[\eps_{k}](y)-\ee^{-y})\Prp[\eps_{k}]{x}(y)\dy}\\*
  +2\abs*{\int_{0}^{x}y\ee^{-y}\biggl(\int_{x-y}^{\infty}\pr[\eps_{k}](z)-\ee^{-z}\dz\biggr)\dy}.
 \end{multline*}
 Due to \cref{Rem:Portmanteau,Prop:weak:convergence:NEW,Lem:primitive:uniform:convergence:NEW} the right-hand side converges to zero as $\eps_{k}\to 0$. More precisely, the first integral converges to zero due to the weak convergence of $y\pr[\eps_{k}](y)$ to $y\ee^{-y}$ and the fact that $(0,x)$ is a continuity set of $y\ee^{-y}$ (\cref{Rem:Portmanteau,Prop:weak:convergence:NEW}). The second integral converges since $y\pr[\eps_{k}](y)\rightharpoonup y\ee^{-y}$ and $\Prp[\eps_{k}]{x}\to \Prp[0]{x}$ uniformly on $(0,\infty)$ as $\eps_{k}\to 0$ (\cref{Prop:weak:convergence:NEW,Lem:primitive:uniform:convergence:NEW}). For the last integral, we use the uniform convergence of $y\mapsto \int_{x-y}^{\infty}\pr[\eps_{k}](z)\dz$ to $y\mapsto \int_{x-y}^{\infty}\ee^{-z}\dz$ (Lemma~\ref{Lem:primitive:uniform:convergence:NEW}).
\end{proof}

\section{Stability of self-similar profiles -- Proof of Theorem~\ref{Thm:closeness:profiles:NEW}}\label{Sec:proof:stability}

As a preparing step, we show that for sufficiently small $\eps>0$ all self-similar profiles are close to $\ee^{-\cdot}$ in $L^1(0,\infty)$.

\begin{proposition}\label{Prop:L1:convergence:NEW}
 For each $\delta>0$ there exists $\eps_{*}>0$ such that each self-similar profile $\pr$ with $\eps\in(0,\eps_{*})$ satisfies $\norm{\pr-\ee^{-\cdot}}_{L^1(0,\infty)}<\delta$.
\end{proposition}

\begin{proof}
 We argue by contradiction and thus assume that the claim is not true. Then, we may find $\delta_{*}>0$ and a sequence $\eps_{k}\to 0$ as $k\to 0$ such that 
 \begin{equation}\label{eq:pf:stability:1}
   \norm{\pr[\eps_{k}]-\ee^{-\cdot}}_{L^{1}(0,\infty)}\geq \delta_{*} \qquad \text{for all }k\in\N.
 \end{equation}
However, by means of Lemma~\ref{Lem:pointwise:conv:mass:density:NEW} we have $\pr[\eps_{k}](x)\to \ee^{-x}$ for all $x\in(0,\infty)$ as $k\to\infty$. Furthermore, for given $\nu\in(0,1-\alpha)$ and $\eps_{k}$ sufficiently small (i.e.\@ $k\in\N$ sufficiently large) Lemma~\ref{Lem:profiles:unfiorm:bound} yields that there exists $a>0$ such that $\pr[\eps_{k}](x)\lesssim \weight{-\alpha-\nu}{0}(x)\ee^{-ax}$. Since the right-hand side is integrable on $(0,\infty)$, Lebesgue's dominated convergence theorem implies that $\pr[\eps_{k}]\to \ee^{-\cdot}$ in $L^1(0,\infty)$ which contradicts~\eqref{eq:pf:stability:1} and thus finishes the proof.
\end{proof}

Based on Proposition~\ref{Prop:L1:convergence:NEW} we can now give the proof of Theorem~\ref{Thm:closeness:profiles:NEW}.

\begin{proof}[Proof of Theorem~\ref{Thm:closeness:profiles:NEW}]
 The second part is a direct consequence of the first one due to $\dpr{1}-\dpr{2}=(\dpr{1}-\ee^{-\cdot})+(\ee^{-\cdot}-\dpr{2})$ and the triangle inequality.
 
 To prove the first statement, we take constants $r<1$ and $R>1$ to be fixed later, and split the integral defining $\norm{\cdot}_{\X{a}{b}}$ which yields
 \begin{multline*}
  \norm{\pr-\ee^{-\cdot}}_{\X{a}{b}}\\*
  =\int_{0}^{r}\weight{a}{b}(x)\abs{\pr(x)-\ee^{-x}}\dx+\int_{r}^{R}\weight{a}{b}(x)\abs{\pr(x)-\ee^{-x}}\dx+\int_{R}^{\infty}\weight{a}{b}(x)\abs{\pr(x)-\ee^{-x}}\dx\\*
  =\int_{0}^{r}x^{a}\abs{\pr(x)-\ee^{-x}}\dx+\int_{r}^{R}\weight{a}{b}(x)\abs{\pr(x)-\ee^{-x}}\dx+\int_{R}^{\infty}x^{b}\abs{\pr(x)-\ee^{-x}}\dx.
 \end{multline*}
 Using that $x^{a}=x^{\frac{a-1}{2}}x^{\frac{a+1}{2}}\leq r^{\frac{a+1}{2}} x^{\frac{a-1}{2}}$ for $x\in(0,r)$ and $\frac{a-1}{2}>-1$ as well as $x^{b}=x^{-b}x^{2b}\leq R^{-b}x^{2b}$ for $x>R$ the right-hand side of the previous equation can be further estimated to get
 \begin{multline*}
  \norm{\pr-\ee^{-\cdot}}_{\X{a}{b}}\leq r^{\frac{a+1}{2}}\int_{0}^{\infty}x^{\frac{a-1}{2}}\bigl(\pr(x)+\ee^{-x})\dx+\max_{x\in[r,R]}\weight{a}{b}(x)\int_{0}^{\infty}\abs{\pr(x)-\ee^{-x}}\dx\\*
  +R^{-b}\int_{0}^{\infty}x^{2b}\bigl(\pr(x)+\ee^{-x}\bigr)\dx.
 \end{multline*}
 Since $\frac{a-1}{2}>-1$ we can estimate the right-hand side further by means of Lemma~\ref{Lem:moments} which yields
 \begin{equation*}
  \norm{\pr-\ee^{-\cdot}}_{\X{a}{b}}\leq  C(r^{\frac{a+1}{2}}+R^{-b})+\max_{x\in[r,R]}\weight{a}{b}(x)\int_{0}^{\infty}\abs{\pr(x)-\ee^{-x}}\dx.
 \end{equation*}
 The claim thus follows if we first choose $r<1$ and $R>1$ sufficiently small and large respectively such that $C(r^{\frac{a+1}{2}}+R^{-b})\leq \delta/2$ and then $\eps>0$ small such that $\max_{x\in[r,R]}\weight{a}{b}(x)\int_{0}^{\infty}\abs{\pr(x)-\ee^{-x}}\dx\leq \delta/2$ by means of Proposition~\ref{Prop:L1:convergence:NEW}.
\end{proof}

\section{Preparing the proof of uniqueness}\label{Sec:preparation:uniqueness}

\subsection{Notation}

In this section, we collect the key results from which we will finally deduce the uniqueness of self-similar profiles. Moreover, to simplify the notation later, we define the two bilinear forms
\begin{equation}\label{eq:def:bilinear:forms}
 \begin{split}
   \B_{2}[g,h](x)&\vcc=\frac{2}{x^2}\int_{0}^{x}\int_{x-y}^{\infty}yg(y)h(z)\dz\dy\\
   \B_{W}[g,h](x)&\vcc=\frac{1}{x^2}\int_{0}^{x}\int_{x-y}^{\infty}yW(y,z)g(y)h(z)\dz\dy.
 \end{split}
\end{equation}
\begin{remark}
 We will precise in \cref{Prop:continuity:B2,Prop:continuity:BW} where these operators are well-defined and to which spaces they map.
\end{remark}
This allows to rewrite~\eqref{eq:selfsim} in compact form as
\begin{equation}\label{eq:selfsim:abstract}
 \pr=\B_{2}[\pr,\pr]+\eps\B_{W}[\pr,\pr].
\end{equation}
Moreover, if we linearise this equation for $\eps=0$ around the explicit profile $\ee^{-x}$ this leads to the linearised operator $\LL$ which is given by
\begin{equation}\label{eq:linearised:operator:abstract}
 \LL[g]=g-\B_2[g,\exp(-\cdot)]-\B_2[\exp(-\cdot),g].
\end{equation}
We also fix a parameter 
\begin{equation}\label{eq:choice:beta:parameter}
 \beta\in(1+\alpha,2).
\end{equation}
Finally, we note that we will exploit the fact that due to Fubini's theorem we have the relation
\begin{equation}\label{eq:Fubini}
 \int_{0}^{\infty}\int_{0}^{x}\int_{x-y}^{\infty}(\cdots)\dz\dy\dx=\int_{0}^{\infty}\int_{0}^{\infty}\int_{y}^{y+z}(\cdots)\dx\dz\dy.
\end{equation}

\subsection{Continuity of bilinear forms}

\begin{proposition}\label{Prop:continuity:B2}
 The operator $\B_{2}$ as given in~\eqref{eq:def:bilinear:forms} is well-defined from $\X{-\alpha}{\beta}$ to itself and continuous in the sense that
 \begin{equation*}
  \norm*{\B_{2}[g,h]}_{\X{-\alpha}{\beta}}\lesssim \norm*{g}_{\X{-\alpha}{\beta}}\norm*{h}_{\X{-\alpha}{\beta}}.
 \end{equation*}
\end{proposition}

\begin{proof}
 The definition of $\norm{\cdot}_{\X{-\alpha}{\beta}}$ and~\eqref{eq:Fubini} yield
 \begin{multline*}
  \norm*{\B_{2}[g,h]}_{\X{-\alpha}{\beta}}=\int_{0}^{\infty}\abs*{\frac{2}{x^2}\int_{0}^{x}\int_{x-y}^{\infty}yg(y)h(z)\dz\dy}\weight{-\alpha}{\beta}(x)\dx\\*
  \lesssim \int_{0}^{\infty}\int_{0}^{\infty}y\abs{g(y)}\abs{h(z)}\int_{y}^{y+z}\frac{\weight{-\alpha}{\beta}(x)}{x^2}\dx\dz\dy.
 \end{multline*}
 Applying Lemma~\ref{Lem:aux:int:weight:1} and recalling~\eqref{eq:weight:shift} we further estimate
 \begin{equation*}
  \norm*{\B_{2}[g,h]}_{\X{-\alpha}{\beta}}\lesssim \int_{0}^{\infty}\abs{g(y)}y\weight{-\alpha-1}{0}(y)\dy\int_{0}^{\infty}\abs{h(z)}\weight{0}{\beta-1}(z)\dz\leq \norm{g}_{\X{-\alpha}{1}}\norm{h}_{\X{0}{\beta-1}}.
 \end{equation*}
The claim finally follows from~\eqref{eq:spaces:embedding} since $\beta>1$ and $-\alpha<0$.
\end{proof}

\begin{remark}
 The proof shows that the estimate in Proposition~\ref{Prop:continuity:B2} is non-optimal. In fact the operator has a slightly regularising effect. However, since we will not exploit the latter we only stated this weaker version.
\end{remark}

In contrast to $\B_{2}$, the operator $\B_{W}$ does not map $\X{-\alpha}{\beta}$ to itself (and also no other space $\X{a}{b}$) due to the singular behaviour of $W$. We only get the following weaker statement.

\begin{proposition}\label{Prop:continuity:BW}
 The operator $\B_{W}$ is well-defined from $\X{-\alpha}{\beta}$ to $\X{1-\alpha}{\beta}$ and continuous in the sense that
 \begin{equation*}
  \norm*{\B_{W}[g,h]}_{\X{1-\alpha}{\beta}}\lesssim \norm{g}_{\X{-\alpha}{\beta}}\norm{h}_{\X{-\alpha}{\beta}}.
 \end{equation*}
\end{proposition}

\begin{proof}
 The definition of $\norm{\cdot}_{\X{1-\alpha}{\beta}}$ together with \cref{eq:Fubini,eq:pert:est:weight,eq:weight:shift} gives
 \begin{multline*}
  \norm*{\B_{W}[g,h]}_{\X{1-\alpha}{\beta}}=\int_{0}^{\infty}\abs*{\int_{0}^{x}\int_{x-y}^{\infty}yW(y,z)g(y)h(z)\dz\dy}\weight{1-\alpha}{\beta}(x)\dx\\*
  \lesssim \int_{0}^{\infty}\int_{0}^{\infty}\weight{1-\alpha}{1+\alpha}(y)\weight{-\alpha}{\alpha}(z)\abs{g(y)}\abs{h(z)}\int_{y}^{y+z}\frac{\weight{1-\alpha}{\beta}(x)}{x^2}\dx\dz\dy.
 \end{multline*}
 Recalling Lemma~\ref{Lem:aux:int:weight:1} as well as~\eqref{eq:weight:shift} we get
 \begin{multline*}
  \norm*{\B_{W}[g,h]}_{\X{1-\alpha}{\beta}}\lesssim \int_{0}^{\infty}\int_{0}^{\infty}\weight{1-\alpha}{1+\alpha}(y)\weight{-\alpha}{\alpha}(z)\weight{-\alpha}{0}(y)\weight{0}{\beta-1}(z)\abs{g(y)}\abs{h(z)}\\*
  \lesssim \norm{g}_{\X{1-2\alpha}{1+\alpha}}\norm{h}_{\X{-\alpha}{\beta+\alpha-1}}.
 \end{multline*}
Since $\alpha\in(0,1)$ and $\beta\in(1+\alpha,2)$ we thus conclude together with~\eqref{eq:weight:monotonicity} that $\norm*{\B_{W}[g,h]}_{\X{1-\alpha}{\beta}}\lesssim \norm{g}_{\X{-\alpha}{\beta}}\norm{h}_{\X{-\alpha}{\beta}}$ which finishes the proof.  
\end{proof}

\subsection{Continuity and invertibility of $\LL$}

Next, we prove continuity of the linearised operator $\LL$.

\begin{proposition}\label{Prop:continuity:L:NEW}
 For all $a\in(-1,1)$ and $b>1$, the operator $\LL$ as given by~\eqref{eq:linearised:operator:abstract} maps $\X{a}{b}$ into itself and is continuous, i.e.\@
 \begin{equation*}
  \norm{\LL[g]}_{\X{a}{b}}\lesssim \norm{g}_{\X{a}{b}}.
 \end{equation*}
 In particular, $\LL$ is well-defined on $\X{a}{b}$.
\end{proposition}

\begin{proof}
 We recall from \eqref{eq:L:4} that 
  \begin{equation*}
  \LL[g](x)=h(x)+\frac{2(x+1)\ee^{-x}}{x^2}\int_{0}^{x}(1-\ee^{z})g(z)\dz+\frac{2(x\ee^{-x}+\ee^{-x}-1)}{x^2}\int_{x}^{\infty}g(z)\dz.
 \end{equation*}
 One then immediately checks that 
 \begin{equation}\label{eq:pf:cont:L:1}
  \frac{2(x+1)\ee^{-x}}{x^2}\lesssim \weight{-2}{-1}(x)\ee^{-x}\qquad \text{and}\qquad \abs*{\frac{2(x\ee^{-x}+\ee^{-x}-1)}{x^2}}\lesssim \weight{0}{-2}(x).
 \end{equation}
 Moreover, we have $\abs{1-\ee^{z}}\lesssim \weight{1}{0}(z)\ee^{z}$. Thus, together with Fubini's theorem the definition of $\norm{\cdot}_{\X{a}{b}}$ and \cref{eq:pf:cont:L:1,eq:weight:shift} we obtain
 \begin{multline*}
  \norm*{\LL[g]}_{\X{a}{b}}\leq \norm{g}_{\X{a}{b}}+\int_{0}^{\infty}\frac{2(x+1)\ee^{-x}}{x^2}\weight{a}{b}(x)\int_{0}^{x}\abs{1-\ee^{z}}\abs{g(z)}\dz\dx\\*
  \shoveright{+\int_{0}^{\infty}\abs*{\frac{2(x\ee^{-x}+\ee^{-x}-1)}{x^2}}\weight{a}{b}(x)\int_{x}^{\infty}\abs{g(z)}\dz\dx}\\*
  \lesssim\norm{h}_{\X{a}{b}}+\int_{0}^{\infty}\abs{g(z)}\weight{1}{0}(z)\ee^{z}\int_{z}^{\infty}\weight{a-2}{b-1}(x)\ee^{-x}\dx\dz+\int_{0}^{\infty}\abs{g(z)}\int_{0}^{z}\weight{a}{b-2}(x)\dx\dz.
 \end{multline*}
We recall \cref{eq:prim:weight:2,eq:prim:weight:3} and note that $a\in(-1,1)$ and $b>1$. Then, using also \eqref{eq:weight:additivity} we get
\begin{multline}
 \norm*{\LL[g]}_{\X{a}{b}}\lesssim \norm{g}_{\X{a}{b}} +\int_{0}^{\infty}\abs{g(z)}\weight{a}{b-1}(z)\dz+\int_{0}^{\infty}\abs{g(z)}\weight{a+1}{b-1}(z)\dz\\*
 =\norm{g}_{\X{a}{b}}+\norm{g}_{\X{a}{b-1}}+\norm{g}_{\X{a+1}{b-1}}.
\end{multline}
The claim finally follows from~\eqref{eq:spaces:embedding}.
\end{proof}

The next proposition states that the linearised operator is invertible on suitable spaces $\X[0]{a}{b}$ with continuous inverse $\LL^{-1}$.

\begin{proposition}\label{Prop:continuity:L:inverse}
 For each $a\in(-1,1)$ and $b>1$, the operator $\LL$ is invertible on $\X[0]{a}{b}$ with bounded inverse $\LL^{-1}$ given explicitly by~\eqref{eq:def:inverse} and which maps $\X{a}{b}$ continuously to $\X[0]{a}{b}$. In particular, $\LL^{-1}$ is well-defined on $\X{a}{b}$ and we have the estimate
 \begin{equation*}
  \norm{\LL^{-1}[g]}_{\X{a}{b}}\lesssim \norm{g}_{\X{a}{b}}.
 \end{equation*}
\end{proposition}

Since the proof of Proposition~\ref{Prop:continuity:L:inverse}, even though it is mainly elementary, is quite lengthy, we will move it to Section~\ref{Sec:proof:inversion} below.

\subsection{Estimate of the difference of profiles close to zero}

Moreover, we have the following statement which provides an estimate on the difference of two self-similar profiles.

 \begin{proposition}\label{Prop:bound:layer:est}
  For each $\mu>0$ there exist constants $C_{\mu}>0$ and $\eps_{*}>0$ such that each pair of self-similar profiles $\pr_1$ and $\pr_2$ satisfies the estimate
  \begin{equation*}
   \norm{\ee^{-\cdot}(\dpr{1}-\dpr{2})}_{\X{-\alpha}{\beta}}\lesssim \mu \norm{\dpr{1}-\dpr{2}}_{\X{-\alpha}{\beta}}+C_{\mu}\norm{(1-\ee^{-x})(\dpr{1}-\dpr{2})}_{\X{-\alpha}{\beta}}
  \end{equation*}
 provided that $\eps\in (0,\eps_{*})$.
 \end{proposition}
 
The proof of this result is again relatively long and slightly technical which is why we postpone it to Section~\ref{Sec:proof:bl}.

\section{Uniqueness of profiles}\label{Sec:proof:uniqueness}

Based on the results collected in Section~\ref{Sec:preparation:uniqueness} we will now show that self-similar profiles are unique. The abstract argument can be interpreted as a non-linear version of the implicit function theorem and has already been used in the previous works \cite{NTV15,Thr17a} for the Laplace-transformed quantities. Due to this the proof below is essentially the same as the ones in \cite{NTV15,Thr17a}. However, the essential difference lies in the proofs of the key estimates \cref{Thm:closeness:profiles:NEW,Prop:continuity:B2,Prop:continuity:BW,Prop:continuity:L:inverse,Prop:bound:layer:est} which are now done in an $L^1$ setting rather than for the Laplace transform. 

\begin{proof}[Proof of Theorem~\ref{Thm:uniqueness}]
  Let $\dpr{1}$ and $\dpr{2}$ be two self-similar profiles, i.e.\@ both satisfying~\eqref{eq:selfsim:abstract}. Taking the difference, we obtain
 \begin{equation*}
  \dpr{1}-\dpr{2}=\B_{2}[\dpr{1},\dpr{1}]+\eps\B_{W}[\dpr{1},\dpr{1}]-\B_{2}[\dpr{2},\dpr{2}]+\eps\B_{W}[\dpr{2},\dpr{2}].
 \end{equation*}
 Using the bilinearity of $\B_{2}$ and $\B_{W}$ we can rearrange this equation to get
 \begin{multline*}
  (\dpr{1}-\dpr{2})-\B_{2}\bigl[(\dpr{1}-\dpr{2}),\ee^{-\cdot}\bigr]-\B_{2}\bigl[\ee^{-\cdot},(\dpr{1}-\dpr{2})\bigr]\\*
  =\B_{2}\bigl[(\dpr{1}-\dpr{2}),(\dpr{1}-\ee^{-\cdot})\bigr]+\B_{2}\bigl[(\dpr{2}-\ee^{-\cdot}),(\dpr{1}-\dpr{2})\bigr]\\*
  +\eps\Bigl(\B_{W}\bigl[(\dpr{1}-\dpr{2}),\dpr{1}\bigr]+\B_{W}\bigl[\dpr{2},(\dpr{1}-\dpr{2})\bigr]\Bigr).
 \end{multline*}
 Recalling~\eqref{eq:linearised:operator:abstract} this can be rewritten as
 \begin{multline*}
  \LL[\dpr{1}-\dpr{2}]=\B_{2}\bigl[(\dpr{1}-\dpr{2}),(\dpr{1}-\ee^{-\cdot})\bigr]+\B_{2}\bigl[(\dpr{2}-\ee^{-\cdot}),(\dpr{1}-\dpr{2})\bigr]\\*
  +\eps\Bigl(\B_{W}\bigl[(\dpr{1}-\dpr{2}),\dpr{1}\bigr]+\B_{W}\bigl[\dpr{2},(\dpr{1}-\dpr{2})\bigr]\Bigr).
 \end{multline*}
 We recall from~\eqref{eq:normalisation} that we normalised $\dpr{1}$ and $\dpr{2}$ to have total mass one. Thus, in particular, we have $\dpr{1}-\dpr{2}\in\X[0]{-\alpha}{\beta}$ and by means of Proposition~\ref{Prop:continuity:L:inverse}, the previous equation can be transformed into
 \begin{multline*}
  \dpr{1}-\dpr{2}=\LL^{-1}\Bigl[\B_{2}\bigl[(\dpr{1}-\dpr{2}),(\dpr{1}-\ee^{-\cdot})\bigr]+\B_{2}\bigl[(\dpr{2}-\ee^{-\cdot}),(\dpr{1}-\dpr{2})\bigr]\Bigr]\\*
  +\eps \LL^{-1}\Bigl[\B_{W}\bigl[(\dpr{1}-\dpr{2}),\dpr{1}\bigr]+\B_{W}\bigl[\dpr{2},(\dpr{1}-\dpr{2})\bigr]\Bigr].
 \end{multline*}
 Multiplying by $(1-\ee^{-\cdot})$ and taking the norm $\norm{\cdot}_{\X{-\alpha}{\beta}}$ we find
 \begin{multline*}
  \norm{(1-\ee^{-\cdot})(\dpr{1}-\dpr{2})}_{\X{-\alpha}{\beta}}\\*
  \leq \norm*{(1-\ee^{-\cdot})\LL^{-1}\Bigl[\B_{2}\bigl[(\dpr{1}-\dpr{2}),(\dpr{1}-\ee^{-\cdot})\bigr]+\B_{2}\bigl[(\dpr{2}-\ee^{-\cdot}),(\dpr{1}-\dpr{2})\bigr]\Bigr]}_{\X{-\alpha}{\beta}}\\*
  +\eps \norm*{(1-\ee^{-\cdot})\LL^{-1}\Bigl[\B_{W}\bigl[(\dpr{1}-\dpr{2}),\dpr{1}\bigr]+\B_{W}\bigl[\dpr{2},(\dpr{1}-\dpr{2})\bigr]\Bigr]}_{\X{-\alpha}{\beta}}.
 \end{multline*}
By means of \cref{Lem:norm:regularising,Prop:continuity:L:inverse} the right-hand side can be estimated further as
 \begin{multline*}
  \norm{(1-\ee^{-\cdot})(\dpr{1}-\dpr{2})}_{\X{-\alpha}{\beta}}\\*
  \leq \norm*{\LL^{-1}\Bigl[\B_{2}\bigl[(\dpr{1}-\dpr{2}),(\dpr{1}-\ee^{-\cdot})\bigr]+\B_{2}\bigl[(\dpr{2}-\ee^{-\cdot}),(\dpr{1}-\dpr{2})\bigr]\Bigr]}_{\X{-\alpha}{\beta}}\\*
  \shoveright{+\eps \norm*{\LL^{-1}\Bigl[\B_{W}\bigl[(\dpr{1}-\dpr{2}),\dpr{1}\bigr]+\B_{W}\bigl[\dpr{2},(\dpr{1}-\dpr{2})\bigr]\Bigr]}_{\X{1-\alpha}{\beta}}}\\*
  \lesssim \norm*{\B_{2}\bigl[(\dpr{1}-\dpr{2}),(\dpr{1}-\ee^{-\cdot})\bigr]+\B_{2}\bigl[(\dpr{2}-\ee^{-\cdot}),(\dpr{1}-\dpr{2})\bigr]}_{\X{-\alpha}{\beta}}\\*
  +\eps \norm*{\B_{W}\bigl[(\dpr{1}-\dpr{2}),\dpr{1}\bigr]+\B_{W}\bigl[\dpr{2},(\dpr{1}-\dpr{2})\bigr]}_{\X{1-\alpha}{\beta}}.
 \end{multline*}
 Thus, \cref{Prop:continuity:B2,Prop:continuity:BW} yield
 \begin{multline*}
  \norm{(1-\ee^{-\cdot})(\dpr{1}-\dpr{2})}_{\X{-\alpha}{\beta}}\lesssim \Bigl(\norm{\dpr{1}-\ee^{-\cdot}}_{\X{-\alpha}{\beta}}+\norm{\dpr{2}-\ee^{-\cdot}}_{\X{-\alpha}{\beta}}\Bigr)\norm{\dpr{1}-\dpr{2}}_{\X{-\alpha}{\beta}}\\*
  +\eps \Bigl(\norm{\dpr{1}}_{\X{-\alpha}{\beta}}+\norm{\dpr{2}}_{\X{-\alpha}{\beta}}\Bigr)\norm{\dpr{1}-\dpr{2}}_{\X{-\alpha}{\beta}}.
 \end{multline*}
 As a consequence of \cref{Thm:closeness:profiles:NEW,Lem:moments} we thus get
 \begin{equation}\label{eq:proof:unique:small}
  \norm{(1-\ee^{-\cdot})(\dpr{1}-\dpr{2})}_{\X{-\alpha}{\beta}}\lesssim o(\eps) \norm{\dpr{1}-\dpr{2}}_{\X{-\alpha}{\beta}} \qquad \text{with }o(\eps)\to 0\quad \text{as }\eps \to 0.
 \end{equation}
 Writing $(\dpr{1}-\dpr{2})=\ee^{-\cdot}(\dpr{1}-\dpr{2})+(1-\ee^{-\cdot})(\dpr{1}-\dpr{2})$ and recalling in addition Proposition~\ref{Prop:bound:layer:est} we get for each $\mu>0$ that
 \begin{multline*}
  \norm{\dpr{1}-\dpr{2}}_{\X{-\alpha}{\beta}}\lesssim \norm{\ee^{-\cdot}(\dpr{1}-\dpr{2})}_{\X{-\alpha}{\beta}}+\norm{(1-\ee^{-\cdot})(\dpr{1}-\dpr{2})}_{\X{-\alpha}{\beta}}\\*
  \lesssim \mu \norm{\dpr{1}-\dpr{2}}_{\X{-\alpha}{\beta}} +(C_{\mu}+1)\norm{(1-\ee^{-\cdot})(\dpr{1}-\dpr{2})}_{\X{-\alpha}{\beta}}\\*
  \lesssim \Bigl(\mu+(C_{\mu}+1)o(\eps)\Bigr)\norm{\dpr{1}-\dpr{2}}_{\X{-\alpha}{\beta}}.
 \end{multline*}
To conclude the proof, we take first $\mu>0$ and then $\eps>0$ sufficiently small to get the inequality $\norm{\dpr{1}-\dpr{2}}_{\X{-\alpha}{\beta}}\leq \frac{1}{2}\norm{\dpr{1}-\dpr{2}}_{\X{-\alpha}{\beta}}$ which implies $\dpr{1}=\dpr{2}$.
\end{proof}

\section{Proof of Proposition~\ref{Prop:bound:layer:est}}\label{Sec:proof:bl}

 This section is devoted to the proof of Proposition~\ref{Prop:bound:layer:est} which is relatively technical. To simplify the structure of the actual proof as much as possible, we proceed as follows. In the next subsection, we recall the \emph{boundary layer equation} together with some notation from \cite{NTV15,NTV16} (see also~\cite{Thr17a}). The following two subsections contain then a series of lemmas which provide preliminary results and estimates to prepare the actual proof of Proposition~\ref{Prop:bound:layer:est} in Section~\ref{Sec:actual:proof:bl}. We also note that most arguments in this section follow corresponding ones in \cite{NTV15,Thr17a}. However, as already explained in Section~\ref{Sec:comparison:old:proof}, the complete reasoning in proving Proposition~\ref{Prop:bound:layer:est} in summary is much simpler and to stress this point and to be self-contained, we present all proofs.

\subsection{The boundary layer equation}

We recall in this section the boundary layer equation which has been derived in~\cite{NTV15}. For this, we also recapitulate some notation introduced in~\cite{NTV15} and which we reuse here with slight adaptations again to simplify the comparison. Let $\dpr{k}$ be a self-similar profile, i.e.\@ a solution to~\eqref{eq:selfsim}. Then we define the expressions
\begin{equation}\label{eq:def:beta}
 \beta_{2}[\dpr{k}]\vcc=2\int_{0}^{\infty}\dpr{k}(z)\dz \quad \text{and}\quad \beta_{W}[\dpr{k}](x)\vcc=\int_{0}^{\infty}W(x,z)\dpr{k}(z)\dz.
\end{equation}
Moreover, we introduce the operator
\begin{equation}\label{eq:def:Phi}
 \Phi[\dpr{k}](x)\vcc=\eps\int_{x}^{\infty}\frac{\beta_{W}[\dpr{k}](y)}{y}\ee^{-y}\dy
\end{equation}
as well as the constant
\begin{equation}\label{eq:def:kappa}
 \kappa[\dpr{k}]\vcc=\beta_{2}[\dpr{k}]-2.
\end{equation}
To shorten the notation at some places, we also write
\begin{equation*}
 \Phi_{k}\vcc=\Phi[\dpr{k}]\qquad \text{and}\qquad  \kappa_{k}\vcc=\kappa[\dpr{k}].
\end{equation*}
With this, we can now recall the boundary layer equation from~\cite{NTV15}. Precisely, each self-similar profile $\dpr{k}$ satisfies
\begin{multline}\label{eq:bl:1}
 \dpr{k}(x)=-\eps \int_{x}^{\infty}\Bigl(\frac{x}{z}\Bigr)^{\kappa[\dpr{k}]}\ee^{\Phi[\dpr{k}](z)-\Phi[\dpr{k}](x)}\beta_{W}[\dpr{k}](z)\frac{1-\ee^{-z}}{z}\dpr{k}(z)\dz\\*
 +\int_{x}^{\infty}\Bigl(\frac{x}{z}\Bigr)^{\kappa[\dpr{k}]}\frac{1}{z^2}\ee^{\Phi[\dpr{k}](z)-\Phi[\dpr{k}](x)}\int_{0}^{z}K_{\eps}(y,z-y)y\dpr{k}(y)\dpr{k}(z-y)\dy\dz.
\end{multline}
\begin{remark}
 We note that the derivation of this equation requires the self-similar profiles to be differentiable which is why we provide this property in Proposition~\ref{Prop:differentiability}.
\end{remark}
Since we will have to derive an estimate for the difference of two solutions, let $\dpr{1}$ and $\dpr{2}$ be self-similar profiles which thus both satisfy~\eqref{eq:bl:1}. We take the difference of the corresponding equations and rewrite the right-hand side which leads to
\begin{multline}\label{eq:bl:0}
 \dpr{1}(x)-\dpr{2}(x)=-\eps \int_{x}^{\infty}\Bigl[(x/z)^{\kappa_{1}}-(x/z)^{\kappa_{2}}\Bigr]\ee^{\Phi_{1}(z)-\Phi_{1}(x)}\beta_{W}[\dpr{1}](z)\frac{1-\ee^{-z}}{z}\dpr{1}(z)\dz\\*
 -\eps \int_{x}^{\infty}(x/z)^{\kappa_{2}}\Bigl[\ee^{\Phi_{1}(z)-\Phi_{1}(x)}-\ee^{\Phi_{2}(z)-\Phi_{2}(x)}\Bigr]\beta_{W}[\dpr{1}](z)\frac{1-\ee^{-z}}{z}\dpr{1}(z)\dz\\*
 -\eps \int_{x}^{\infty}(x/z)^{\kappa_{2}}\ee^{\Phi_{2}(z)-\Phi_{2}(x)}\Bigl[\beta_{W}[\dpr{1}](z)-\beta_{W}[\dpr{2}](z)\Bigr]\frac{1-\ee^{-z}}{z}\dpr{1}(z)\dz\\*
 -\eps \int_{x}^{\infty}(x/z)^{\kappa_{2}}\ee^{\Phi_{2}(z)-\Phi_{2}(x)}\beta_{W}[\dpr{2}](z)\frac{1-\ee^{-z}}{z}\Bigl[\dpr{1}(z)-\dpr{2}(z)\Bigr]\dz\\*
 +\int_{x}^{\infty}\Bigl[(x/z)^{\kappa_{1}}-(x/z)^{\kappa_{2}}\Bigr]\frac{1}{z^2}\ee^{\Phi_{1}(z)-\Phi_{1}(x)}\int_{0}^{z}K_{\eps}(y,z-y)y\dpr{1}(y)\dpr{1}(z-y)\dy\dz\\*
 +\int_{x}^{\infty}(x/z)^{\kappa_{2}}\frac{1}{z^2}\Bigl[\ee^{\Phi_{1}(z)-\Phi_{1}(x)}-\ee^{\Phi_{2}(z)-\Phi_{2}(x)}\Bigr]\int_{0}^{z}K_{\eps}(y,z-y)y\dpr{1}(y)\dpr{1}(z-y)\dy\dz\\*
 +\int_{x}^{\infty}\!(x/z)^{\kappa_{2}}\frac{\ee^{\Phi_{2}(z)-\Phi_{2}(x)}}{z^2}\int_{0}^{z}K_{\eps}(y,z-y)\Bigl[\dpr{1}(y)-\dpr{2}(y)\Bigl]\Bigl(y\dpr{1}(z-y)+(z-y)\dpr{2}(z-y)\Bigr)\dy\dz\\*
 =\vcc -\eps\sum_{k=1}^{4} \K_{k}[\dpr{1},\dpr{2}](x)+\sum_{k=1}^{3}\J_{k}[\dpr{1},\dpr{2}](x).
\end{multline}
Note that in the last step we also exploited the symmetry of the kernel $K_{\eps}(x,y)$. Precisely, rewriting 
\begin{multline*}
 \dpr{1}(y)\dpr{1}(z-y)-\dpr{2}(y)\dpr{2}(z-y)\\*
 =\bigl(\dpr{1}(y)-\dpr{2}(y)\bigr)\dpr{1}(z-y)+\dpr{2}(y)\bigl(\dpr{1}(z-y)-\dpr{2}(z-y)\bigr)
\end{multline*}
 and using the change of variables $y\mapsto z-y$ in the second part of the integral together with the symmetry of $K_{\eps}(x,y)$ we get
 \begin{multline*}
  \int_{0}^{z}K_{\eps}(y,z-y)y\bigl(\dpr{1}(y)\dpr{1}(z-y)-\dpr{2}(y)\dpr{2}(z-y)\bigr)\dy\\*
  =\int_{0}^{z}K_{\eps}(y,z-y)y\Bigl(\bigl(\dpr{1}(y)-\dpr{2}(y)\bigr)\dpr{1}(z-y)+\dpr{2}(y)\bigl(\dpr{1}(z-y)-\dpr{2}(z-y)\bigr)\Bigr)\dy\\*
  \shoveleft{=\int_{0}^{z}K_{\eps}(y,z-y)y\bigl(\dpr{1}(y)-\dpr{2}(y)\bigr)\dpr{1}(z-y)\dy}\\*
  \shoveright{+\int_{0}^{z}K_{\eps}(z-y,y)(z-y)\dpr{2}(z-y)\bigl(\dpr{1}(y)-\dpr{2}(y)\bigr)\dy}\\*
  =\int_{0}^{z}K_{\eps}(y,z-y)\Bigl(\dpr{1}(y)-\dpr{2}(y)\Bigl)\Bigl(y\dpr{1}(z-y)+(z-y)\dpr{2}(z-y)\Bigr)\dy.
 \end{multline*}
 The task now consists in estimating the expressions $\exp(-\cdot )\K_{k}[\dpr{1},\dpr{2}]$ for $k=1\ldots 4$ and $\exp(-\cdot)\J_{k}[\dpr{1},\dpr{2}]$ for $k=1\ldots 3$ and we will treat the respective terms separately. Before we give the actual estimates, let us collect several preliminary results to simplify the structure of the following proofs.

 \subsection{Preparatory  estimates}
 
 The first result which we prove provides estimates on two auxiliary integrals.
 
 \begin{lemma}\label{Lem:aux:int:bl:est}
  Assume $\alpha\in(0,1)$ and let $\beta$ as in~\eqref{eq:choice:beta:parameter}. For $a\in[0,2\alpha]$ and $b\leq 0$ we have
  \begin{align*}
   \int_{0}^{\infty}\int_{0}^{\infty}\weight{-a}{b}(y+z)W(y,z)g(y)h(z)\dz\dy&\lesssim \norm{g}_{\X{-\alpha}{\beta}}\norm{h}_{\X{-\alpha}{\beta}}
   \shortintertext{and}
   \int_{0}^{\infty}\int_{0}^{\infty}\weight{-a}{b}(y+z)K_{\eps}(y,z)g(y)h(z)\dz\dy&\lesssim \norm{g}_{\X{-\alpha}{\beta}}\norm{h}_{\X{-\alpha}{\beta}}
  \end{align*}
 for all $f,g\in\X{-\alpha}{\beta}$.
 \end{lemma}

 \begin{proof}
  Since~\eqref{eq:Ass:K2} gives $W(y,z)\lesssim ((y/z)^{\alpha}+(z/y)^{\alpha})$ but also implies $K_{\eps}(y,z)\lesssim ((y/z)^{\alpha}+(z/y)^{\alpha})$, it suffices to show the stated estimate with $K_{\eps}$ or respectively $W$ replaced by $((y/z)^{\alpha}+(z/y)^{\alpha})$. Moreover, due to~\eqref{eq:weight:monotonicity} it is sufficient to consider only $a=2\alpha$ and $b=0$. For this case, we note that we have both, $\weight{-2\alpha}{0}(y+z)\leq \weight{-2\alpha}{0}(y)$ and $\weight{-2\alpha}{0}(y+z)\leq\weight{-2\alpha}{0}(z)$ which yields together with~\eqref{eq:weight:shift} and $x^{-\alpha}=\weight{-\alpha}{-\alpha}(x)$ that
  \begin{multline*}
   \int_{0}^{\infty}\int_{0}^{\infty}\weight{-2\alpha}{0}(y+z)\biggl(\Bigl(\frac{y}{z}\Bigr)^{\alpha}+\Bigl(\frac{z}{y}\Bigr)^{\alpha}\biggr)g(y)h(z)\dz\dy\\*
   \lesssim \int_{0}^{\infty}\int_{0}^{\infty}\bigl(\weight{-\alpha}{\alpha}(y)\weight{-\alpha}{-\alpha}(z)+\weight{-\alpha}{-\alpha}(y)\weight{-\alpha}{\alpha}(z)\bigr)g(y)h(z)\dz\dy\\*
   =\norm{g}_{\X{-\alpha}{\alpha}}\norm{h}_{\X{-\alpha}{-\alpha}}+\norm{g}_{\X{-\alpha}{-\alpha}}\norm{h}_{\X{-\alpha}{\alpha}}.
  \end{multline*}
 This estimate together with~\eqref{eq:spaces:embedding} finishes the proof since $\alpha<1+\alpha<\beta$.
 \end{proof}
 
 The next lemma shows that the constant $\kappa$, as defined in~\eqref{eq:def:kappa}, becomes small if $\eps$ is small.

  \begin{lemma}\label{Lem:smallness:kappa}
  For each $\eps>0$ there exists $\nu_{\eps}\geq 0$ satisfying $\nu_{\eps}\to 0$ as $\eps\to 0$ such that
  \begin{equation*}
   \abs{\kappa[\pr]}\leq \nu_{\eps}
  \end{equation*}
 for all self-similar profiles $\pr$ and $\kappa$ as defined in~\eqref{eq:def:kappa}.
 \end{lemma}
 
 \begin{proof}
  Observing $\abs{\kappa[\pr]}=\abs*{\beta_{2}[\pr]-2}=2\abs*{\int_{0}^{\infty}\pr(x)-\ee^{-x}\dx}\leq 2\int_{0}^{\infty}\abs{\pr(x)-\ee^{-x}}\dx=2\norm{\pr-\ee^{-\cdot}}_{L^{1}(0,\infty)}$, the result is an immediate consequence of Proposition~\ref{Prop:L1:convergence:NEW}.
 \end{proof}

 The next lemma provides an interpolation statement for the $L^1$ norm of the difference of two self-similar profiles.
 
 \begin{lemma}\label{Lem:est:kappa}
  For each $\mu\in(0,1)$ there exists $C_{\mu}>0$ such that
  \begin{equation*}
   \int_{0}^{\infty}\abs{\dpr{1}(z)-\dpr{2}(z)}\dz \leq \mu\norm{\dpr{1}-\dpr{2}}_{\X{-\alpha}{\beta}}+C_{\mu}\norm*{\bigl(1-\exp(-\cdot)\bigr)(\dpr{1}-\dpr{2})}_{\X{-\alpha}{\beta}}
  \end{equation*}
  for each pair of solutions $\dpr{1}$ and $\dpr{2}$ to~\eqref{eq:selfsim}. In particular, we have
  \begin{equation*}
   \abs*{\kappa[\dpr{1}]-\kappa[\dpr{2}]}\leq \mu\norm{\dpr{1}-\dpr{2}}_{\X{-\alpha}{\beta}}+C_{\mu}\norm*{\bigl(1-\exp(-\cdot)\bigr)(\dpr{1}-\dpr{2})}_{\X{-\alpha}{\beta}}
  \end{equation*}
  with $\kappa$ given by~\eqref{eq:def:kappa}. 
 \end{lemma}

 \begin{proof}
  We note that due to the definition of $\kappa$ in~\eqref{eq:def:kappa} we have
  \begin{equation*}
   \abs*{\kappa[\dpr{1}]-\kappa[\dpr{2}]}=2\abs*{\int_{0}^{\infty}\dpr{1}(z)-\dpr{2}(z)\dz}\leq 2\int_{0}^{\infty}\abs*{\dpr{1}(z)-\dpr{2}(z)}\dz.
  \end{equation*}
 Thus, it suffices to prove only the first part of the lemma. For given $\mu>0$ we split the integral and rewrite to get
 \begin{multline*}
  \int_{0}^{\infty}\abs*{\dpr{1}(z)-\dpr{2}(z)}\dz\\*
  =\int_{0}^{\mu^{1/\alpha}}z^{\alpha}z^{-\alpha}\abs*{\dpr{1}(z)-\dpr{2}(z)}\dz+\int_{\mu^{1/\alpha}}^{\infty}\frac{1}{1-\ee^{-z}}(1-\ee^{-z})\abs*{\dpr{1}(z)-\dpr{2}(z)}\dz\\*
  \leq \mu \int_{0}^{1}z^{-\alpha}\abs*{\dpr{1}(z)-\dpr{2}(z)}\dz+\frac{1}{1-\ee^{-\mu^{1/\alpha}}}\int_{0}^{\infty}(1-\ee^{-z})\abs*{\dpr{1}(z)-\dpr{2}(z)}\dz.
 \end{multline*}
 Since $1\leq \weight{-\alpha}{\beta}(z)$ we can estimate the right-hand side which further yields
 \begin{multline*}
  \int_{0}^{\infty}\abs*{\dpr{1}(z)-\dpr{2}(z)}\dz\\*
  \leq \mu \int_{0}^{1}\abs*{\dpr{1}(z)-\dpr{2}(z)}\weight{-\alpha}{\beta}(z)\dz+\frac{1}{1-\ee^{-\mu^{1/\alpha}}}\int_{0}^{\infty}(1-\ee^{-z})\abs*{\dpr{1}(z)-\dpr{2}(z)}\weight{-\alpha}{\beta}(z)\dz\\*
  =\mu\norm{\dpr{1}-\dpr{2}}_{\X{-\alpha}{\beta}}+C_{\mu}\norm*{\bigl(1-\exp(-\cdot)\bigr)(\dpr{1}-\dpr{2})}_{\X{-\alpha}{\beta}}.
 \end{multline*}
 This finishes the proof.
 \end{proof}
 
 \begin{remark}\label{Rem:diff:kappa}
  Note that Lemma~\ref{Lem:est:kappa} together with Lemma~\ref{Lem:norm:regularising} in particular gives
  \begin{equation*}
   \abs{\kappa[\dpr{1}]-\kappa[\dpr{2}]}\lesssim \norm{\dpr{1}-\dpr{2}}_{\X{-\alpha}{\beta}}
  \end{equation*}
 for each pair of self-similar profiles $\dpr{1}$ and $\dpr{2}$. We also note that the first part of Lemma~\ref{Lem:est:kappa} holds true for each pair $g,h\in\X{-\alpha}{\beta}$ instead of $\dpr{1}$ and $\dpr{2}$.
 \end{remark}
 
 The next two lemmas provide estimates on the functional $\beta_{W}$ applied to self-similar profiles.
 
 \begin{lemma}\label{Lem:est:beta:W}
  The functional $\beta_{W}$ as given by~\eqref{eq:def:beta} satisfies the estimate
  \begin{equation*}
   \abs{\beta_{W}[\pr](x)}\lesssim \weight{-\alpha}{\alpha}(x)\qquad \text{for }x>0
  \end{equation*}
 and all self-similar profiles $\pr$ if $\eps$ is sufficiently small.
 \end{lemma}
 
 \begin{proof}
  From the definition of $\beta_{W}$ in~\eqref{eq:def:beta} together with~\eqref{eq:pert:est:weight} we deduce
  \begin{equation*}
   \abs{\beta_{W}[\pr](x)}=\abs*{\int_{0}^{\infty}W(x,z)\pr(z)\dz}\lesssim \weight{-\alpha}{\alpha}(x)\int_{0}^{\infty}\weight{-\alpha}{\alpha}(z)\pr(z)\dz\lesssim \weight{-\alpha}{\alpha}(x).
  \end{equation*}
 In the last step we exploited the uniform boundedness of the integral which is guaranteed by Lemma~\ref{Lem:moments}.
 \end{proof}

 \begin{lemma}\label{Lem:est:diff:beta:W}
  For $\beta_{W}$ as defined in~\eqref{eq:def:beta} we have the estimate
  \begin{equation*}
   \abs{\beta_{W}[\dpr{1}](x)-\beta_{W}[\dpr{2}](x)}\lesssim \weight{-\alpha}{\alpha}(x)\norm{\dpr{1}-\dpr{2}}_{\X{-\alpha}{\beta}} \qquad \text{for all }x>0
  \end{equation*}
  and for all pairs of self-similar profiles $\dpr{1}$ and $\dpr{2}$ with sufficiently small $\eps$.
 \end{lemma}
 
 \begin{proof}
  The estimate follows immediately from the definition and~\eqref{eq:pert:est:weight}. In fact, we have
  \begin{multline*}
   \abs{\beta_{W}[\dpr{1}](x)-\beta_{W}[\dpr{2}](x)}=\abs*{\int_{0}^{\infty}W(x,z)\bigl(\dpr{1}(z)-\dpr{2}(z)\bigr)\dz}\\*
   \lesssim \weight{-\alpha}{\alpha}(x)\int_{0}^{\infty}\weight{-\alpha}{\alpha}(z)\abs{\dpr{1}(z)-\dpr{2}(z)}\dz=\weight{-\alpha}{\alpha}(x)\norm{\dpr{1}-\dpr{2}}_{\X{-\alpha}{\alpha}}.
  \end{multline*}
 The claim thus follows from~\eqref{eq:spaces:embedding} together with $\beta>\alpha$.
 \end{proof}
 
 Next, we prove that the difference of the functional $\Phi$ applied to two self-similar solutions can be estimated by the norm of the difference of the profiles.
 
 \begin{lemma}\label{Lem:est:diff:Phi}
  For $\Phi$ as defined in~\eqref{eq:def:Phi} we have the estimate
  \begin{equation*}
   \abs{\Phi[\dpr{1}](x)-\Phi[\dpr{2}](x)}\lesssim \eps \weight{-\alpha}{\alpha-1}(x)\ee^{-x}\norm{\dpr{1}-\dpr{2}}_{\X{-\alpha}{\beta}}\qquad \text{for all }x>0
  \end{equation*}
 and each pair of self-similar profiles $\dpr{1}$ and $\dpr{2}$ with $\eps$ sufficiently small. In particular, we also have the bound $\abs{\Phi[\dpr{1}](x)-\Phi[\dpr{2}](x)}\lesssim \eps \weight{-\alpha}{\alpha-1}(x)\norm{\dpr{1}-\dpr{2}}_{\X{-\alpha}{\beta}}$.
 \end{lemma}
 
 \begin{proof}
  Using Lemma~\ref{Lem:est:diff:beta:W} together with the definition of $\Phi$ and~\eqref{eq:weight:shift} we get
  \begin{multline*}
   \abs{\Phi[\dpr{1}](x)-\Phi[\dpr{2}](x)}=\eps\abs*{\int_{x}^{\infty}\frac{(\beta_{W}[\dpr{1}](y)-\beta_{W}[\dpr{2}](y)}{y}\ee^{-y}\dy}\\*
   \lesssim \eps\norm{\dpr{1}-\dpr{2}}_{\X{-\alpha}{\beta}}\int_{x}^{\infty}\frac{\weight{-\alpha}{\alpha}(y)}{y}\ee^{-y}\dy=\eps\norm{\dpr{1}-\dpr{2}}_{\X{-\alpha}{\beta}}\int_{x}^{\infty}\weight{-\alpha-1}{\alpha-1}(y)\ee^{-y}\dy.
  \end{multline*}
 Recalling from~\eqref{eq:prim:weight:3} that $\int_{x}^{\infty}\weight{-\alpha-1}{\alpha-1}(y)\ee^{-y}\dy\lesssim \weight{-\alpha}{\alpha-1}(x)\ee^{-x}$ for all $x>0$ the claim follows.
 \end{proof}
 
 For $\alpha\geq 1/2$ we require also a lower bound on finite differences of the functional $\Phi$.

 \begin{lemma}\label{Lem:diff:Phi:lower:bound}
  Assume that $\alpha\in[1/2,1)$ and let $W$ satisfy additionally~\eqref{eq:W:lower:bound}. For sufficiently small $\eps>0$ there exists a constant $c_{\alpha}>0$ such that 
  \begin{equation*}
   \Phi[\pr](x)-\Phi[\pr](x+\tau)\geq  c_{\alpha}\eps\begin{cases}
                                                   x^{-\alpha} &\text{if }\tau\geq x\\
                                                   x^{-1-\alpha}\tau & \text{if }\tau\leq x
                                                  \end{cases}
  \qquad \text{for all }x\leq 1
  \end{equation*}
 and all self-similar profiles $\pr$ where $\Phi[\pr]$ is defined in~\eqref{eq:def:Phi}.
 \end{lemma}
 
 \begin{proof}
  We first note that the non-negativity of $\pr$ together with~\cref{eq:W:lower:bound,eq:def:Phi} yields
  \begin{multline*}
   -\Phi'[\pr](x)=\eps \int_{0}^{\infty}\frac{W(x,z)}{x}\ee^{-x}\pr(z)\dz\\*
   \geq c_{*}\eps \int_{0}^{\infty}x^{-1-\alpha}z^{\alpha}\ee^{-x}\pr(z)\dz\geq \frac{c_{*}\eps}{\ee}\mom_{\alpha}[\pr]x^{-1-\alpha}.
  \end{multline*}
  Since $\abs{\mom_{\alpha}[\pr]-\mom_{\alpha}[\exp(-\cdot)]}\leq \norm{\pr-\exp(-\cdot)}_{\X{\alpha}{\alpha}}$ we deduce from Theorem~\ref{Thm:closeness:profiles:NEW} upon choosing $\eps>0$ sufficiently small that $\mom_{\alpha}[\pr]\geq \mom_{\alpha}[\exp(-\cdot)]/2$. Thus, we get $-\Phi'[\pr](x)\geq \frac{c_{*}\eps}{2\ee}\mom_{\alpha}[\exp(-\cdot)]x^{-1-\alpha}$ which further yields
 \begin{equation}\label{eq:proof:diff:Phi:1}
  \Phi[\pr](x)-\Phi[\pr](x+\tau)=\int_{x}^{x+\tau}-\Phi'[\pr](z)\dz\geq \frac{c_{*}\eps}{2\ee}\mom_{\alpha}[\exp(-\cdot)]\int_{x}^{x+\tau}z^{-1-\alpha}\dz.
 \end{equation}
 Now we have to distinguish whether $\tau\geq x$ or $\tau\leq x$. In the first case we compute the integral on the right-hand side and estimate further to get
  \begin{equation}\label{eq:proof:diff:Phi:2}
  \Phi[\pr](x)-\Phi[\pr](x+\tau)\geq \frac{c_{*}\eps}{2\ee\alpha}\mom_{\alpha}[\exp(-\cdot)]\bigl(x^{-\alpha}-(x+\tau)^{-\alpha}\bigr)\geq \frac{c_{*}\eps}{2\ee\alpha}\mom_{\alpha}[\exp(-\cdot)]\bigl(1-2^{-\alpha}\bigr)x^{-\alpha}.
 \end{equation}
 In the last step we used that $-(x+\tau)^{-\alpha}\geq -2^{-\alpha}x^{-\alpha}$ since $\tau\geq x$. On the other hand, if $\tau\leq x$, we estimate the right-hand side of~\eqref{eq:proof:diff:Phi:1} to get
 \begin{multline}\label{eq:proof:diff:Phi:3}
  \Phi[\pr](x)-\Phi[\pr](x+\tau)\geq \frac{c_{*}\eps}{2\ee}\mom_{\alpha}[\exp(-\cdot)]\geq \frac{c_{*}\eps}{2\ee}\mom_{\alpha}[\exp(-\cdot)](x+\tau)^{-\alpha-1}\tau\\*
  \geq \frac{c_{*}\eps}{2\ee} 2^{-1-\alpha}\mom_{\alpha}[\exp(-\cdot)]x^{-1-\alpha}\tau.
 \end{multline}
 In the last step, we exploited $\tau\leq x$ to estimate $(x+\tau)^{-1-\alpha}\geq 2^{-1-\alpha}x^{-1-\alpha}$. Summarising~\cref{eq:proof:diff:Phi:2,eq:proof:diff:Phi:3} and choosing $c_{\alpha}\vcc=(c_{*}/(2\ee))\min\{2^{-1-\alpha},(1-2^{-\alpha})/\alpha\}$ the claim follows.
 \end{proof}
 
 The next statement quantifies the regularising effect of the exponential decay of self-similar profiles close to zero.
 
 \begin{lemma}\label{Lem:regularising:effect}
  Assume $\alpha\geq 1/2$ and that $W$ satisfies~\eqref{eq:W:lower:bound}. If $a\in(0,1+\alpha)$ and $\delta\in(0,\min\{1+\alpha-a,\alpha\})$ we have for sufficiently small $\eps>0$ and all $b\in\R$ that 
  \begin{equation*}
   \int_{0}^{z}\weight{-a}{b}(x)\ee^{-x}\ee^{-(\Phi[\pr](x)-\Phi[\pr](z))}\dx \lesssim \eps^{\frac{\delta-\alpha}{\alpha}}\weight{1+\alpha-a-\delta}{0}(z)
  \end{equation*}
 for all self-similar solutions $\pr$.
 \end{lemma}

 \begin{proof}
  We first consider the case $z\leq 1$. Splitting the integral, using $\ee^{-x}\leq 1$ and recalling Lemma~\ref{Lem:diff:Phi:lower:bound} we get
  \begin{equation}\label{eq:proof:reg:effect:1}
   \int_{0}^{z}\weight{-a}{b}(x)\ee^{-x}\ee^{-(\Phi[\pr](x)-\Phi[\pr](z)}\dx\leq \int_{0}^{z/2}x^{-a}\ee^{-c_{\alpha}\eps x^{-\alpha}}\dx+\int_{z/2}^{z}x^{-a}\ee^{-c_{\alpha}\eps x^{-\alpha-1}(z-x)}\dx.
  \end{equation}
 We estimate the two terms on the right-hand side separately. For this, we first rewrite $x^{-a}=x^{\alpha-a-\delta}x^{\delta-\alpha}$ and note that the function $x\mapsto x^{\delta-\alpha}\ee^{-c_{\alpha}\eps x^{-\alpha}}$ attains its unique maximum at $x_{*}=((\alpha-\delta)/(c_{\alpha}\alpha\eps))^{-1/\alpha}$. Thus, we have the bound
 \begin{multline*}
  x^{-a}\ee^{-c_{\alpha}\eps x^{-\alpha}}\leq x^{\alpha-a-\delta}x^{\delta-\alpha}\ee^{-c_{\alpha}\eps x^{-\alpha}}\leq x_{*}^{\delta-\alpha}\ee^{-c_{\alpha}\eps x_{*}^{-\alpha}}x^{\alpha-a-\delta}\\*
  \leq \Bigl(\frac{c_{\alpha}\alpha}{\alpha-\delta}\Bigr)^{\frac{\alpha-\delta}{\alpha}}\ee^{-\frac{\alpha-\delta}{\alpha}}\eps^{-\frac{\alpha-\delta}{\alpha}}x^{\alpha-a-\delta}\lesssim\eps^{-\frac{\alpha-\delta}{\alpha}}x^{\alpha-a-\delta}.
 \end{multline*}
 From this, we get for the first integral on the right-hand side of~\eqref{eq:proof:reg:effect:1} that
 \begin{equation}\label{eq:proof:reg:effect:2}
  \int_{0}^{z/2}x^{-a}\ee^{-c_{\alpha}\eps x^{-\alpha}}\dx\lesssim \eps^{-\frac{\alpha-\delta}{\alpha}}\int_{0}^{z/2}x^{\alpha-a-\delta}\dx\lesssim \eps^{-\frac{\alpha-\delta}{\alpha}}z^{1+\alpha-a-\delta}.
 \end{equation}
 To estimate the second integral on the right-hand side of~\eqref{eq:proof:reg:effect:1} we first change variables $x\mapsto zx$ and then use $x^{-a}\leq 2^a$ and $-x^{-1-\alpha}\leq -1$ for $x\in[1/2,1]$ which yields
 \begin{multline*}
  \int_{z/2}^{z}x^{-a}\ee^{-c_{\alpha}\eps x^{-\alpha-1}(z-x)}\dx=z^{1-a}\int_{1/2}^{1}x^{-a}\ee^{-c_{\alpha}\eps z^{-\alpha}x^{-1-\alpha}(1-x)}\dx\\*
  \leq 2^{a} z^{1-a}\int_{1/2}^{1}\ee^{-c_{\alpha}\eps z^{-\alpha}(1-x)}\dx.
 \end{multline*}
The last integral on the right-hand side can be computed explicitly such that we get
 \begin{equation*}
  \int_{z/2}^{z}x^{-a}\ee^{-c_{\alpha}\eps x^{-\alpha-1}(z-x)}\dx\leq 2^{a}z^{1-a}\bigl(c_{\alpha}\eps z^{-\alpha}\bigr)^{-1}\bigl(1-\ee^{-\frac{c_{\alpha}\eps z^{-\alpha}}{2}}\bigr).
 \end{equation*}
 Exploiting that $1-\ee^{-A x}\lesssim (Ax)^{\gamma}$ for each $\gamma\in[0,1]$, we conclude with $\gamma=\delta/\alpha$ that
  \begin{equation}\label{eq:proof:reg:effect:3}
  \int_{z/2}^{z}x^{-a}\ee^{-c_{\alpha}\eps x^{-\alpha-1}(z-x)}\dx\leq 2^{a}z^{1-a}\bigl(c_{\alpha}\eps z^{-\alpha}\bigr)^{\frac{\delta}{\alpha}-1}\lesssim \eps^{-\frac{\alpha-\delta}{\alpha}}z^{1-a+\alpha-\delta}.
 \end{equation}
Combining~\cref{eq:proof:reg:effect:1,eq:proof:reg:effect:2,eq:proof:reg:effect:3} the claim follows for $z\leq 1$.

For $z\geq 1$, we now exploit the monotonicity of $\Phi[\pr](\cdot)$ which yields $\ee^{-(\Phi[\pr](x)-\Phi[\pr](z))}\leq 1$ for $x\leq z$. Then, splitting the integral again we find together with $\ee^{-x}\leq 1$ that
\begin{multline*}
 \int_{0}^{z}\weight{-a}{b}(x)\ee^{-x}\ee^{-(\Phi[\pr](x)-\Phi[\pr](z))}\dx\leq \int_{0}^{1}x^{-a}\ee^{-(\Phi[\pr](x)-\Phi[\pr](z)}\dx+\int_{1}^{z}x^{b}\ee^{-x}\dx\\*
 \leq \int_{0}^{1}x^{-a}\ee^{-(\Phi[\pr](x)-\Phi[\pr](z))}\dx+\int_{1}^{\infty}x^{b}\ee^{-x}\dx.
\end{multline*}
 The second integral on the right-hand side is obviously bounded by a constant. According to the first part of the proof, also the first integral can be uniformly bounded by a constant which thus finishes the proof.
\end{proof}

Finally, we provide estimates on an auxiliary integral which will appear later in the proof of Proposition~\ref{Prop:bound:layer:est} and which is the reason why the case $\alpha\geq 1/2$ needs some special care.
 
 \begin{lemma}\label{Lem:reg:ef:diff:Phi}
  Let $\Phi$ be as defined in~\eqref{eq:def:Phi} and $\eps$ sufficiently small. If $\alpha\in(0,1/2)$ we have  for $a<1-\alpha$ and all $b\in\R$ the estimate
  \begin{equation*}
   \int_{0}^{z}\weight{-a}{b}(x)\ee^{-x}\abs*{\ee^{\Phi[\dpr{1}](z)-\Phi[\dpr{1}](x)}-\ee^{\Phi[\dpr{2}](z)-\Phi[\dpr{2}](x)}}\dx\lesssim \eps \weight{1-a-\alpha}{0}(z)\norm{\dpr{1}-\dpr{2}}_{\X{-\alpha}{\beta}}
  \end{equation*}
 for all $z>0$ and for each pair of self-similar profiles $\dpr{1}$ and $\dpr{2}$.
 
 If $\alpha\in[1/2,1)$, assume in addition that $W$ satisfies~\eqref{eq:W:lower:bound}. Then for $a\in(0,1)$ and $\delta\in(0,\min\{1-a,\alpha\})$ we have for sufficiently small $\eps>0$ that
  \begin{equation*}
   \int_{0}^{z}\weight{-a}{b}(x)\ee^{-x}\abs*{\ee^{\Phi[\dpr{1}](z)-\Phi[\dpr{1}](x)}-\ee^{\Phi[\dpr{2}](z)-\Phi[\dpr{2}](x)}}\dx\lesssim \eps^{\frac{\delta}{\alpha}} \weight{1-a-\delta}{0}(z)\norm{\dpr{1}-\dpr{2}}_{\X{-\alpha}{\beta}}.
  \end{equation*}
 \end{lemma}
 
 \begin{proof}
  We first note that the elementary inequality $\abs{\ee^{-u}-\ee^{-v}}\lesssim \ee^{-\min\{u,v\}}\abs{u-v}$ for $u,v\geq 0$ together with Lemma~\ref{Lem:est:diff:Phi}, the monotonicity of $\Phi$ and the assumption $x\leq z$ yields for all $\alpha\in(0,1)$ the estimate
 \begin{align}
  &\phantom{{}\leq{}}\abs*{\ee^{\Phi[\dpr{1}](z)-\Phi[\dpr{1}](x)}-\ee^{\Phi[\dpr{2}](z)-\Phi[\dpr{2}](x)}}\nonumber\\
  &\lesssim \ee^{-\min_{m\in\{1,2\}}\{\Phi[\dpr{m}](x)-\Phi[\dpr{m}](z)\}}\Bigl[\abs{\Phi[\dpr{1}](x)-\Phi[\dpr{2}](x)}+\abs{\Phi[\dpr{1}](z)-\Phi[\dpr{2}](z)}\Bigr]\nonumber\\
  &\lesssim\eps\ee^{-\min_{m\in\{1,2\}}\{\Phi[\dpr{m}](x)-\Phi[\dpr{m}](z)\}} \bigl(\weight{-\alpha}{\alpha-1}(x)+\weight{-\alpha}{\alpha-1}(z)\bigr)\norm{\dpr{1}-\dpr{2}}_{\X{-\alpha}{\beta}}\label{eq:phi:diff:est:1}\\
  &\leq \eps\bigl(\weight{-\alpha}{\alpha-1}(x)+\weight{-\alpha}{\alpha-1}(z)\bigr)\norm{\dpr{1}-\dpr{2}}_{\X{-\alpha}{\beta}}.\label{eq:phi:diff:est:2}
 \end{align}
 We consider first the case $\alpha<1/2$. Precisely, together with \cref{eq:phi:diff:est:2,eq:weight:monotonicity} we deduce
 \begin{multline*}
  \int_{0}^{z}\weight{-a}{b}(x)\ee^{-x}\abs*{\ee^{\Phi[\dpr{1}](z)-\Phi[\dpr{1}](x)}-\ee^{\Phi[\dpr{2}](z)-\Phi[\dpr{2}](x)}}\dx\\*
  \lesssim \eps\biggl(\int_{0}^{z}\weight{-a-\alpha}{\alpha+b-1}(x)\ee^{-x}\dx+\int_{0}^{z}\weight{-a}{b}(x)\ee^{-x}\dx\weight{-\alpha}{\alpha-1}(z)\biggr)\norm{\dpr{1}-\dpr{2}}_{\X{-\alpha}{\beta}}.
 \end{multline*}
 By assumption, we have $-a-\alpha>-1$, and thus also $a>-1$, such that we may apply~\eqref{eq:prim:weight:1} to deduce
  \begin{multline*}
  \int_{0}^{z}\weight{-a}{b}(x)\ee^{-x}\abs*{\ee^{\Phi[\dpr{1}](z)-\Phi[\dpr{1}](x)}-\ee^{\Phi[\dpr{2}](z)-\Phi[\dpr{2}](x)}}\dx\\*
  \lesssim \eps\bigl(\weight{1-a-\alpha}{0}(z)+\weight{1-a-\alpha}{\alpha-1}(z)\bigr)\norm{\dpr{1}-\dpr{2}}_{\X{-\alpha}{\beta}}\lesssim \eps\weight{1-a-\alpha}{0}(z)\norm{\dpr{1}-\dpr{2}}_{\X{-\alpha}{\beta}}.
 \end{multline*}
 In the last step we also exploited~\eqref{eq:weight:monotonicity} and $\alpha-1<0$. This shows the first part of the lemma. 
 
 For the second part, i.e.\@ for $\alpha\geq 1/2$, we have to exploit the exponential decay of the profiles close to zero. In fact, proceeding as above but now with~\eqref{eq:phi:diff:est:1} we obtain
  \begin{multline*}
  \int_{0}^{z}\weight{-a}{b}(x)\ee^{-x}\abs*{\ee^{\Phi[\dpr{1}](z)-\Phi[\dpr{1}](x)}-\ee^{\Phi[\dpr{2}](z)-\Phi[\dpr{2}](x)}}\dx\\*
  \shoveleft{\lesssim \eps\biggl(\int_{0}^{z}\weight{-a-\alpha}{\alpha+b-1}(x)\ee^{-\min_{m\in\{1,2\}}\{\Phi[\dpr{m}](x)-\Phi[\dpr{m}](z)\}}\dx}\\*
  +\int_{0}^{z}\weight{-a}{b}(x)\ee^{-x}\ee^{-\min_{m\in\{1,2\}}\{\Phi[\dpr{m}](x)-\Phi[\dpr{m}](z)\}}\dx\weight{-\alpha}{\alpha-1}(z)\biggr)\norm{\dpr{1}-\dpr{2}}_{\X{-\alpha}{\beta}}.
 \end{multline*}
 Together with Lemma~\ref{Lem:regularising:effect} we then deduce
 \begin{multline*}
  \int_{0}^{z}\weight{-a}{b}(x)\ee^{-x}\abs*{\ee^{\Phi[\dpr{1}](z)-\Phi[\dpr{1}](x)}-\ee^{\Phi[\dpr{2}](z)-\Phi[\dpr{2}](x)}}\dx\\*
  \lesssim \eps^{\frac{\delta}{\alpha}}\bigl(\weight{1-a-\delta}{0}(z)+\weight{1-a-\delta}{\alpha-1}(z)\bigr)\norm{\dpr{1}-\dpr{2}}_{\X{-\alpha}{\beta}}\lesssim \eps^{\frac{\delta}{\alpha}}\weight{1-a-\delta}{0}(z)\norm{\dpr{1}-\dpr{2}}_{\X{-\alpha}{\beta}}.
 \end{multline*}
 In the last step, we again used that $\alpha<1$ together with~\eqref{eq:weight:monotonicity}. This finishes also the second part of the proof.
 \end{proof}

 \subsection{Proof of Proposition~\ref{Prop:bound:layer:est}}\label{Sec:actual:proof:bl}
 
 Relying on the results collected in the preceding subsection, we will now give the proof of Proposition~\ref{Prop:bound:layer:est}.
 
 \subsubsection{Estimating~\eqref{eq:bl:0}}
 
 We will estimate the different expressions appearing in~\eqref{eq:bl:0} separately. For this, we recall the following elementary bounds from~\cite{NTV15,Thr17a}. Precisely, the inequality $\abs{\ee^{-u}-\ee^{-v}}\lesssim \abs{u-v}$ for $u,v\geq 0$ together with Lemma~\ref{Lem:smallness:kappa} and the assumption $x\leq z$ allows to estimate
  \begin{equation*}
  \abs*{(x/z)^{\kappa_{1}}-(x/z)^{\kappa_{2}}}\lesssim (z/x)^{\nu_{\eps}}\abs{\kappa_1-\kappa_2}.
 \end{equation*}
 Similarly, we have for $x\leq z$ that 
 \begin{equation*}
  (x/z)^{\kappa_{k}}\lesssim (z/x)^{\nu_{\eps}}\qquad\text{and}\qquad \ee^{\Phi_{k}(z)-\Phi_{k}(x)}\leq 1\qquad \text{for }k=1,2.
 \end{equation*}
Finally, we note that
\begin{equation*}\label{eq:exp:reg:effect}
 \frac{1-\ee^{-z}}{z}\leq 1 \qquad \text{for all }z\geq 0.
\end{equation*}
Then, taking the absolute value, and using additionally that $y/z\leq 1$ as well as $(z-y)/z\leq 1$ for $y\leq z$ we can estimate from~\eqref{eq:bl:0} that
\begin{multline*}
 \abs*{\dpr{1}-\dpr{2}}\leq \eps\sum_{k=1}^{4} \abs*{\K_{k}[\dpr{1},\dpr{2}]}+\sum_{k=1}^{3}\abs*{\J_{k}[\dpr{1},\dpr{2}]}\\*
 \lesssim \eps \int_{x}^{\infty}(z/x)^{\nu_{\eps}}\abs{\beta_{W}[\dpr{1}](z)}\dpr{1}(z)\dz\abs{\kappa_1-\kappa_2}\\*
 +\eps \int_{x}^{\infty}(z/x)^{\nu_{\eps}} \abs*{\ee^{\Phi_{1}(z)-\Phi_{1}(x)}-\ee^{\Phi_{2}(z)-\Phi_{2}(x)}}\abs{\beta_{W}[\dpr{1}](z)}\dpr{1}(z)\dz\\*
 +\eps \int_{x}^{\infty}(z/x)^{\nu_{\eps}}\abs{\beta_{W}[\dpr{1}](z)-\beta_{W}[\dpr{2}](z)}\dpr{1}(z)\dz\\*
  +\eps \int_{x}^{\infty}\int_{0}^{z}(z/x)^{\nu_{\eps}}\abs{\beta_{W}[\dpr{2}](z)}\abs{\dpr{1}(z)-\dpr{2}(z)}\dz\\*
  +\int_{x}^{\infty}\int_{0}^{z}(z/x)^{\nu_{\eps}}\frac{1}{z}K_{\eps}(y,z-y)\dpr{1}(y)\dpr{1}(z-y)\dy\dz\abs{\kappa_1-\kappa_2}\\*
 +\int_{x}^{\infty}\int_{0}^{z}(z/x)^{\nu_{\eps}}\frac{1}{z}\abs*{\ee^{\Phi_{1}(z)-\Phi_{1}(x)}-\ee^{\Phi_{2}(z)-\Phi_{2}(x)}}K_{\eps}(y,z-y)\dpr{1}(y)\dpr{1}(z-y)\dy\dz\\*
 +\int_{x}^{\infty}\int_{0}^{z}(z/x)^{\nu_{\eps}}\frac{1}{z}K_{\eps}(y,z-y)\abs*{\dpr{1}(y)-\dpr{2}(y)}\Bigl(\dpr{1}(z-y)+\dpr{2}(z-y)\Bigr)\dy\dz.
\end{multline*}

Based on this inequality, we will now separately estimate the terms $\K_{k}$ and $\J_{k}$. For this, we will also exploit Fubini's theorem in the forms
\begin{multline}\label{eq:proof:bl:Fubini}
 \int_{0}^{\infty}\int_{x}^{\infty}(\cdots)\dz\dx=\int_{0}^{\infty}\int_{0}^{z}(\cdots)\dx\dz\qquad \text{and}\\*
 \int_{0}^{\infty}\int_{x}^{\infty}\int_{0}^{z}F(x,y,z)\dy\dz\dx=\int_{0}^{\infty}\int_{0}^{\infty}\int_{0}^{y+z}F(x,y,z+y)\dx\dz\dy.
\end{multline}

\subsubsection{Estimate of $\K_{1}$}

For $\eps>0$ sufficiently small such that $\nu_{\eps}<1-\alpha$ we get together with \cref{Lem:est:beta:W,Lem:moments,eq:prim:weight:1,eq:weight:shift,eq:weight:additivity} that 
\begin{multline*}
 \norm{\ee^{-\cdot}\K_{1}[\dpr{1},\dpr{2}]}_{\X{-\alpha}{\beta}}\lesssim \int_{0}^{\infty}z^{\nu_{\eps}}\weight{-\alpha}{\alpha}(z)\dpr{1}(z)\int_{0}^{z}\weight{-\alpha-\nu_{\eps}}{\beta-\nu_{\eps}}(x)\ee^{-x}\dx\dz\abs{\kappa_1-\kappa_2}\\*
 \lesssim \int_{0}^{\infty}\weight{1-2\alpha}{\alpha+\nu_{\eps}}(z)\dpr{1}(z)\dz\abs{\kappa_1-\kappa_2}\lesssim \abs{\kappa_1-\kappa_2}.
\end{multline*}
Thus, taking also Remark~\ref{Rem:diff:kappa} into account, we conclude
\begin{equation}\label{eq:proof:bl:K1}
 \eps\norm{\ee^{-\cdot}\K_{1}[\dpr{1},\dpr{2}]}_{\X{-\alpha}{\beta}}\lesssim \eps \norm{\dpr{1}-\dpr{2}}_{\X{-\alpha}{\beta}}.
\end{equation}

\subsubsection{Estimate of $\K_{2}$}

For this term, we have to distinguish whether $\alpha<1/2$ or $\alpha\in [1/2,1)$. In the first case, we take $\eps>0$ small such that $\nu_{\eps}<1-2\alpha$ which is possible since $\alpha<1/2$. Then, relying on \cref{Lem:reg:ef:diff:Phi,Lem:est:beta:W,eq:weight:additivity,eq:prim:weight:1} we obtain
\begin{multline*}
 \norm{\ee^{-\cdot}\K_{2}[\dpr{1},\dpr{2}]}_{\X{-\alpha}{\beta}}\\*
 \lesssim \int_{0}^{\infty}z^{\nu_{\eps}}\weight{-\alpha}{\alpha}(z)\dpr{1}(z)\int_{0}^{z}\weight{-\alpha-\nu_{\eps}}{\beta}(x)\ee^{-x}\abs*{\ee^{\Phi_{1}(z)-\Phi_{1}(x)}-\ee^{\Phi_{2}(z)-\Phi_{2}(x)}}\dx\dz\\*
 \lesssim \eps\int_{0}^{\infty}\weight{1-3\alpha}{0}(z)\dpr{1}(z)\dz\norm{\dpr{1}-\dpr{2}}_{\X{-\alpha}{\beta}}.
\end{multline*}
Note that the assumption $\alpha<1/2$ was used for the application of Lemma~\ref{Lem:reg:ef:diff:Phi}. Moreover, it also guarantees that $1-3\alpha>-1$ and thus Lemma~\ref{Lem:moments} allows to estimate the last integral on the right-hand side to get
\begin{equation}\label{eq:proof:bl:K2:small}
 \eps\norm{\ee^{-\cdot}\K_{2}[\dpr{1},\dpr{2}]}_{\X{-\alpha}{\beta}}\lesssim \eps \norm{\dpr{1}-\dpr{2}}_{\X{-\alpha}{\beta}}.
\end{equation}

If $\alpha\geq 1/2$, we proceed analogously exploiting the second estimate in Lemma~\ref{Lem:reg:ef:diff:Phi}. In fact, choosing $\delta\in(0,1-\alpha)$ and $\eps>0$ sufficiently small such that $\nu_{\eps}<1-\alpha-\delta$ we find
\begin{multline*}
 \norm{\ee^{-\cdot}\K_{2}[\dpr{1},\dpr{2}]}_{\X{-\alpha}{\beta}}\\*
 \lesssim \int_{0}^{\infty}z^{\nu_{\eps}}\weight{-\alpha}{\alpha}(z)\dpr{1}(z)\int_{0}^{z}\weight{-\alpha-\nu_{\eps}}{\beta}(x)\ee^{-x}\abs*{\ee^{\Phi_{1}(z)-\Phi_{1}(x)}-\ee^{\Phi_{2}(z)-\Phi_{2}(x)}}\dx\dz\\*
 \lesssim \eps^{\frac{\delta}{\alpha}}\int_{0}^{\infty}\weight{1-2\alpha-\delta}{0}(z)\dpr{1}(z)\dz\norm{\dpr{1}-\dpr{2}}_{\X{-\alpha}{\beta}}.
\end{multline*}
From the choice of $\delta$ we have $1-2\alpha-\delta>-\alpha>-1$ and thus Lemma~\ref{Lem:moments} allows to conclude
\begin{equation}\label{eq:proof:bl:K2:large}
 \eps\norm{\ee^{-\cdot}\K_{2}[\dpr{1},\dpr{2}]}_{\X{-\alpha}{\beta}}\lesssim \eps^{1+\frac{\delta}{\alpha}} \norm{\dpr{1}-\dpr{2}}_{\X{-\alpha}{\beta}}.
\end{equation}

\subsubsection{Estimate of $\K_{3}$}

To estimate $\K_{3}$, we take $\eps>0$ small such that $\nu_{\eps}<1-\alpha$ and recall \cref{Lem:est:diff:beta:W,eq:weight:shift,eq:weight:additivity,eq:prim:weight:1} to get
\begin{multline*}
 \norm{\ee^{-\cdot}\K_{3}[\dpr{1},\dpr{2}]}_{\X{-\alpha}{\beta}}\\*
 \lesssim \int_{0}^{\infty}z^{\nu_{\eps}}\weight{-\alpha}{\alpha}(z)\dpr{1}(z)\int_{0}^{z}\weight{-\alpha-\nu_{\eps}}{\beta-\nu_{\eps}}(x)\ee^{-x}\dx\dz\norm{\dpr{1}-\dpr{2}}_{\X{-\alpha}{\beta}}\\*
 \lesssim \int_{0}^{\infty}\weight{1-2\alpha}{\nu_{\eps}+\alpha}(z)\dpr{1}(z)\dz\norm{\dpr{1}-\dpr{2}}_{\X{-\alpha}{\beta}}.
\end{multline*}
Since $1-2\alpha>-1$ for all $\alpha\in(0,1)$, we can apply Lemma~\ref{Lem:moments} to bound the last integral on the right-hand side provided $\eps>0$ is sufficiently small. With this, we get
\begin{equation}\label{eq:proof:bl:K3}
 \eps\norm{\ee^{-\cdot}\K_{3}[\dpr{1},\dpr{2}]}_{\X{-\alpha}{\beta}}\lesssim \eps \norm{\dpr{1}-\dpr{2}}_{\X{-\alpha}{\beta}}.
\end{equation}

\subsubsection{Estimate of $\K_{4}$}

We choose $\eps>0$ small such that $\nu_{\eps}<1-\alpha$ and use \cref{Lem:est:beta:W,eq:weight:shift,eq:prim:weight:1,eq:weight:shift} to estimate
\begin{multline*}
 \norm{\ee^{-\cdot}\K_{4}[\dpr{1},\dpr{2}]}_{\X{-\alpha}{\beta}}\lesssim \int_{0}^{\infty}z^{\nu_{\eps}}\weight{-\alpha}{\alpha}(z)\abs{\dpr{1}(z)-\dpr{2}(z)}\int_{0}^{z}\weight{-\alpha-\nu_{\eps}}{\beta-\nu_{\eps}}(x)\ee^{-x}\dx\dz\\*
 \lesssim \int_{0}^{\infty}\weight{1-2\alpha}{\alpha+\nu_{\eps}}(z)\abs{\dpr{1}(z)-\dpr{2}(z)}\dz.
\end{multline*}
Since $\alpha<1$ we have $1-2\alpha>-\alpha$ and from the choice of $\nu_{\eps}$ we deduce $\alpha+\nu_{\eps}<1<\beta$. Thus, together with \cref{eq:weight:monotonicity,Lem:moments} we obtain for sufficiently small $\eps>0$ that
\begin{equation}\label{eq:proof:bl:K4}
 \eps\norm*{\ee^{-\cdot}\K_{3}[\dpr{1},\dpr{2}]}\lesssim \eps\int_{0}^{\infty}\weight{-\alpha}{\beta}(z)\abs{\dpr{1}(z)-\dpr{2}(z)}\dz\lesssim \eps \norm{\dpr{1}-\dpr{2}}_{\X{-\alpha}{\beta}}.
\end{equation}
 
\subsubsection{Estimate of $\J_{1}$}

For $\eps>0$ sufficiently small such that $\nu_{\eps}<1-\alpha$ we deduce together with \cref{eq:proof:bl:Fubini,eq:prim:weight:1,eq:weight:shift} that
\begin{multline}\label{eq:intermediate:1}
 \norm{\ee^{-\cdot}\J_{1}[\dpr{1},\dpr{2}]}_{\X{-\alpha}{\beta}}\\*
 \lesssim \int_{0}^{\infty}\int_{0}^{\infty}(y+z)^{\nu_{\eps}-1}K_{\eps}(y,z)\dpr{1}(y)\dpr{1}(z)\int_{0}^{y+z}\weight{-\alpha-\nu_{\eps}}{\beta-\nu_{\eps}}(x)\ee^{-x}\dx\dz\dy\abs{\kappa_{2}-\kappa_{2}}\\*
 \lesssim \int_{0}^{\infty}\int_{0}^{\infty}\weight{-\alpha}{\nu_{\eps}-1}(y+z)K_{\eps}(y,z)\dpr{1}(y)\dpr{1}(z)\dz\dy\abs{\kappa_{2}-\kappa_{2}}.
\end{multline}
For sufficiently small $\eps>0$ the integral on the right-hand side can be estimated by a uniform constant due to \cref{Lem:moments,Lem:aux:int:bl:est}. Thus, together with Lemma~\ref{Lem:est:kappa} we finally get for any $\mu>0$ that
\begin{equation}\label{eq:proof:bl:J1}
  \norm{\ee^{-\cdot}\J_{1}[\dpr{1},\dpr{2}]}_{\X{-\alpha}{\beta}}\lesssim \mu\norm{\dpr{1}-\dpr{2}}_{\X{-\alpha}{\beta}}+C_{\mu}\norm*{\bigl(1-\exp(-\cdot)\bigr)(\dpr{1}-\dpr{2})}_{\X{-\alpha}{\beta}}.
\end{equation}

\subsubsection{Estimate of $\J_{2}$}

To estimate $\J_{2}$ we first note that for all $\alpha\in(0,1)$ we have

\begin{multline}\label{eq:intermediate:2}
 \norm{\ee^{-\cdot}\J_{2}[\dpr{1},\dpr{2}]}_{\X{-\alpha}{\beta}}\\*
 \shoveleft{\lesssim \int_{0}^{\infty}\int_{0}^{\infty}(y+z)^{\nu_{\eps}-1}K_{\eps}(y,z)\dpr{1}(y)\dpr{1}(z)\times}\\*
 \times\int_{0}^{y+z}\weight{-\alpha-\nu_{\eps}}{\beta-\nu_{\eps}}(x)\ee^{-x}\abs*{\ee^{\Phi_{1}(y+z)-\Phi_{1}(x)}-\ee^{\Phi_{2}(y+z)-\Phi_{2}(x)}}\dx\dz\dy.
\end{multline}
Now we need to distinguish again whether $\alpha\in(0,1/2)$ or $\alpha\in[1/2,1)$. In the first case we choose $\eps>0$ sufficiently small such that $\nu_{\eps}<1-2\alpha$ which is possible for $\alpha<1/2$. Then we get by means of \cref{Lem:reg:ef:diff:Phi,eq:weight:shift} that
\begin{multline*}
 \norm{\ee^{-\cdot}\J_{2}[\dpr{1},\dpr{2}]}_{\X{-\alpha}{\beta}}\\*
 \lesssim \eps \int_{0}^{\infty}\int_{0}^{\infty}\weight{-2\alpha}{\nu_{\eps}-1}(y+z)K_{\eps}(y,z)\dpr{1}(y)\dpr{1}(z)\dz\dy\norm{\dpr{1}-\dpr{2}}_{\X{-\alpha}{\beta}}.
\end{multline*}
Together with \cref{Lem:moments,Lem:aux:int:bl:est} this can be bounded as
\begin{equation}\label{eq:proof:bl:J2:small}
 \norm{\ee^{-\cdot}\J_{2}[\dpr{1},\dpr{2}]}_{\X{-\alpha}{\beta}}\lesssim \eps\norm{\dpr{1}-\dpr{2}}_{\X{-\alpha}{\beta}}.
\end{equation}
If $\alpha\geq 1/2$, we proceed similarly, i.e.\@ we fix $\delta<1-\alpha$ and take $\eps>0$ sufficiently small such that $\nu_{\eps}<1-\alpha-\delta$. Then, \cref{Lem:reg:ef:diff:Phi,eq:weight:shift} imply
\begin{multline*}
 \norm{\ee^{-\cdot}\J_{2}[\dpr{1},\dpr{2}]}_{\X{-\alpha}{\beta}}\\*
 \lesssim \eps^{\frac{\delta}{\alpha}}\int_{0}^{\infty}\int_{0}^{\infty}\weight{-\alpha-\delta}{\nu_{\eps}-1}(y+z)K_{\eps}(y,z)\dpr{1}(y)\dpr{1}(z)\dz\dy\norm{\dpr{1}-\dpr{2}}_{\X{-\alpha}{\beta}}.
\end{multline*}
Finally, since for $\alpha\geq 1/2$ our choice of $\delta$ in particular ensures that $\alpha+\delta\leq 2\alpha$, \cref{Lem:moments,Lem:aux:int:bl:est} as before imply
\begin{equation}\label{eq:proof:bl:J2:large}
 \norm{\ee^{-\cdot}\J_{2}[\dpr{1},\dpr{2}]}_{\X{-\alpha}{\beta}}\lesssim \eps^{\frac{\delta}{\alpha}}\norm{\dpr{1}-\dpr{2}}_{\X{-\alpha}{\beta}}.
\end{equation}

\subsubsection{Estimate of $\J_{3}$}

We take $\eps>0$ small such that $\nu_{\eps}<1-\alpha$ and recall \cref{eq:proof:bl:Fubini,eq:weight:shift,eq:prim:weight:1} to get
\begin{multline*}
 \norm{\ee^{-\cdot}\J_{2}[\dpr{1},\dpr{2}]}_{\X{-\alpha}{\beta}}\\*
 \lesssim \int_{0}^{\infty}\int_{0}^{\infty}\frac{K_{\eps}(y,z)}{(y+z)^{1-\nu_{\eps}}}\abs{\dpr{1}(y)-\dpr{2}(y)}\bigl(\dpr{1}(z)+\dpr{2}(z)\bigr)\int_{0}^{y+z}\weight{-\alpha-\nu_{\eps}}{\beta-\nu_{\eps}}(x)\ee^{-x}\dx\dz\dy\\*
 \lesssim \int_{0}^{\infty}\int_{0}^{\infty}\weight{-\alpha}{\nu_{\eps}-1}(y+z)K_{\eps}(y,z)\abs{\dpr{1}(y)-\dpr{2}(y)}\bigl(\dpr{1}(z)+\dpr{2}(z)\bigr)\dz\dy.
\end{multline*}
Recalling from~\eqref{eq:Ass:K2} that $K_{\eps}=2+\eps W$, we have
\begin{multline*}
 \norm{\ee^{-\cdot}\J_{2}[\dpr{1},\dpr{2}]}_{\X{-\alpha}{\beta}}\lesssim \int_{0}^{\infty}\int_{0}^{\infty}\weight{-\alpha}{\nu_{\eps}-1}(y+z)\abs{\dpr{1}(y)-\dpr{2}(y)}\bigl(\dpr{1}(z)+\dpr{2}(z)\bigr)\dz\dy\\*
 +\eps\int_{0}^{\infty}\int_{0}^{\infty}\weight{-\alpha}{\nu_{\eps}-1}(y+z)W(y,z)\abs{\dpr{1}(y)-\dpr{2}(y)}\bigl(\dpr{1}(z)+\dpr{2}(z)\bigr)\dz\dy.
\end{multline*}
Since $\nu_{\eps}-1<0$ we can easily estimate together with~\eqref{eq:weight:monotonicity} that $\weight{-\alpha}{\nu_{\eps}-1}(y+z)\lesssim \weight{-\alpha}{0}(y+z)\lesssim \weight{-\alpha}{0}(z)$ which further implies
\begin{multline*}
 \norm{\ee^{-\cdot}\J_{2}[\dpr{1},\dpr{2}]}_{\X{-\alpha}{\beta}}\lesssim \int_{0}^{\infty}\abs{\dpr{1}(y)-\dpr{2}(y)}\dy\int_{0}^{\infty}\weight{-\alpha}{0}(z)\bigl(\dpr{1}(z)+\dpr{2}(z)\bigr)\dz\\*
 +\eps\int_{0}^{\infty}\int_{0}^{\infty}\weight{-\alpha}{\nu_{\eps}-1}(y+z)W(y,z)\abs{\dpr{1}(y)-\dpr{2}(y)}\bigl(\dpr{1}(z)+\dpr{2}(z)\bigr)\dz\dy.
\end{multline*}
Thus, together with \cref{Lem:est:kappa,Lem:moments,Lem:aux:int:bl:est} we obtain for all $\mu>0$ and sufficiently small $\eps>0$ that
\begin{equation}\label{eq:proof:bl:J3}
 \norm{\ee^{-\cdot}\J_{2}[\dpr{1},\dpr{2}]}_{\X{-\alpha}{\beta}}\lesssim (\mu+\eps)\norm{\dpr{1}-\dpr{2}}_{\X{-\alpha}{\beta}}+C_{\mu}\norm*{\bigl(1-\exp(-\cdot)\bigr)(\dpr{1}-\dpr{2})}_{\X{-\alpha}{\beta}}.
\end{equation}

\subsubsection{End of the proof}

Thus, in the case of $\alpha<1/2$ combining~\cref{eq:proof:bl:K1,eq:proof:bl:K2:small,eq:proof:bl:K3,eq:proof:bl:K4,eq:proof:bl:J1,eq:proof:bl:J2:small,eq:proof:bl:J3} or in the case of $\alpha\geq 1/2$, combining~\cref{eq:proof:bl:K1,eq:proof:bl:K2:large,eq:proof:bl:K3,eq:proof:bl:K4,eq:proof:bl:J1,eq:proof:bl:J2:large,eq:proof:bl:J3} the stated estimate easily follows by re-choosing $\mu>0$ appropriately and taking $\eps>0$ sufficiently small.

 \section{Proof of Proposition~\ref{Prop:continuity:L:inverse}}\label{Sec:proof:inversion}
 
 In this section, we will give the proof of Proposition~\ref{Prop:continuity:L:inverse}. This is in principle straightforward but the corresponding calculations are relatively lengthy. More precisely, we define the operators
 \begin{equation}\label{eq:def:pre:inverse}
  \A[g](x)\vcc=g(x)+2\hs{1}(x)\int_{1}^{x}\biggl(\frac{\ee^{y}}{y}-\int_{1}^{y}\frac{\ee^{z}}{z}\dz\biggr)g(y)\dy  -2\hs{2}(x)\int_{x}^{\infty}g(y)\dy
 \end{equation}
 and 
 \begin{equation}\label{eq:def:inverse}
  \A_{0}[g]\vcc=\A[g]+\int_{0}^{\infty}y\A[g](y)\dy \hs{1}
 \end{equation}
with 
\begin{equation}\label{eq:m1:m2}
 \hs{1}(x)=(1-x)\ee^{-x}\qquad \text{and}\qquad \hs{2}(x)=1+(1-x)\ee^{-x}\int_{1}^{x}\frac{\ee^{z}}{z}\dz.
\end{equation}
We then claim, that the inverse $\LL^{-1}$ of $\LL$ is given by $\A_{0}$. More precisely, the strategy is as follows:
\begin{enumerate}
 \item In Section~\ref{Sec:rew:lin:op} we will derive several equivalent formulas for $\LL$ as a preparatory step.
 \item In Section~\ref{Sec:inv:cont} we will show that $\A_{0}$ is well-defined on $\X{a}{b}$ and maps continuously to $\X[0]{a}{b}$ for each $a\in(-1,1)$ and $b>1$.
 \item In Section~\ref{Sec:lin:inj} we will prove that $\ker \LL\cap \X[0]{a}{b}=\{0\}$, i.e.\@ that $\LL$ is injective on $\X[0]{a}{b}=\{0\}$.
 \item Section~\ref{Sec:lin:surj} is devoted to showing that $\LL\circ \A_{0}=\id$ on $\X{a}{b}$, i.e.\@ that $\LL$ maps surjectively to $\X{a}{b}$. 
\end{enumerate}
Together, this then yields that $\LL$ is invertible on $\X[0]{a}{b}$ with inverse $\LL^{-1}=\A_{0}$. A derivation of~\eqref{eq:def:pre:inverse} or~\eqref{eq:def:inverse} can be found in \cref{Sec:inverse:derivation}.

 \subsection{Rewriting the linearised operator}\label{Sec:rew:lin:op}
 
 As a first step, we give some alternative representations of $\LL$ which will simplify some of the computations below. In fact, we recall from \cref{eq:linearised:operator:abstract,eq:def:bilinear:forms} that
 \begin{equation}\label{eq:linearised:operator:0}
  \begin{split}
   \LL[h](x)&=h(x)-\frac{2}{x^2}\int_{0}^{x}\int_{x-y}^{\infty}y\ee^{-y}h(z)+yh(y)\ee^{-z}\dz\dy\\
   &=h(x)-\frac{2}{x^2}\int_{0}^{x}\int_{x-y}^{\infty}y\ee^{-y}h(z)\dz\dy-\frac{2}{x^2}\int_{0}^{x}\int_{x-y}^{\infty}yh(y)\ee^{-z}\dz\dy.
  \end{split}
 \end{equation}
 Now, we split the first integral on the right-hand side and apply Fubini's theorem in the form $\int_{0}^{x}\int_{x-y}^{\infty}(\cdots)\dz\dy=\int_{0}^{x}\int_{x-z}^{x}(\cdots)\dy\dz+\int_{x}^{\infty}\int_{0}^{x}(\cdots)\dy\dz$ while in the second term we just compute the $z$-integral. Together this yields
 \begin{multline*}
  \LL[h](x)=h(x)-\frac{2}{x^2}\int_{0}^{x}\int_{x-z}^{x}y\ee^{-y}\dy h(z)\dz\\*
  -\frac{2}{x^2}\int_{0}^{x}y\ee^{-y}\dy\int_{x}^{\infty}h(z)\dz-\frac{2\ee^{-x}}{x^2}\int_{0}^{x}y\ee^{y}h(y)\dy.
 \end{multline*}
 Computing now the first two integrals in $y$ on the right-hand side, we find
  \begin{multline*}
  \LL[h](x)=h(x)+\frac{2(x+1)\ee^{-x}}{x^2}\int_{0}^{x}h(z)\dz-\frac{2\ee^{-x}}{x}\int_{0}^{x}\ee^{z}h(z)\dz+\frac{2\ee^{-x}}{x^2}\int_{0}^{x}(z-1)\ee^{z}h(z)\dz\\*
  +\frac{2(x+1)\ee^{-x}-2}{x^2}\int_{x}^{\infty}h(z)\dz-\frac{2\ee^{-x}}{x^2}\int_{0}^{x}y\ee^{y}h(y)\dy.
 \end{multline*}
 Summarising these expressions, we finally get
 \begin{equation}\label{eq:L:2}
  \LL[h](x)=h(x)+\frac{2(x+1)\ee^{-x}}{x^2}\int_{0}^{\infty}h(z)\dz-\frac{2(x+1)\ee^{-x}}{x^2}\int_{0}^{x}\ee^{z}h(z)\dz-\frac{2}{x^2}\int_{x}^{\infty}h(z)\dz.
 \end{equation}
 Moreover, we derive a further representation of $\LL$ which will allow to exploit some cancellation at zero. Precisely, splitting the integral $\int_{0}^{x}h(z)\dz=\int_{0}^{x}h(z)\dz+\int_{x}^{\infty}h(z)\dz$ we find
 \begin{equation}\label{eq:L:4}
  \LL[h](x)=h(x)+\frac{2(x+1)\ee^{-x}}{x^2}\int_{0}^{x}(1-\ee^{z})h(z)\dz+\frac{2(x\ee^{-x}+\ee^{-x}-1)}{x^2}\int_{x}^{\infty}h(z)\dz.
 \end{equation}
 
\subsection{Continuity of $\LL^{-1}$}\label{Sec:inv:cont}
 
 Recalling from~\eqref{eq:mass:m1} that $\int_{0}^{\infty}x\hs{1}(x)\dx=-1$, we notice that $\A_{0}$ has explicitly been constructed such that $\int_{0}^{\infty}x\A_{0}[g](x)\dx=0$ whenever the first moment exists. Thus, to prove that $\A_{0}\colon \X{a}{b}\to \X[0]{a}{b}$ is well-defined and continuous for $a\in(-1,1)$ and $b>1$, it suffices to show the continuity.
 Recalling Remark~\ref{Rem:explicit:profile} and~\eqref{eq:weight:monotonicity} together with $a<1$ and $b>1$ we find
 \begin{equation*}
  \norm*{\int_{0}^{\infty}y\A[g](y)\dy \hs{1}}_{\X{a}{b}}\lesssim \int_{0}^{\infty}y\abs*{\A[g](y)}\dy\lesssim \int_{0}^{\infty}\abs*{\A[g](y)}\weight{a}{b}(y)\dy=\norm*{\A[g]}_{\X{a}{b}}.
 \end{equation*}
Thus, it suffices to show that $\norm*{\A[g]}_{\X{a}{b}}\lesssim\norm*{g}_{\X{a}{b}}$. Due to the explicit form of $\A[g]$ given in~\eqref{eq:def:pre:inverse} we observe, that it is in fact enough to consider only the expressions
\begin{equation}\label{eq:cont:LL:inv:0}
 \hs{1}(x)\int_{1}^{x}\biggl(\frac{\ee^{y}}{y}-\int_{1}^{y}\frac{\ee^{z}}{z}\dz\biggr)g(y)\dy\qquad \text{and}\qquad \hs{2}(x)\int_{x}^{\infty}g(y)\dy
\end{equation}
which we will study separately. Starting with the second one, the definition of the norm and Fubini's theorem imply
\begin{equation*}
 \norm*{\hs{2}\int_{\cdot}^{\infty}g(y)\dy}_{\X{a}{b}}=\int_{0}^{\infty}\abs{g(y)}\int_{0}^{y}\abs{\hs{2}(x)}\weight{a}{b}(x)\dx\dy.
\end{equation*}
Recalling~\eqref{eq:prim:m2:est:2} together with the assumption $b>1$ further yields
\begin{equation}\label{eq:cont:LL:inv:1}
 \norm*{\hs{2}\int_{\cdot}^{\infty}g(y)\dy}_{\X{a}{b}}\lesssim \int_{0}^{\infty}\abs{g(y)}\weight{a}{b-1}(y)\dy=\norm{g}_{\X{a}{b-1}}.
\end{equation}
Similarly, for the first expression in~\eqref{eq:cont:LL:inv:0} we use the definition of $\norm{\cdot}_{\X{a}{b}}$, split the integral $\int_{0}^{\infty}(\cdots)\dx=\int_{0}^{1}(\cdots)\dx+\int_{1}^{\infty}(\cdots)\dx$ and apply Fubini's theorem which gives
\begin{multline*}
 \norm*{\hs{1}\int_{1}^{\cdot}\biggl(\frac{\ee^{y}}{y}-\int_{1}^{y}\frac{\ee^{z}}{z}\dz\biggr)g(y)\dy}_{\X{a}{b}}\\*
 \leq \int_{0}^{1}\abs*{\frac{\ee^{y}}{y}-\int_{1}^{y}\frac{\ee^{z}}{z}\dz}\abs{g(y)}\int_{0}^{y}\abs{\hs{1}(x)}\weight{a}{b}(x)\dx\dy\\*
 +\int_{1}^{\infty}\abs*{\frac{\ee^{y}}{y}-\int_{1}^{y}\frac{\ee^{z}}{z}\dz}\abs{g(y)}\int_{y}^{\infty}\abs{\hs{1}(x)}\weight{a}{b}(x)\dx\dy.
\end{multline*}
Recalling \cref{eq:prim:m1:est:1,eq:prim:m1:est:2,eq:est:aux:expr} we can further estimate
\begin{multline}\label{eq:cont:LL:inv:2}
 \norm*{\hs{1}\int_{1}^{\cdot}\biggl(\frac{\ee^{y}}{y}-\int_{1}^{y}\frac{\ee^{z}}{z}\dz\biggr)g(y)\dy}_{\X{a}{b}}\lesssim \int_{0}^{1}y^{-1}\abs{g(y)}y^{1+a}\dy+\int_{1}^{\infty}\frac{\ee^{y}}{y^2}\abs{g(y)}y^{b+1}\ee^{-y}\dy\\*
 =\int_{0}^{1}y^{a}\abs{g(y)}\dy+\int_{1}^{\infty}y^{b-1}\abs{g(y)}\dy=\norm{g}_{\X{a}{b-1}}.
\end{multline}
Thus, combining \cref{eq:cont:LL:inv:1,eq:cont:LL:inv:2} the claimed continuity follows together with~\eqref{eq:spaces:embedding}.

\subsection{$\LL$ is injective}\label{Sec:lin:inj}

In this subsection, we will show that $\ker \LL\cap \X[0]{a}{b}=\{0\}$, i.e.\@ b$\LL$ is injective on $\X[0]{a}{b}$. Alternatively, the latter result could also be obtained by explicitly evaluating the composition $\A_{0}\circ \LL$. In fact, a calculation similar to that one in Section~\ref{Sec:lin:surj} below shows that this expression equals $\id$ on $\X[0]{a}{b}$. However, since this leads to very long formulas, we will follow another shorter approach here. In fact, we prove that $\ker\LL=\{C\hs{1}\;|\; C\in \R\}$ using the (desingularised) Laplace transform. More precisely, one immediately checks that $\LL[\hs{1}]=0$, i.e.\@ $\hs{1} \in\ker\LL$. In fact using the expression~\eqref{eq:L:2} and noting that $\int_{0}^{\infty}\hs{1}(x)\dx=0$ as well as $\hs{1}(x)=(x\ee^{-x})'$ we find
\begin{multline}\label{eq:m1:kernel}
 \LL[\hs{1}](x)=(1-x)\ee^{-x}-\frac{2(x+1)\ee^{-x}}{x^2}\int_{0}^{x}(1-z)\dz-\frac{2}{x^2}\int_{x}^{\infty}(x\ee^{-x})'\dz\\*
 =(1-x)\ee^{-x}-\frac{(2-x)(x+1)\ee^{-x}}{x}+2\frac{\ee^{-x}}{x}=0.
\end{multline}
To see that $\dim (\ker\LL)=1$, let $g\in \X{a}{b}$ solve $\LL[g]=0$. Then we can define the desingularised Laplace transform $\T[g](q)\vcc=\int_{0}^{\infty}(1-\ee^{-qx})g(x)\dx$. It is well-known that $\T[g]$ is differentiable for all $q\in (0,\infty)$ and from~\cite[eq.\@ (27)]{NiV14a} one deduces that $\T[g]$ has to satisfy the ODE
\begin{equation}\label{eq:ODE:kernel:Laplace}
 \frac{\dd}{\dd{q}}\T[g](q)+\frac{q-1}{q(q+1)}\T[g]=0.
\end{equation}
One now easily checks that the general solution to~\eqref{eq:ODE:kernel:Laplace} is given by $\T[g](q)=C\frac{q}{(q+1)^2}$ parametrised by the constant $C$. On the other hand, one immediately verifies that $\T[g](q)$ is exactly the desingularised Laplace transform of the function $-C\hs{1}$. Since a function in $\X{a}{b}$ with $a\in(-1,1)$ and $b>1$ (or more generally a finite measure) is uniquely determined by its desingularised Laplace transform (see \cite{SSV10}) this shows that the operator $\LL$ has a one-dimensional kernel which is spanned by the function $\hs{1}$. On the other hand, $\int_{0}^{\infty}x\hs{1}(x)\dx=-1$ which thus shows that $\ker \LL\cap \X[0]{a}{b}=\{0\}$. Consequently, $\LL$ is injective on $\X[0]{a}{b}$.

\subsection{$\LL$ is surjective}\label{Sec:lin:surj}

To conclude the proof of Proposition~\ref{Prop:continuity:L:inverse} it only remains to show that $\LL$ maps $\X[0]{a}{b}$ onto $\X{a}{b}$. We will prove this, by verifying that $\LL\circ \A_{0}=\id$ on $\X{a}{b}$. Moreover, since during the following calculations we frequently integrate by parts, we note that it suffices to verify this relation on the dense subspace $C_{\text{c}}^{\infty}(0,\infty)\subset \X{a}{b}$ of smooth functions which are compactly supported on $(0,\infty)$. In fact, due to the continuity of $\LL$ and $\A_{0}$ provided by \cref{Prop:continuity:L:NEW,Sec:inv:cont} the equality on $\X{a}{b}$ then follows by approximation.

To simplify the notation let us introduce the abbreviation 
\begin{equation}\label{eq:def:E}
 E(x)\vcc=\frac{\ee^{x}}{x}-\int_{1}^{x}\frac{\ee^{z}}{z}\dz.
\end{equation}
Moreover, we note that $\hs{1}(x)=(x\ee^{-x})'$ and we recall $\LL[\hs{1}]=0$ from Section~\ref{Sec:lin:inj}. Then, for $g\in C_{\text{c}}^{\infty}(0,\infty)$ given we obtain with~\eqref{eq:L:4} that
\begin{multline*}
 \bigl(\LL\circ \A_{0}[g]\bigr)(x)=\A[g]+\frac{2(x+1)\ee^{-x}}{x^2}\int_{0}^{x}(1-\ee^{z})g(z)\dz+\frac{2(x+1)\ee^{-x}-2}{x^2}\int_{x}^{\infty}g(z)\dz\\*
 \shoveleft{+\frac{4(x+1)\ee^{-x}}{x^2}\biggl[\int_{0}^{x}\del_{\xi}\biggl(\int_{0}^{\xi}(1-\ee^{\eta})\hs{1}(\eta)\deta\biggr)\int_{1}^{\xi}E(y)g(y)\dy\dxi}\\*
 \shoveright{-\int_{0}^{x}\del_{\xi}\biggl(\int_{0}^{\xi}(1-\ee^{\eta})\hs{2}(\eta)\deta\biggr)\int_{\xi}^{\infty}g(y)\dy\biggr]}\\*
 +\frac{4(x+1)\ee^{-x}-4}{x^2}\biggl[\int_{x}^{\infty}(\xi\ee^{-\xi})'\int_{1}^{\xi}E(y)g(y)\dy\dxi-\int_{x}^{\infty}\del_{\xi}\biggl(\int_{0}^{\xi}\hs{2}(\eta)\deta\biggr)\int_{\xi}^{\infty}g(y)\dy\dxi\biggr].
\end{multline*}
Integration by parts in the last four integrals on the right-hand side and plugging in the formula for $\A[g]$ from~\eqref{eq:def:pre:inverse} yields
\begin{multline*}
 \bigl(\LL\circ \A_{0}[g]\bigr)(x)=g(x)+2\hs{1}(x)\int_{1}^{x}E(y)g(y)\dy-2\hs{2}(x)\int_{x}^{\infty}g(y)\dy\\*
 +\frac{2(x+1)\ee^{-x}}{x^2}\int_{0}^{x}(1-\ee^{z})g(z)\dz+\frac{2(x+1)\ee^{-x}-2}{x^2}\int_{x}^{\infty}g(z)\dz\\*
 +\frac{4(x+1)\ee^{-x}}{x^2}\biggl[\int_{0}^{x}(1-\ee^{\eta})\hs{1}(\eta)\deta\int_{1}^{x}E(y)g(y)\dy-\int_{0}^{x}\int_{0}^{\xi}(1-\ee^{\eta})\hs{1}(\eta)\deta E(\xi)g(\xi)\dxi\\*
 \shoveright{-\int_{0}^{x}(1-\ee^{\eta})\hs{2}(\eta)\deta\int_{x}^{\infty}g(y)\dy-\int_{0}^{x}\int_{0}^{\xi}(1-\ee^{\eta})\hs{2}(\eta)\deta g(\xi)\dxi\biggr]}\\*
 \shoveleft{+\frac{4(x+1)\ee^{-x}-4}{x^2}\biggl[-x\ee^{-x}\int_{1}^{x}E(y)g(y)\dy-\int_{x}^{\infty}\xi\ee^{-\xi}E(\xi)g(\xi)\dxi}\\*
 +\int_{0}^{x}\hs{2}(\eta)\deta\int_{x}^{\infty}g(y)\dy-\int_{x}^{\infty}\int_{0}^{\xi}\hs{2}(\eta)\deta g(\xi)\dxi\biggr].
\end{multline*}
Collecting pre-factors of the same integrals and also exploiting certain cancellations, we further find
\begin{multline}\label{eq:L:surj:1}
 \bigl(\LL\circ \A_{0}[g]\bigr)(x)=g(x)\\*
 +\biggl[2\hs{1}(x)+\frac{4(x+1)\ee^{-x}}{x^2}\int_{0}^{x}(1-\ee^{\eta})\hs{1}(\eta)\deta-\frac{4(x+1)\ee^{-2x}-4\ee^{-x}}{x}\biggr]\int_{1}^{x}E(y)g(y)\dy\\*
 +\biggl[-2\hs{2}(x)+\frac{2(x+1)\ee^{-x}-2}{x^2}+\frac{4(x+1)\ee^{-x}}{x^2}\int_{0}^{x}\ee^{\eta}\hs{2}(\eta)\deta-\frac{4}{x^2}\int_{0}^{x}\hs{2}(\eta)\deta\biggr]\int_{x}^{\infty}g(y)\dy\\*
 +\frac{2(x+1)\ee^{-x}}{x^2}\int_{0}^{x}\biggl(1-\ee^{\xi}-2\int_{0}^{\xi}(1-\ee^{\eta})\hs{1}(\eta)\deta E(\xi)-2\int_{0}^{\xi}(1-\ee^{\eta})\hs{2}(\eta)\deta\biggr)g(\xi)\dxi\\*
 +\frac{4(x+1)\ee^{-x}-4}{x^2}\int_{x}^{\infty}\biggl(-\xi\ee^{-\xi}E(\xi)-\int_{0}^{\xi}\hs{2}(\eta)\deta\biggr)g(\xi)\dxi.
\end{multline}
To simplify the presentation, we consider the terms in brackets separately. Together with~\eqref{eq:prim:m1} we first obtain
\begin{multline}\label{eq:L:surj:2}
 2\hs{1}(x)+\frac{4(x+1)\ee^{-x}}{x^2}\int_{0}^{x}(1-\ee^{\eta})\hs{1}(\eta)\deta-\frac{4(x+1)\ee^{-2x}-4\ee^{-x}}{x}\\*
 =2(1-x)\ee^{-x}+\frac{4(x+1)\ee^{-2x}}{x}-\frac{4(x+1)\ee^{-x}}{x}+2(x+1)\ee^{-x}-\frac{4(x+1)\ee^{-2x}}{x}+\frac{4\ee^{-x}}{x}=0.
\end{multline}
Moreover, \cref{eq:prim:m2:1,eq:prim:m2:2} yield
\begin{multline}\label{eq:L:surj:3}
 -2\hs{2}(x)+\frac{2(x+1)\ee^{-x}-2}{x^2}+\frac{4(x+1)\ee^{-x}}{x^2}\int_{0}^{x}\ee^{\eta}\hs{2}(\eta)\deta-\frac{4}{x^2}\int_{0}^{x}\hs{2}(\eta)\deta\\*
 =-2-2(1-x)\ee^{-x}\int_{1}^{x}\frac{\ee^{z}}{z}\dz+\frac{2(x+1)\ee^{-x}-2}{x^2}+\frac{2(x+1)(2-x)\ee^{-x}}{x}\int_{1}^{x}\frac{\ee^{z}}{z}\dz\\*
 +\frac{2(x^2-1)+2(x+1)\ee^{-x}}{x^2}-\frac{4\ee^{-x}}{x}\int_{1}^{x}\frac{\ee^{z}}{z}\dz=\frac{4(x+1)\ee^{-x}-4}{x^2}.
\end{multline}
Similarly, by means of \cref{eq:def:E,eq:prim:m1,eq:prim:m2:1,eq:prim:m2:2} we get
\begin{multline}\label{eq:L:surj:4}
 1-\ee^{\xi}-2\int_{0}^{\xi}(1-\ee^{\eta})\hs{1}(\eta)\deta E(\xi)-2\int_{0}^{\xi}(1-\ee^{\eta})\hs{2}(\eta)\deta\\*
 \shoveleft{=1-\ee^{\xi}-2\xi\Bigl(\ee^{-\xi}-1+\frac{\xi}{2}\Bigr)\biggl(\frac{\ee^{\xi}}{\xi}-\int_{1}^{\xi}\frac{\ee^{z}}{z}\dz\biggr)-2\xi\ee^{-\xi}\int_{1}^{\xi}\frac{\ee^{z}}{z}\dz}\\*
 +\xi(2-\xi)\int_{1}^{\xi}\frac{\ee^{z}}{z}\dz+(\xi-1)\ee^{\xi}+1=0.
\end{multline}
Finally, \cref{eq:def:E,eq:prim:m2:1} imply
\begin{equation}\label{eq:L:surj:5}
 -\xi\ee^{-\xi}E(\xi)-\int_{0}^{\xi}\hs{2}(\eta)\deta=-1+\xi\ee^{-\xi}\int_{1}^{\xi}\frac{\ee^{z}}{z}\dz-\xi\ee^{-\xi}\int_{1}^{\xi}\frac{\ee^{z}}{z}=-1.
\end{equation}
Thus, summarising \cref{eq:L:surj:1,eq:L:surj:2,eq:L:surj:3,eq:L:surj:4,eq:L:surj:5} we find
\begin{equation*}
 \bigl(\LL\circ \A_{0}[g]\bigr)(x)=g(x)+\frac{4(x+1)\ee^{-x}-4}{x^2}\int_{x}^{\infty}g(y)\dy-\frac{4(x+1)\ee^{-x}-4}{x^2}\int_{x}^{\infty}g(y)\dy=g.
\end{equation*}
This shows that $\LL\circ \A_{0}=\id$ on $C_{\text{c}}^{\infty}(0,\infty)$ and by density, the same relation thus holds true on $\X{a}{b}$ for $a\in(-1,1)$ and $b>1$.

\section*{Acknowledgements}

This work has been funded by the Deutsche Forschungsgemeinschaft (DFG, German Research Foundation) – Projektnummer 396845724. 

The author thanks Jos{\'e} Ca{\~n}izo and Juan Vel{\'a}zquez for helpful conversation on the linearised operator.

\appendix

\section{Elementary properties of the weights and certain auxiliary functions}\label{Sec:properties:weight}

We collect in this section several properties and estimates on the weights $\weight{a}{b}$ as well as on the functions $\hs{1}$ and $\hs{2}$ which appear in the formula for $\LL^{-1}$. 

\subsection{Estimates on the weight $\weight{a}{b}$}

Let us first mention the following estimates on certain primitives of $\weight{a}{b}$:

\begin{align}
 \int_{0}^{x}\weight{a}{b}(y)\ee^{-y}\dy&\lesssim \weight{1+a}{0}(x) \qquad  \text{for } a>-1\text{ and all }b\in\R. \label{eq:prim:weight:1}\\
 \int_{0}^{x}\weight{a}{b}(y)\dy&\lesssim \begin{cases}
                                          \weight{a+1}{0}(x) &\text{if }b<-1\\
                                          \weight{a+1}{b+1}(x) & \text{if }b>-1.
                                         \end{cases} \qquad \text{and }a>-1
\label{eq:prim:weight:2}\\
\int_{x}^{\infty}\weight{a}{b}(y)\ee^{-y}\dy&\lesssim \begin{cases}
                                                       \weight{0}{b}(x)\ee^{-x} &\text{if }a>-1\\
                                                       \weight{a+1}{b}(x)\ee^{-x} & \text{if }a<-1.
                                                      \end{cases}\qquad \text{and all }b\in \R
\label{eq:prim:weight:3}
\end{align}
These bounds follow immediately from the definition of $\weight{a}{b}$ and we thus omit the proofs.

The following lemma gives an estimate on an auxiliary integral which appears during the proof of continuity of $\B_{2}$ and $\B_{W}$ in Section~\ref{Sec:preparation:uniqueness}.

\begin{lemma}\label{Lem:aux:int:weight:1}
 For all $a<1$ and $b\in(1,2)$ we have the estimate
 \begin{equation*}
  \int_{y}^{y+z}\frac{\weight{a}{b}(x)}{x^2}\dx\lesssim \weight{a-1}{0}(y)\weight{0}{b-1}(z)\qquad \text{for all }y,z\in(0,\infty).
 \end{equation*}
\end{lemma}

\begin{proof}
 We have to distinguish several cases: first, if $y\leq 1$ and $y+z\leq 1$ we use $a<1$ to estimate
 \begin{equation}\label{eq:aux:int:weight:1:proof:1}
  \int_{y}^{y+z}\frac{\weight{a}{b}(x)}{x^2}\dx=\int_{y}^{y+z}x^{a-2}\dx\leq \int_{y}^{\infty}x^{a-2}\dx\lesssim y^{a-1}.
 \end{equation}
 Second, if $y\leq 1$ but $y+z\geq 1$ and exploiting also that $x\mapsto x^{b-1}$ is Hölder continuous with exponent $b-1$ since $b\in(1,2)$ we get
 \begin{multline}\label{eq:aux:int:weight:1:proof:2}
  \int_{y}^{y+z}\frac{\weight{a}{b}(x)}{x^2}\dx=\int_{y}^{1}x^{a-2}\dx+\int_{1}^{y+z}x^{b-2}\dx\leq \int_{y}^{\infty}x^{a-2}\dx+\int_{1}^{1+z}x^{b-2}\dx\\*
  \lesssim y^{a-1}+(1+z)^{b-1}-1\lesssim y^{a-1}+z^{b-1}.
 \end{multline}
 Finally, if $y\geq 1$ we get
 \begin{equation}\label{eq:aux:int:weight:1:proof:3}
  \int_{y}^{y+z}\frac{\weight{a}{b}(x)}{x^2}\dx=\int_{y}^{y+z}x^{b-2}\dx\lesssim (y+z)^{b-1}-y^{b-1}\lesssim z^{b-1}.
 \end{equation}
Combining \cref{eq:aux:int:weight:1:proof:1,eq:aux:int:weight:1:proof:2,eq:aux:int:weight:1:proof:3} we can obtain the estimate
\begin{equation*}
 \int_{y}^{y+z}\frac{\weight{a}{b}(x)}{x^2}\dx\lesssim \weight{a-1}{0}(y)+\weight{b-1}{b-1}(z).
\end{equation*}
 Noting that $1=\weight{0}{0}$ and exploiting~\eqref{eq:weight:monotonicity} we may deduce
 \begin{equation*}
 \int_{y}^{y+z}\frac{\weight{a}{b}(x)}{x^2}\dx\lesssim \weight{a-1}{0}(y)\weight{0}{b-1}(z)+\weight{a-1}{0}(y)\weight{0}{b-1}(z)\lesssim\weight{a-1}{0}(y)\weight{0}{b-1}(z).
\end{equation*}
This finishes the proof.
\end{proof}

\begin{remark}
 Note that the estimate in the previous lemma is not optimal, in the sense that we could in fact obtain some regularising behaviour for small $z$. However, this estimate suffices for our purpose.
\end{remark}

Finally, we show the following elementary lemma which gives a certain regularising effect of $1-\ee^{-x}$.

\begin{lemma}\label{Lem:norm:regularising}
 For all $a,b\in\R$ we have the estimate
 \begin{equation*}
  \norm*{\bigl(1-\exp(-\cdot)\bigr)g}_{\X{a}{b}}\leq \norm{g}_{\X{1+a}{b}},
 \end{equation*}
i.e.\@ the multiplication operator given by $g(x)\mapsto (1-\ee^{-x})g(x)$ maps continuously from $\X{1+a}{b}$ to $\X{a}{b}$. Due to~\eqref{eq:spaces:embedding} we have in particular that $\norm*{\bigl(1-\exp(-\cdot)\bigr)g}_{\X{a}{b}}\leq \norm{g}_{\X{a}{b}}$.
\end{lemma}

\begin{proof}
 The proof follows immediately from the definition and~\eqref{eq:weight:regularising}. In fact
 \begin{equation*}
  \norm*{\bigl(1-\exp(-\cdot)\bigr)g}_{\X{a}{b}}=\int_{0}^{\infty}\abs{g(x)}(1-\ee^{-x})\weight{a}{b}(x)\dx\leq \int_{0}^{\infty}\abs{g(x)}\weight{a+1}{b}(x)\dx=\norm{g}_{\X{a+1}{b}}.
 \end{equation*}
\end{proof}

\subsection{Properties of $\hs{1}$ and $\hs{2}$}

We provide in this section several properties and estimates on the functions $\hs{1}$ and $\hs{2}$ which are mainly used to simplify the proofs of \cref{Prop:continuity:L:NEW,Prop:continuity:L:inverse}. We recall from~\eqref{eq:m1:m2} that 
\begin{equation*}
 \hs{1}(x)=(1-x)\ee^{-x}\qquad \text{and}\qquad \hs{2}(x)=1+(1-x)\ee^{-x}\int_{1}^{x}\frac{\ee^{z}}{z}\dz.
\end{equation*}
First, we immediately obtain the estimate
\begin{equation}\label{eq:est:m1}
 \abs*{\hs{1}(x)}\lesssim \begin{cases}
                          1 &\text{if }x\leq 1\\
                          x\ee^{-x} &\text{if }x\geq 1
                         \end{cases}
\qquad \text{or equivalently}\qquad \abs*{\hs{1}(x)}\lesssim \weight{0}{1}(x)\ee^{-x}.
\end{equation}
\begin{remark}\label{Rem:explicit:profile}
 As a direct consequence of~\eqref{eq:est:m1} we see that $\hs{1}$ has exponential decay at infinity and thus all moments of non-negative order are well-defined.  In particular $\hs{1},\exp(-\cdot)\in\X{a}{b}$ for all $a>-1$ and $b\in\R$.
\end{remark}
One also immediately checks by an explicit calculation that
\begin{equation}\label{eq:mass:m1}
 \int_{0}^{\infty}x\hs{1}(x)\dx=-1.
\end{equation}
Conversely, the function $\hs{2}$ has a worse behaviour especially at infinity. However, there is also some cancellation taking place in the stated formula if $x$ is large. Therefore, before stating a bound on $\hs{2}$ let us make some preliminary considerations. In fact, we first look at the auxiliary integral $\int_{1}^{x}\frac{\ee^{z}}{z}\dz$. Using l'H{\^o}pital's rule it is immediate to check that
\begin{equation}\label{eq:est:Ei}
 \int_{1}^{x}\frac{\ee^{z}}{z^n}\dz\lesssim \frac{\ee^{x}}{x^n}\quad \text{for } n\in\N_{0} \text{ if }x\geq 1 \qquad \text{as well as}\qquad \int_{1}^{x}\frac{\ee^{z}}{z}\dz\lesssim \abs*{\log(x)} \quad \text{for }x\leq 1.
\end{equation}
To obtain the precise decay behaviour of $\hs{2}$ we write $\ee^{z}=\del_{z}^{2}(\ee^{z})$ and integrate by parts twice which leads to
\begin{multline*}
  \hs{2}(x)=1+\frac{1-x}{x}+\frac{1-x}{x^2}-2\ee(1-x)\ee^{-x}+2(1-x)\ee^{-x}\int_{1}^{x}\frac{\ee^{z}}{z^3}\dz\\*
 =\frac{1}{x^2}-2\ee(1-x)\ee^{-x}+2(1-x)\ee^{-x}\int_{1}^{x}\frac{\ee^{z}}{z^3}\dz.
\end{multline*}
Thus, together with~\eqref{eq:est:Ei} and the explicit form of $\hs{2}$ one immediately verifies that
\begin{equation}\label{eq:est:m2}
 \abs*{\hs{2}(x)}\lesssim \begin{cases}
                          1+\abs{\log(x)} & \text{if } x\leq 1\\
                          \frac{1}{x^2} &\text{if }x\geq 1.
                         \end{cases}
\qquad \text{as well as} \qquad \hs{2}(x)\sim x^{-2} \text{ as }x\to\infty.                         
\end{equation}

For the proof of Proposition~\ref{Prop:continuity:L:inverse} one also needs an estimate for the expression $\frac{\ee^{x}}{x}-\int_{1}^{x}\frac{\ee^{z}}{z}\dz$. Integrating by parts we obtain similarly as above that
\begin{equation*}
 \frac{\ee^{x}}{x}-\int_{1}^{x}\frac{\ee^{z}}{z}\dz=\ee-\int_{1}^{x}\frac{\ee^{z}}{z^2}\dz.
\end{equation*}
Thus, together with~\eqref{eq:est:Ei} one immediately deduces
\begin{equation}\label{eq:est:aux:expr}
 \abs*{\frac{\ee^{x}}{x}-\int_{1}^{x}\frac{\ee^{z}}{z}\dz}\lesssim \begin{cases}
                                                                    \frac{1}{x} &\text{if }x\leq 1\\
                                                                    \frac{\ee^{x}}{x^2} &\text{if }x\geq 1
                                                                   \end{cases}
\quad \text{or equivalently}\quad \abs*{\frac{\ee^{x}}{x}-\int_{1}^{x}\frac{\ee^{z}}{z}\dz}\lesssim \weight{-1}{-2}(x)\ee^{x}.
\end{equation}
Based on the considerations above, we collect several estimates on (weighted) primitives of $\hs{1}$ and $\hs{2}$. Precisely, for $x\in(0,\infty)$ we have
\begin{align}
 \int_{0}^{x}\abs*{\hs{1}(y)}\weight{a}{b}(y)\dy&\lesssim \weight{a+1}{0}(x) && \text{for all }a> -1\text{ and } b\in\R \label{eq:prim:m1:est:1}\\
 \int_{x}^{\infty}\abs*{\hs{1}(y)}\weight{a}{b}(y)\dy&\lesssim \weight{0}{b+1}(x)\ee^{-x} && \text{for all }a> -1\text{ and } b\in\R. \label{eq:prim:m1:est:2}
\end{align}
Moreover, $\hs{2}$ satisfies the estimate
\begin{equation}\label{eq:prim:m2:est:1}
 \int_{0}^{x}\abs*{\hs{2}(y)}\weight{a}{b}(y)\dy\lesssim \begin{cases}
                                                         x^{a+1}\bigl(1+\abs*{\log(x)}\bigr) &\text{if }x\leq 1\\
                                                         1+x^{b-1} &\text{if }x\geq 1
                                                        \end{cases}
\qquad \text{for all } a>-1 \text{ and } b\neq -1.
\end{equation}
From this, we can deduce in particular for $a>-1$ that
\begin{equation}\label{eq:prim:m2:est:2}
 \int_{0}^{x}\abs*{\hs{2}(y)}\weight{a}{b}(y)\dy\lesssim \weight{a}{0}(x)\quad \text{if } b<1\quad \text{and}\quad \int_{0}^{x}\abs*{\hs{2}(y)}\weight{a}{b}(y)\dy\lesssim \weight{a}{b-1}(x)\quad \text{if } b>1.
\end{equation}
The proofs of \cref{eq:prim:m1:est:1,eq:prim:m1:est:2,eq:prim:m2:est:1,eq:prim:m2:est:2} are elementary and follow essentially directly from the definition of $\weight{a}{b}$ and \cref{eq:est:m1,eq:est:m2} which is why they are omitted here.

To conclude this section, let us finally compute several integrals of $\hs{1}$ and $\hs{2}$ explicitly. In fact, noting that $\hs{1}(\eta)=(1-\eta)\ee^{-\eta}=(\eta\ee^{-\eta})'$ we get
\begin{equation}\label{eq:prim:m1}
 \int_{0}^{x}(1-\eta)\hs{1}(\eta)\deta=\int_{0}^{x}(\eta\ee^{-\eta})'-(1-\eta)\deta=x\ee^{-x}-x+\frac{x^2}{2}=x\Bigl(\ee^{-x}-1+\frac{x}{2}\Bigr).
\end{equation}
With $\hs{2}(\eta)=1+(1-\eta)\ee^{-\eta}\int_{1}^{\eta}\frac{\ee^{z}}{z}\dz=1+(\eta\ee^{-\eta})'\int_{1}^{\eta}\frac{\ee^{z}}{z}\dz$ integration by parts shows
\begin{equation}\label{eq:prim:m2:1}
 \int_{0}^{x}\hs{2}(\eta)\deta=\int_{0}^{x}1+(\eta\ee^{-\eta})'\int_{1}^{\eta}\frac{\ee^{z}}{z}\dz\deta=x+x\ee^{-x}\int_{1}^{x}\frac{\ee^{z}}{z}\dz-\int_{0}^{x}\deta=x\ee^{-x}\int_{1}^{x}\frac{\ee^{z}}{z}\dz.
\end{equation}
Finally, we obtain in a similar way
\begin{multline}\label{eq:prim:m2:2}
 \int_{0}^{x}\ee^{\eta}\hs{2}(\eta)\deta=\int_{0}^{x}\ee^{\eta}+(1-\eta)\int_{1}^{\eta}\frac{\ee^{z}}{z}\dz\deta=\ee^{x}-1+\int_{0}^{x}\Bigl(\frac{\eta}{2}(2-\eta)\Bigr)'\int_{1}^{\eta}\frac{\ee^{z}}{z}\dz\deta\\*
 =\ee^{x}-1+\frac{x(2-x)}{2}\int_{1}^{x}\frac{\ee^{z}}{z}\dz-\frac{1}{2}\int_{0}^{x}(2-\eta)\ee^{\eta}\deta=\frac{x(2-x)}{2}\int_{1}^{x}\frac{\ee^{z}}{z}\dz+\frac{(x-1)\ee^{x}+1}{2}.
\end{multline}

\section{Derivation of $\LL^{-1}$}\label{Sec:inverse:derivation}

 In this section we will illustrate how the formulas \cref{eq:def:pre:inverse,eq:def:inverse} for the inverse $\LL^{-1}$ can be obtained. We particularly emphasise that the approach for this derivation is completely formal and the entire proof of Proposition~\ref{Prop:continuity:L:inverse} is contained in Section~\ref{Sec:proof:inversion}. In fact, this section only serves for illustrative purposes.
 
 To determine the inverse $\LL^{-1}$ we have to solve $\LL[\ode]=\rhs$ for given $\rhs$. However, classical methods from integral equations (e.g.\@ \cite{Waz15}) allow to convert this problem into a linear ODE of second order for $\ode$. For the latter, well-known methods then provide a solution formula which yields the desired expression in~\eqref{eq:def:inverse} for $\LL^{-1}$. The same idea has been pursued in~\cite{Sch16a} while a similar approach to transform coagulation equations into (systems of) ODEs can for example also be found in~\cite{MNV11}.
  
 To derive the ODE, let us consider the equation
 \begin{equation}\label{eq:linearised:1}
  \rhs=\LL[\ode]=\ode-\frac{2}{x^2}\int_{0}^{x}\int_{x-y}^{\infty}y\ee^{-y}\ode(z)+y\ode(y)\ee^{-z}\dz\dy.
 \end{equation}
 Multiplying by $x^2$ and differentiating on both sides, we find
 \begin{multline*}
  2x\rhs(x)+x^2\rhs'(x)=2x\ode(x)+x^2\ode'(x)-2x\ee^{-x}\int_{0}^{\infty}\ode(z)\dz-2x\ode(x)\int_{0}^{\infty}\ee^{-z}\dz\\*
  +2\int_{0}^{x}y\ee^{-y}\ode(x-y)\dy+2\int_{0}^{x}y\ode(y)\ee^{-(x-y)}\dy.
 \end{multline*}
 Changing variables $y\mapsto (x-y)$ in the third integral on the right-hand side and using that $\int_{0}^{\infty}\ee^{-z}\dz=1$ this equation simplifies as
 \begin{equation*}
  2x\rhs(x)+x^2\rhs'(x)=x^2\ode'(x)-2x\ee^{-x}\int_{0}^{\infty}\ode(z)\dz+2x\ee^{-x}\int_{0}^{x}\ee^{y}\ode(y)\dy.
 \end{equation*}
 Now, multiplying by $\ee^{x}/x$ and differentiating again yields
 \begin{equation*}
  2\ee^{x}\bigl(\rhs(x)+\rhs'(x)\bigr)+\ee^{x}\bigl((1+x)\rhs'(x)+x\rhs''(x)\bigr)=\ee^{x}\bigl((1+x)\ode'(x)+x\ode''(x)\bigr)+2\ee^{x}\ode(x).
 \end{equation*}
 Thus, dividing by $x\ee^{x}$ we obtain the ODE
 \begin{equation}\label{eq:ODE:1}
  \ode''(x)+\frac{1+x}{x}\ode'(x)+\frac{2}{x}\ode(x)=\rhs''(x)+\frac{3+x}{x}\rhs'(x)+\frac{2}{x}\rhs(x).
 \end{equation}
 To derive a solution to this equation we follow the standard method of variation of parameters (see e.g.\@ \cite{BeO99}). To this end, we first have to obtain the general solution to
 \begin{equation}\label{eq:ODE:hom}
  \ode''(x)+\frac{1+x}{x}\ode'(x)+\frac{2}{x}\ode(x)=0.
 \end{equation}
 One immediately checks, that one solution is given by $\hs{1}(x)=(1-x)\ee^{-x}$, while a second linearly independent solution can be derived by \emph{reduction of order} (e.g.\@ \cite{BeO99}). In fact, after some elementary computation one finds that a linearly independent solution is given by $\hs{2}(x)=1+(1-x)\ee^{-x}\int_{1}^{x}\frac{\ee^{z}}{z}\dz$. Note that the lower integral bound can be chosen arbitrarily while we take $1$ here for convenience. 
 
 \begin{remark}\label{Rem:m2:not:sol}
 For completeness, let us mention that even though $\hs{2}$ solves~\eqref{eq:ODE:hom} we have $\LL[\hs{2}]\not\equiv 0$, i.e.\@ $\hs{2}\not\in \ker \LL$ independently of the choice of the space. The reason for this is simply that the differentiation used to transform the integral equation into an ODE creates this additional homogeneous solution.
\end{remark}
 
 To derive now a formula for the solution of~\eqref{eq:ODE:1} it turns out to be convenient to define $\rode\vcc=\ode-\rhs$ such that~\eqref{eq:ODE:1} can be rewritten in terms of $\rode$ as
 \begin{equation}\label{eq:ODE:2}
  \rode''(x)+\frac{1+x}{x}\rode'(x)+\frac{2}{x}\rode(x)=\frac{2}{x}\rhs'(x).
 \end{equation}
 It is well-known that the set of all solutions of~\eqref{eq:ODE:2} is given by $\{\rode^{p}+C_{1}\hs{1}+C_{2}\hs{2}\;|\; C_1, C_2\in\R\}$, where $\rode^{p}$ can be chosen to be any \emph{particular solution} to~\eqref{eq:ODE:2}. Following~\cite{BeO99} such a $\rode^{p}$ is given by the formula
 \begin{equation}\label{eq:part:sol:gen}
  \rode^{p}(x)=-\hs{1}(x)\int^{x}\frac{\frac{2\rhs'(y)}{y}\hs{2}(y)}{W(y)}\dy+\hs{2}(x)\int^{x}\frac{\frac{2\rhs'(y)}{y}\hs{1}(y)}{W(y)}\dy
 \end{equation}
 where the notation $\int^{x}f(y)\dy$ indicates a primitive of $f$ and $W$ is the Wronskian corresponding to $\hs{1}$ and $\hs{2}$, i.e.\@ $W(x)=\hs{1}(x)\hs{2}'(x)-\hs{1}'(x)\hs{2}$. This quantity can be easily computed. In fact, we have
 \begin{multline*}
  W(x)=(1-x)\ee^{-x}\biggl(-\ee^{-x}\int_{1}^{x}\frac{\ee^{z}}{z}\dz-(1-x)\ee^{-x}\int_{1}^{x}\frac{\ee^{z}}{z}\dz+\frac{1-x}{x}\biggr)\\*
  -\bigl(-\ee^{-x}-(1-x)\ee^{-x}\bigr)\biggl(1+(1-x)\ee^{-x}\int_{1}^{x}\frac{\ee^{z}}{z}\dz\biggr).
 \end{multline*}
Summarising the expressions on the right-hand side we get
\begin{equation*}
 W(x)=\frac{(1-x)^{2}\ee^{-x}}{x}+\ee^{-x}+(1-x)\ee^{-x}=\frac{\ee^{-x}}{x}.
\end{equation*}
Thus, from~\eqref{eq:part:sol:gen} we obtain the following expression for a solution $\rode^{p}$ of~\eqref{eq:ODE:2}
\begin{equation*}
 \rode^{p}(x)=-2\hs{1}(x)\int_{a}^{x}\ee^{y}\hs{2}(y)\rhs'(y)\dy+2\hs{2}\int_{b}^{x}\ee^{y}\hs{1}(y)\rhs'(y)\dy.
\end{equation*}
We are still free to choose appropriate constants $a$ and $b$. Moreover, in order to get a formula which only depends on $g$ rather than on $g'$ we have to integrate by parts which leads to
\begin{multline*}
 \rode^{p}(x)=-2\hs{1}(x)\hs{2}(x)\ee^{x}\rhs(x)+2\hs{1}(x)\hs{2}(a)\ee^{a}\rhs(a)+2\hs{1}(x)\int_{a}^{x}\ee^{y}\bigl(\hs{2}(y)+\hs{2}'(y)\bigr)\rhs(y)\dy\\*
 +2\hs{1}(x)\hs{2}(x)\ee^{x}\rhs(x)-2\hs{2}(x)\hs{1}(b)\ee^{b}\rhs(b)-2\hs{2}(x)\int_{b}^{x}\ee^{y}\bigl(\hs{1}(y)+\hs{1}'(y)\bigr)\rhs(y)\dy.
\end{multline*}
Since we want to invert $\LL$ for $L^1$ functions, the final formula should not contain expressions such as $\rhs(a)$ or $\rhs(b)$. However, it is shown in Section~\ref{Sec:lin:inj} that $\hs{1}\in\ker \LL$. Thus, we can subtract the term $2\hs{1}(x)\hs{2}(a)\ee^{a}\rhs(a)$ and still get a particular solution. On the other hand, we choose $b$ such that the boundary term vanishes (at least formally). In fact, we take $b=\infty$ which is motivated by the assumption that $\rhs$ is in $L^1$ such that one expects that $\rhs(\infty)=0$ in a suitable sense. We note again that this derivation is completely formal and the correctness of the result has to be verified a posteriori (see Section~\ref{Sec:proof:inversion}). Thus, computing also
\begin{equation*}
 \hs{1}(x)+\hs{1}'(x)=-\ee^{-x}\qquad \text{and}\qquad \hs{2}(x)+\hs{2}'(x)=\frac{1}{x}-\ee^{-x}\int_{1}^{x}\frac{\ee^{z}}{z}\dz
\end{equation*}
we obtain the final formula for $\rode^{p}$ as
\begin{equation*}
 \rode^{p}(x)=2\hs{1}(x)\int_{1}^{x}\biggl(\frac{\ee^{y}}{y}-\int_{1}^{y}\frac{\ee^{z}}{z}\dz\biggr)\rhs(y)\dy-2\hs{2}(x)\int_{x}^{\infty}\rhs(y)\dy.
\end{equation*}
Thus, a solution $\ode^{p}$ to our original equation~\eqref{eq:ODE:1} is given by $\ode^{p}=\rode^{p}+\rhs$, i.e.\@
\begin{equation}\label{eq:particular:ODE:1}
 \ode^{p}(x)=\rhs(x)+2\hs{1}(x)\int_{1}^{x}\biggl(\frac{\ee^{y}}{y}-\int_{1}^{y}\frac{\ee^{z}}{z}\dz\biggr)\rhs(y)\dy-2\hs{2}(x)\int_{x}^{\infty}\rhs(y)\dy.
\end{equation}

However, this is not yet the full expression for $\LL^{-1}$ since we have in principle the freedom to add any multiple of $\hs{1}$ and $\hs{2}$. In fact, since~\eqref{eq:Smol} is mass conserving we cannot even expect to be able to invert $\LL$ on $\X{a}{b}$ for appropriate $a$ and $b$. Instead, it appears natural to invert $\LL$ on the subset $\X[0]{a}{b}$ with zero total mass. Precisely, as shown in Section~\ref{Sec:lin:inj} the latter space satisfies $\X[0]{a}{b}\cap \ker \LL=\{0\}$.

The idea thus is to add suitable multiples of $\hs{1}$ and $\hs{2}$ to $\ode^{p}$ to obtain a solution to~\eqref{eq:ODE:1} whose first moment is zero. However, \eqref{eq:est:m2} shows that $\int_{0}^{\infty}x\hs{2}(x)\dx$ does not exist. Therefore, since we already implicitly constructed $\rode^{p}$ and $\ode^{p}$ to have a finite first moment (by the choice of $b=\infty$), it only remains to use $\hs{1}$ to adjust the total mass of $\ode^{p}$. In fact, recalling~\eqref{eq:mass:m1} the formula 
\begin{multline*}
  \ode^{p}(x)+\int_{0}^{\infty}y\ode^{p}(y)\dy \hs{1}(x)\\*
  =\rhs(x)+2\hs{1}(x)\int_{1}^{x}\biggl(\frac{\ee^{y}}{y}-\int_{1}^{y}\frac{\ee^{z}}{z}\dz\biggr)\rhs(y)\dy-2\hs{2}(x)\int_{x}^{\infty}\rhs(y)\dy+\int_{0}^{\infty}y\ode^{p}(y)\dy \hs{1}(x)
\end{multline*}
provides a solution to~\eqref{eq:ODE:1} whose first moment equals zero. For given $\rhs$, the right-hand side is exactly the expression $\A_{0}[\rhs]$ as defined in~\eqref{eq:def:inverse} while $\A[\rhs]$ corresponds to $\ode^{p}$.

\section{Regularity of self-similar profiles}\label{Sec:regularity:profiles}
 
 We provide here continuity and differentiability of self-similar profiles. The corresponding proofs follow standard methods which have been applied in similar form already before (e.g.\@ in~\cite{EsM06,FoL06}). However, since those previous results do not exactly cover the profiles under the assumptions considered in this work, we give the proofs for completeness, since we rely on differentiability especially for the proof of Proposition~\ref{Prop:bound:layer:est}.
 
 \begin{lemma}\label{Lem:continuity}
  Let $K_{\eps}$ satisfy \cref{eq:Ass:K1,eq:Ass:K2}. Then, for sufficiently small $\eps>0$, each self-similar profile $\pr$ is continuous on $(0,\infty)$.
 \end{lemma}

 \begin{proof}
  The profile $\pr$ satisfies $x^{2}\pr(x)=\int_{0}^{x}\int_{x-y}^{\infty}yK_{\eps}(y,z)\pr(y)\pr(z)\dz\dy$. Thus, it suffices to prove continuity of the right-hand side. To see this, let $0<x_1<x_2$ be give. Then, by splitting the integral, and exploiting the non-negativity of the integrand, we get
  \begin{multline*}
   \abs*{\int_{0}^{x_2}\int_{x_2-y}^{\infty}yK_{\eps}(y,z)\pr(y)\pr(z)\dz\dy-\int_{0}^{x_1}\int_{x_1-y}^{\infty}yK_{\eps}(y,z)\pr(y)\pr(z)\dz\dy}\\*
   \leq \int_{x_1}^{x_2}\int_{x_2-y}^{\infty}yK_{\eps}(y,z)\pr(y)\pr(z)\dz\dy+\int_{0}^{x_1}\int_{x_1-y}^{x_{2}-y}yK_{\eps}(y,z)\pr(y)\pr(z)\dz\dy.
  \end{multline*}
 To estimate the right-hand side further, we recall from~\eqref{eq:pert:est:weight} that $K_{\eps}(y,z)\lesssim \weight{-\alpha}{\alpha}(y)\weight{-\alpha}{\alpha}(z)$ and extend the first integral in $z$ which yields together with \cref{eq:weight:shift,Lem:moments} that
   \begin{multline}\label{eq:proof:cont}
   \abs*{\int_{0}^{x_2}\int_{x_2-y}^{\infty}yK_{\eps}(y,z)\pr(y)\pr(z)\dz\dy-\int_{0}^{x_1}\int_{x_1-y}^{\infty}yK_{\eps}(y,z)\pr(y)\pr(z)\dz\dy}\\*
   \shoveleft{\leq \int_{x_1}^{x_2}\weight{1-\alpha}{1+\alpha}(y)\pr(y)\dy\int_{0}^{\infty}\weight{-\alpha}{\alpha}(z)\pr(z)\dz}\\*
   +\int_{0}^{x_1}\int_{x_1-y}^{x_{2}-y}\weight{1-\alpha}{1+\alpha}(y)\weight{-\alpha}{\alpha}(z)\pr(y)\pr(z)\dz\dy\\*
   \lesssim \int_{x_1}^{x_2}\weight{1-\alpha}{1+\alpha}(y)\pr(y)\dy+\int_{0}^{x_1}\int_{x_1-y}^{x_{2}-y}\weight{1-\alpha}{1+\alpha}(y)\weight{-\alpha}{\alpha}(z)\pr(y)\pr(z)\dz\dy.
  \end{multline}
  Lemma~\ref{Lem:moments} ensures that $y\mapsto \weight{1-\alpha}{1+\alpha}(y)\pr(y)$ and $(y,z)\mapsto \weight{1-\alpha}{1+\alpha}(y)\weight{-\alpha}{\alpha}(z)\pr(y)\pr(z)$ are integrable functions on $(0,\infty)$ and $(0,\infty)^2$ respectively. As a consequence, the right-hand side of~\eqref{eq:proof:cont} converges to zero if $\abs{x_2-x_1}\to 0$ (e.g.\@ \cite[eq.\@ (A3-10)]{Alt16}). This then finishes the proof.
 \end{proof}

 \begin{proposition}\label{Prop:differentiability}
  Let $K_{\eps}$ satisfy \cref{eq:Ass:K1,eq:Ass:K2}. Then, for sufficiently small $\eps>0$, each self-similar profile $\pr$ is differentiable on $(0,\infty)$.
 \end{proposition}

 \begin{proof}
  We claim that the distributional derivative of $x\mapsto x^2\pr(x)$ is given by
  \begin{equation}\label{eq:proof:dif:1}
   x\mapsto x\pr(x)\int_{0}^{\infty}K_{\eps}(x,z)\pr(z)\dy-\int_{0}^{x}yK_{\eps}(y,x-y)\pr(y)\pr(x-y)\dy.
  \end{equation}
 In fact, testing~\eqref{eq:selfsim} by $-\varphi'(x)$ we get by means of Fubini's theorem and the change of variables $z\mapsto z-y$ that
 \begin{multline*}
  -\int_{0}^{\infty} x^2\pr(x)\varphi'(x)\dx=\int_{0}^{\infty}\int_{0}^{\infty}yK_{\eps}(y,z)\pr(y)\pr(z)\bigl(\varphi(y)-\varphi(y+z)\bigr)\dy\dz\\*
  \shoveleft{=\int_{0}^{\infty}\varphi(y)\biggl(y\pr(y)\int_{0}^{\infty}K_{\eps}(y,z)\pr(z)\dz\biggr)\dy}\\*
  \shoveright{-\int_{0}^{\infty}\int_{y}^{\infty}yK_{\eps}(y,z-y)\pr(y)\pr(z-y)\varphi(z)\dz\dy}\\*
  \shoveleft{=\int_{0}^{\infty}\varphi(y)\biggl(y\pr(y)\int_{0}^{\infty}K_{\eps}(y,z)\pr(z)\dz\biggr)\dy}\\*
  -\int_{0}^{\infty}\varphi(z)\int_{0}^{z}yK_{\eps}(y,z-y)\pr(y)\pr(z-y)\dy\dz.
 \end{multline*}
 Relabelling the integration variables, we thus obtain
 \begin{multline*}
  -\int_{0}^{\infty} x^2\pr(x)\varphi'(x)\dx\\*
  =\int_{0}^{\infty}\varphi(x)\biggl(x\pr(x)\int_{0}^{\infty}K_{\eps}(y,z)\pr(z)\dz-\int_{0}^{x}yK_{\eps}(y,x-y)\pr(y)\pr(x-y)\dy\biggr)\dx.
 \end{multline*}
 To see that $\pr$ is differentiable on $(0,\infty)$ it is sufficient to prove that the map~\eqref{eq:proof:dif:1} is continuous on $(0,\infty)$. For the first integral this immediately follows since~\eqref{eq:pert:est:weight} yields
 \begin{equation*}
  K_{\eps}(x,z)\pr(z)\lesssim \weight{-\alpha}{\alpha}(x)\weight{-\alpha}{\alpha}(z)\pr(z)
 \end{equation*}
 and $z\mapsto \weight{-\alpha}{\alpha}(z)\pr(z)$ is integrable due to Lemma~\ref{Lem:moments}. Thus, Lebesgue's theorem (on parameter dependent integrals) together with the assumed continuity of $K_{\eps}$ provides that $x\mapsto \int_{0}^{\infty}K_{\eps}(x,z)\pr(z)\dy$ is continuous. By means of Lemma~\ref{Lem:continuity} also the map formed by the product $x\mapsto x\pr(x)\int_{0}^{\infty}K_{\eps}(x,z)\pr(z)\dy$ is continuous on $(0,\infty)$.
 
 Therefore, it only remains to show the continuity of the second term in~\eqref{eq:proof:dif:1}. For this, let $r,R>0$ with $r<R$ but otherwise arbitrary. Moreover, we fix $\delta\in(0,r/2)$. Then, for $x_1,x_2\in[r,R]$ with $x_1<x_2$ we have
 \begin{multline}\label{eq:proof:dif:2}
  \abs*{\int_{0}^{x_2}yK_{\eps}(y,x_2-y)\pr(y)\pr(x_2-y)\dy-\int_{0}^{x_1}yK_{\eps}(y,x_1-y)\pr(y)\pr(x_1-y)\dy}\\*
  \shoveleft{\leq \int_{x_1}^{x_2}yK_{\eps}(y,x_2-y)\pr(y)\pr(x_2-y)\dy}\\*
  +\int_{0}^{x_1}y\pr(y)\abs*{K_{\eps}(y,x_2-y)\pr(x_2-y)-K_{\eps}(y,x_1-y)\pr(x_1-y)}\dy.
 \end{multline}
We now estimate the two terms separately. For the first one, we change variables $y\mapsto x_{2}-y$ and exploit $K_{\eps}(x,y)\lesssim \weight{-\alpha}{\alpha}(x)\weight{-\alpha}{\alpha}(y)$ from~\eqref{eq:pert:est:weight} together with~\eqref{eq:weight:shift} to get
\begin{multline*}
 \int_{x_1}^{x_2}yK_{\eps}(y,x_2-y)\pr(y)\pr(x_2-y)\dy=\int_{0}^{x_2-x_1}(x_2-y)K_{\eps}(x_2-y,y)\pr(x_2-y)\pr(y)\dy\\*
 \lesssim \int_{0}^{x_2-x_1}\weight{1-\alpha}{1+\alpha}(x_2-y)\pr(x_2-y)\weight{-\alpha}{\alpha}(y)\pr(y)\dy\\*
 \lesssim \sup_{x\in[r,R]}\Bigl(\weight{1-\alpha}{1+\alpha}(x)\pr(x)\Bigr)\int_{0}^{x_2-x_1}\weight{-\alpha}{\alpha}(y)\pr(y)\dy\longrightarrow 0\qquad \text{as }\abs{x_2-x_1}\to 0
\end{multline*}
 due to \cite[eq.\@ (A3-10)]{Alt16} and the integrability of $y\mapsto \weight{-\alpha}{\alpha}(y)\pr(y)$ which is guaranteed by Lemma~\ref{Lem:moments}. To estimate the second term on the right-hand side of~\eqref{eq:proof:dif:2} we split the integral further to find
 \begin{multline}\label{eq:proof:dif:3}
  \int_{0}^{x_1}y\pr(y)\bigl(K_{\eps}(y,x_2-y)\pr(x_2-y)-K_{\eps}(y,x_1-y)\pr(x_1-y)\bigr)\dy\\*
  =\int_{0}^{\delta}(\cdots)\dy+\int_{x_{1}-\delta}^{x_{1}}(\cdots)\dy+\int_{\delta}^{x_{1}-\delta}(\cdots)\dy.
 \end{multline}
 For the first integral we obtain by means of $K_{\eps}(x,y)\lesssim \weight{-\alpha}{\alpha}(x)\weight{-\alpha}{\alpha}(y)$ and~\eqref{eq:weight:shift}
 \begin{multline*}
  \int_{0}^{\delta}y\pr(y)\abs*{K_{\eps}(y,x_2-y)\pr(x_2-y)-K_{\eps}(y,x_1-y)\pr(x_1-y)}\dy\\*
  \lesssim \int_{0}^{\delta}\weight{1-\alpha}{1+\alpha}(y)\pr(y)\bigl(\weight{-\alpha}{\alpha}(x_2-y)\pr(x_2-y)+\weight{-\alpha}{\alpha}(x_1-y)\pr(x_1-y)\bigr)\dy\\*
  \lesssim \sup_{x\in[r/2,R]}\Bigl(\weight{-\alpha}{\alpha}(x)\pr(x)\Bigr)\int_{0}^{\delta}\weight{1-\alpha}{1+\alpha}(y)\pr(y)\dy\longrightarrow 0\qquad \text{as }\delta\to 0
 \end{multline*}
 due to \cite[eq.\@ (A3-10)]{Alt16} and Lemma~\ref{Lem:moments}. The second term on the right-hand side of~\eqref{eq:proof:dif:3} can be treated similarly, while we also have to change variables $y\mapsto x_2-y$ and $y\mapsto x_1-y$ respectively, to get
 \begin{multline*}
  \int_{x_1-\delta}^{x_{1}}y\pr(y)\abs*{K_{\eps}(y,x_2-y)\pr(x_2-y)-K_{\eps}(y,x_1-y)\pr(x_1-y)}\dy\\*
  \leq \int_{0}^{\delta}\weight{1-\alpha}{1+\alpha}(x_1-y)\pr(x_1-y)\weight{-\alpha}{\alpha}(y)\pr(y)\dy\\*
  +\int_{x_2-x_1}^{x_2-x_1+\delta}\weight{1-\alpha}{1+\alpha}(x_2-y)\pr(x_1-y)\weight{-\alpha}{\alpha}(y)\pr(y)\dy\\*
  \lesssim \sup_{x\in[r/2,R]}\Bigl(\weight{1-\alpha}{1+\alpha}(x)\pr(x)\Bigr)\biggl(\int_{0}^{\delta}\weight{-\alpha}{\alpha}(y)\pr(y)\dy+\int_{x_2-x_1}^{x_2-x_1+\delta}\weight{-\alpha}{\alpha}(y)\pr(y)\dy\biggr).
 \end{multline*}
Again, the right-hand side tends to zero if $\delta\to 0$ thanks to \cite[eq.\@ (A3-10)]{Alt16} and Lemma~\ref{Lem:moments}. Thus, it only remains to estimate the third term on the right-hand side of~\eqref{eq:proof:dif:3}. For this, we note that due to Lemma~\ref{Lem:continuity} and the continuity of $K_{\eps}$ the map $(x,y)\mapsto K_{\eps}(y,x)$ is uniformly continuous on the compact set $[\delta,R]\times[\delta,R]$. Thus, we have for $x_1,x_2\in[r,R]$ that
\begin{equation*}
 \abs*{K_{\eps}(y,x_2-y)\pr(x_2-y)-K_{\eps}(y,x_1-y)\pr(x_1-y)}\leq \psi(\abs{x_2-x_1}) 
\end{equation*}
with $\psi(z)\to 0$ as $z\to 0$. This together with Lemma~\ref{Lem:moments} immediately yields
 \begin{multline*}
  \int_{\delta}^{x_{1}-\delta}y\pr(y)\abs*{K_{\eps}(y,x_2-y)\pr(x_2-y)-K_{\eps}(y,x_1-y)\pr(x_1-y)}\dy\\*
  \leq \int_{0}^{\infty}y\pr(y)\dy \psi(\abs{x_2-x_1}) \lesssim \psi(\abs{x_2-x_1}) \to 0
 \end{multline*}
as $\abs{x_2-x_1}\to 0$ which finishes the proof.
\end{proof}

\bibliographystyle{abbrv}

\end{document}